\documentclass[11pt]{amsart}
\usepackage{amsmath,amsfonts,amssymb,amsthm,color,hyperref,tikz,todonotes}
\usepackage[all]{xy}
\usepackage{mathrsfs}

\newtheorem{Theorem}{Theorem}[section]
\newtheorem{Proposition}[Theorem]{Proposition}
\newtheorem{Lemma}[Theorem]{Lemma}
\newtheorem{Question}[Theorem]{Question}
\newtheorem{Corollary}[Theorem]{Corollary}
\newtheorem{Conjecture}[Theorem]{Conjecture}

\newtheorem{Definition}[Theorem]{Definition}

\newtheorem{Remark}[Theorem]{Remark}

\newcommand{\nc}{\newcommand}
\nc{\h}{\mathfrak h}
\nc{\g}{\mathfrak g}
\nc{\fM}{\mathfrak M}
\nc{\bM}{\mathbf M}
\nc{\bR}{\mathbf R}
\nc{\bm}{\mathbf m}
\nc{\bS}{\mathbf S}
\nc{\bT}{\mathbf T}
\nc{\bU}{\mathbf U}
\nc{\bV}{\mathbf V}
\nc{\bi}{\mathbf i}
\nc{\bs}{\mathbf s}
\nc{\C}{\mathbb C}
\nc{\Z}{\mathbb Z}
\nc{\N}{\mathbb N}
\nc{\B}{\mathcal B}
\nc{\con}{\sim}
\nc{\pgl}{\mathfrak{pgl}}
\nc{\ev}{\mathsf{ev}}
\nc{\Hom}{\operatorname{Hom}}
\nc{\one}{\mathbf{1}}
\nc{\bb}{\mathbf{b}}
\nc{\ext}{\operatorname{Ext}}
\nc{\out}{\operatorname{out}}
\nc{\inn}{\operatorname{in}}
\nc{\diam}{\diamond}
\nc{\la}{\lambda}
\nc{\Yml}{Y_\mu^\lambda}
\nc{\bgam}{{\boldsymbol{\gamma}}}
\nc{\blam}{{\boldsymbol{\lambda}}}
\nc{\gr}{\operatorname{gr}}
\nc{\Spec}{\operatorname{Spec}}
\nc{\MaxSpec}{\operatorname{MaxSpec}}
\nc{\Cartan}{\mathscr{H}}

\newcommand{\arxiv}[1]{\href{http://arxiv.org/abs/#1}{\tt arXiv:\nolinkurl{#1}}}
\renewcommand{\O}{\mathcal O}
\nc{\Gr}{\mathsf{Gr}}
\nc{\Grlmbar}{\Gr^{\overline{\lambda}}_\mu}
\nc{\excise}[1]{}

\title[Highest weights and product monomial crystals]{Highest weights for truncated shifted Yangians and product monomial crystals}

\date{\today}

\begin{document}

\begin{abstract}
Truncated shifted Yangians are a family of algebras which are natural quantizations of slices in the affine Grassmannian.  We study the highest weight representations of these algebras.  In particular, we conjecture that the possible highest weights for these algebras are described by product monomial crystals, certain natural subcrystals of Nakajima's monomials.  We prove this conjecture in type A.  We also  place our results in the context of symplectic duality and prove a conjecture of Hikita in this situation.
\end{abstract}

\author[Kamnitzer]{Joel Kamnitzer}
\address{J.~Kamnitzer: Department of Mathematics, University of Toronto, Canada}
\email{jkamnitz@math.toronto.edu}

\author[Tingley]{Peter Tingley}
\address{P.~Tingley: Department of Mathematics and Statistics, Loyola University, Chicago, United States}
\email{ptingley@luc.edu}

\author[Webster]{Ben Webster}
\address{B.~Webster: Dept. of Pure Mathematics, University of Waterloo \&
Perimeter Institute for Theoretical Physics, Canada}
\email{ben.webster@uwaterloo.ca}

\author[Weekes]{Alex Weekes}
\address{A.~Weekes: Perimeter Institute for Theoretical Physics, Canada}
\email{aweekes@perimeterinstitute.ca}

\author[Yacobi]{Oded Yacobi}
\address{O.~Yacobi: School of Mathematics and Statistics, University of Sydney, Australia}
\email{oded.yacobi@sydney.edu.au}
\maketitle

\tableofcontents
\section{Introduction}

\subsection{Affine Grassmannian slices and their quantizations}

Let $ \g $ be a semisimple simply-laced Lie algebra.  In recent years, there has been increasing interest in affine Grassmanian slices (see for example \cite{BFI}).  These varieties $ \Grlmbar $ are indexed by a pair of dominant coweights of $ \g $, and are defined as transversal slices between two spherical Schubert varieties inside the affine Grassmannian of $ G $ (the group of adjoint type integrating $ \g $).  The
slice $ \Grlmbar $ is an important object of study in geometric representation
theory because under the geometric Satake correspondence it is related
to the $\mu$ weight space in the irreducible representation of $
G^\vee $ of highest weight $ \lambda $.

In our previous paper \cite{KWWY}, we defined a Poisson structure on these slices $ \Grlmbar $.  Associated to the same parameters $ \lambda, \mu $, we defined a subquotient, $ Y^\lambda_\mu $, of the Yangian of $ \g $, which we call a truncated shifted Yangian (see \ref{subsection: The truncated shifted Yangians}).  More precisely, we defined a family of  algebras $ Y^\lambda_\mu(\bR) $ depending on a set of parameters $ \bR $.  Our construction is motivated by Brundan and Kleshchev's work on shifted Yangians \cite{BK}; in type A when $\lambda$ is a multiple of the first fundamental weight, $ Y^\lambda_\mu(\bR) $ is isomorphic to a central quotient of a finite W-algebra \cite{WWY}.

In \cite{KWWY}, we conjectured that this family $ Y^\lambda_\mu(\bR) $ was the universal family of quantizations of $ \Grlmbar $, in the sense of Bezrukavnikov-Kaledin \cite{BeKa}.  For each $ \bR $, we constructed a map from the associated graded of $  Y^\lambda_\mu(\bR) $ to $ \C[\Grlmbar] $ which is an isomorphism modulo nilpotents.

\subsection{Representation theory of truncated shifted Yangians}
The varieties $ \Grlmbar $ are examples of conical symplectic singularities.  Moreover they carry a natural Hamiltonian torus action (from the maximal torus $T$ of $ G $).  This torus action carries over to the quantizations $ Y^\lambda_\mu(\bR) $ as well; if we choose a generic cocharacter $\C^*\to T$, this gives us a natural notion of the positive half $ Y^\lambda_\mu(\bR)^\ge $ of this algebra.

In \cite{BLPWgco}, the third author and his collaborators initiated a program to study the ``category $\O$'' of quantizations of conical symplectic resolutions equipped with Hamiltonian torus actions.  Though the affine Grassmannian slices only sometimes admit symplectic resolutions (see Theorem 2.9 of \cite{KWWY} for a precise statement), we are guided by the \cite{BLPWgco} perspective.

In this paper, we study representations of $Y^\lambda_\mu(\bR) $ for integral values of the parameters $ \bR $.  Our main interest are the Verma modules $ M^\lambda_\mu(J,\bR) $ for $ Y^\lambda_\mu(\bR) $ which are obtained by inducing from a 1-dimensional representation of $ Y^\lambda_\mu(\bR)^\ge $ associated to a collection of power series $ J $.  These Verma modules can also be defined as quotients of Verma modules for the shifted Yangians $Y_\mu $ (see Section \ref{Verma modules for truncated shifted Yangians}).  The Verma modules for $Y_\mu $ are quite simple and many aspects of their structure is independent of $ J$.  On the other hand, the Verma modules $ M^\lambda_\mu(J,\bR) $ depend heavily on $ J $;  in fact they are non-zero for only finitely many $ J$.

We write $ H_\mu^\lambda(\bR) $ for the set of $ J$ for which $M^\lambda_\mu(J, \bR) $ is non-zero. The basic question that we will attempt to answer is the following.
\begin{Question}
How can we describe the set $H_\mu^\lambda(\bR) $?
\end{Question}

\subsection{Product monomial crystals}
In order to combinatorially describe the set $ H_\mu^\lambda(\bR) $, we will use Nakajima's monomial crystal for the Langlands dual Lie algebra $ \g^\vee $.  Given $ \bR $, we define a product monomial crystal $ \B(\lambda, \bR)$, a certain subcrystal of Nakajima's monomials. It is the set of all products of monomials lying in fundamental monomial crystals indexed by the parameters $ \bR $ (see Section \ref{Tsection}).  This crystal depends in a subtle way on the set of parameters $ \bR$.  Generically, $ \B(\lambda, \bR)$ is the crystal of the tensor product of fundamental representations whose highest weights add up to $ \lambda $, but at special values, it can be smaller.  At the most degenerate values, it will simply be the irreducible crystal with highest weight $\lambda$.

For any $ J \in H_\mu^\lambda(\bR) $, we construct a Nakajima monomial $ y(J)$ by writing $ J $ as a collection of rational functions (see Section \ref{subsec:highest}).
\begin{Conjecture} \label{co:intro}
 For any dominant coweights $ \lambda, \mu $, the map $ J \mapsto y(J) $ is a bijection between $ H^\lambda_\mu(\bR) $ and $ \B(\lambda, \bR)_\mu $.
\end{Conjecture}

Our main tool for studying the set $ H_\mu^\lambda(\bR) $ is the $ B$-algebra of $ Y^\lambda_\mu(\bR) $.  This is a finite-dimensional algebra whose maximal ideals coincide with the highest weights $ H_\mu^\lambda(\bR) $.  In Section \ref{se:Balgebra}, we give a precise description of the $B$-algebra in type A.  We also obtain a partial description of the $ B$-algebra for any $ \g $.

In Section \ref{se:ProofConjecture}, we give a combinatorial characterization of the product monomial crystal (in terms of inequalities) valid in any simply-laced type.  We then use this characterization and the description of the $B$-algebra to prove the following result.

\begin{Theorem}
\label{thm:mainthm}
 Conjecture \ref{co:intro} is true when $ \g = \mathfrak{sl}_n $.
\end{Theorem}

In section 6.5, we explain how the appearance of the product monomial crystal is closely related to the comultiplication for shifted Yangians, as defined in \cite{FKPRW}.

In the case when $\lambda$ is a multiple of the first fundamental weight, this conjecture can be deduced by the results of Brundan and Kleshchev \cite{BK}.
In Section \ref{se:resolutions}, we concentrate on the case when $ \Grlmbar $ does admit a symplectic resolution and we relate $ \B(\lambda, \bR)$ to the set of $ T$-fixed points in the resolution.  This makes further contact with \cite{BLPWgco}.

\subsection{Motivation from categorification}
Assuming our main conjecture, the set $ H_\mu^\lambda(\bR) $ of highest weights of Verma modules is the $\mu $-weight space of a $ \g^\vee $-crystal.  In fact, we conjecture that the direct sum (over $ \mu $) of the category $ \O $s for $ Y^\lambda_\mu(\bR) $ should carry a categorical action (in the sense of Khovanov-Lauda and Rouquier) of $ \g^\vee $ categorifying a representation $ V(\lambda, \bR)$ of $ \g^\vee$.  In this paper, we only consider the case where $\mu$ is dominant.  However the definition of $Y^\lambda_\mu(\bR)$ can be extended to the case of general $\mu$ following \cite{FKPRW}, \cite[Appendix B]{BFN}.

 For any $ \bR $, we conjecture that the category of finite-dimensional representations of $ Y^\lambda_\mu(\bR) $ should categorify $ \Hom_{\g^\vee}( V(\mu), V(\lambda, \bR))$.  As explained in Section \ref{se:finitedim}, this latter statement is compatible with our main conjecture.  More precisely, we show the following.

 \begin{Theorem} \label{th:intro1}
 The Verma module $ M(J, \bR) $ has a finite-dimensional simple quotient if and only if $ y(J) $ is a highest weight element of the monomial crystal.
 \end{Theorem}

In type A, 60\% of the authors use the main results of this paper to prove that the algebra $Y^\lambda_\mu(\bR) $ is isomorphic to a central quotient of a parabolic W-algebra \cite{WWY}.  Applying the main result of \cite{WebWO} shows in turn that category $\O$ for $Y^\lambda_\mu(\bR) $ is equivalent to a singular block of parabolic category $\O$ for $\mathfrak{gl}_N$, where $N$ is the sum of the parts in $ \lambda $.  Here $\la$ describes the parabolic, and $\mu$ describes the singularity of the central character.  After this equivalence, one can obtain the desired categorical action using tensor product with the standard and dual standard representations, as described  in \cite[7.4]{CR}.

We note that, together with Theorem \ref{thm:mainthm}, the isomorphism from \cite{WWY} mentioned above implies a new result about parabolic W-algebras, namely that the highest weights of the central quotient of a parabolic W-algebra are also governed by the crystals $\B(\lambda, \bR)$. 

\subsection{Symplectic duality and quiver varieties}
The product monomial crystal $\B(\lambda,\bR) $ also appears as the indexing set of connected components of Nakajima's graded quiver varieties (see Section \ref{section: Link to quiver varieties}).  In fact, this was Nakajima's original motivation for the introduction of the monomial crystal.  We use this relationship in order to prove that $ \B(\lambda,\bR) $ is actually a subcrystal of the monomial crystal.

Nakajima's graded quiver variety is defined as the fixed point set of a $\C^\times$-action on a Nakajima quiver variety $ \fM(\bm, W) $.  The quiver variety depends on $\lambda, \mu $ and the action depends on the parameters $ \bR $.

In \cite{BLPWgco}, it was observed that conical symplectic resolutions come in pairs, called symplectic duals.  While the framework of \cite{BLPWgco} is specific to varieties with resolutions, it can be extended in many ways to those which do not have resolutions.
 The variety  $ \fM(\bm, W) $ is conjectured to be the symplectic dual of $ \Grlmbar$.  In particular, as shown in \cite[\S 5.5]{Webqui} and \cite[10.4 \& 10.29]{BLPWgco}, this duality is already established in type A.  This is further motivated by the results of Braverman-Finkelberg-Nakajima \cite{BFN} who show that $\Grlmbar $ is the Coloumb branch associated to the gauge theory for which $ \fM(\bm, W) $ is the Higgs branch.

\subsection{Conjectures of Hikita and Nakajima}
Recently, Hikita \cite{H} has conjectured that if $ Y \rightarrow X $ and $Y^! \rightarrow X^! $ are symplectic dual pairs, then there is an isomorphism $ \C[X^{\C^\times}] \cong H^*(Y^!)$, where $ X^{\C^\times} $ denotes the $ \C^\times$-fixed subscheme with respect to a certain Hamiltonian $\C^\times$-action. In our context, this becomes the following statement which we prove in Section \ref{se:Hikita}.

\begin{Theorem} \label{th:intro2}
 There is an isomorphism of graded algebras $ \C[(\Grlmbar)^{\C^\times}] \cong H^*(\fM(\bm, W)) $.
\end{Theorem}

In unpublished work, Nakajima has extended Hikita's conjecture to a statement that the $ B$-algebra of the quantization of $ X $ is isomorphic to $ H_{\C^\times}^*(Y^!) $, where the $ \C^\times $ action depends on the quantization parameter.  In Section \ref{sec:conjecture-nakajima}, we explain how in our context, Nakajima's conjecture is closely related to Conjecture \ref{co:intro}.

\subsection{Organization of this paper}
The paper is organized as follows.  In section 2, we work purely combinatorially and define the product monomial crystal.  In section 3, we define truncated shifted Yangians and explain how to think of highest weights as elements of the monomial crystal.  In this section, we also formulate Conjecture \ref{co:intro} and prove Theorem \ref{th:intro1}.  In section 4, we describe the related geometry of affine Grassmannian slices.  In section 5, the most technical part of the paper, we study the $B$-algebra of $Y^\lambda_\mu $ and give a precise description of it in type A.  In section 6, we give an algebraic description of the product monomial crystal and combine it with the results of section 5 in order to prove Theorem \ref{thm:mainthm}.  In section 7, we explain the relationship between the product monomial crystal and graded quiver varieties.  Finally, in section 8, we prove Hikita's conjecture (Theorem \ref{th:intro2}) and relate our main conjecture to Nakajima's extension of Hikita's conjecture.

\subsection*{Acknowledgements}
We thank Roman Bezrukavnikov, Alexander Braverman, Pavel Etingof, Dinakar Muthiah, Hiraku Nakajima, Peng Shan, and Xinwen Zhu for useful conversations.  We also thank Alexander Tsymbaliuk for helpful comments on an earlier version of this paper.  J.K. was supported by NSERC, Sloan Foundation, Simons Foundation, and SwissMAP, and he thanks the geometry group at {\'E}PFL for their hospitality. P.T. is supported by the NSF.  B.W. is supported by the NSF and the Sloan Foundation. A.W. was supported by NSERC and the Government of Ontario. O.Y. is supported by the Australian Research Council.

\section{Monomial crystals}

\subsection{Notation}
\label{subsection:Notation}
As above, fix a semsimple simply-laced Lie algebra $ \g $.  We write $
i \con j $ if $ i, j $ are connected in the Dynkin diagram.  Fix a
bipartition $ I = I_{\bar{0}} \cup I_{\bar{1}} $ of the Dynkin diagram of $ \g $ (where $ \bar{0}, \bar{1} \in \Z/2$).  We
call vertices in $I_{\bar{0}}$ {\bf even} and those in $I_{\bar{1}}$ {\bf odd} and assume that all edges run between vertices of different parity.  Thus, we say that $i \in I $ and $k \in \Z$ have the {\bf same parity} if $i\in I_{\bar{k}}$.

Let $\Lambda $ be the coweight lattice of $ \g $.  Fix a dominant coweight $ \lambda$.  Let us write $ \lambda = \sum_i \lambda_i \varpi_i $.

We let $ \C^\lambda = \prod_i \C^{\lambda_i}/S_i $, the set of all collections of multisets of sizes $ (\lambda_i)_{i \in I} $.  A point in $ \C^\lambda $ will be written as $\bR = (R_i)_{i \in I} $ where $ R_i $ is a multiset of size $ \lambda_i $ and it will be called a \textbf{set of parameters} of weight $ \lambda $.

We say that $ \bR $ is \textbf{integral}, if for all $ i$, all elements of $R_i$ are integers and have
the same parity as $i$.   So an integral set of parameters consists of a multiset of $ \lambda_i $ even integers for every $ i \in I_0 $ and $ \lambda_i $ odd integers for every $ i \in I_1 $.

Recall that a {\bf crystal} of $ \g^\vee $ is a set $ \B $ along with a partial inverse permutations $ \tilde{e}_i, \tilde{f}_i : \B \rightarrow \B $ for all $ i \in I $ and a weight map $ wt : \B \rightarrow \Lambda $.  For each dominant coweight $ \lambda $, there is a crystal $ \B(\lambda) $.  We say that a crystal $ \B $ is {\bf normal} if it is  the disjoint union of these crystals $ \B(\lambda) $ (for varying $ \lambda $).

\subsection{Definition of the monomial crystal}
\label{subsection:Definitions}
Let $ \B $ denote the set of all monomials in the variables $ y_{i,k}
$, for $i \in I$ and $k \in \Z$ of the same parity.
Let
$$
z_{i,k} = \frac{y_{i,k} y_{i,k+2}}{\prod_{j  \con i} y_{j,k+1}}
$$

Given a monomial $ p = \prod_{i,k} y_{i,k}^{a_{i,k}} $, let
\begin{equation*} wt(p) = \sum_{i,k} a_{i,k} \varpi_i \quad
\varepsilon_i^k(p) = - \sum_{l \le k} a_{i,l} \quad
\varphi_i^k(p) = \sum_{l \ge k} a_{i,l}
\end{equation*}
and
\begin{align*}
\varepsilon_i(p) = \max_k \varepsilon_i^k(p) \quad
  \varphi_i(p) = \max_k \varphi_i^k(p)
\end{align*}
(In each case, the max is taken over those integers $ k $ of the same parity as $ i $.)

We can define the Kashiwara operators on this set of monomials by the rules:
\begin{align*}
\tilde{e}_i(p) &= \begin{cases} 0, \text{ if $\varepsilon_i(p) = 0 $} \\
z_{i,k}p, \begin{aligned}
            &\text{otherwise, where $ k $ is the smallest integer (of the same parity as $ i $)} \\
&\text{such that $ \varepsilon_i^k(p) = \varepsilon_i(p)$}
           \end{aligned}
\end{cases} \\
\tilde{f}_i(p) &= \begin{cases} 0, \text{ if $\varphi_i(p) = 0 $} \\
z_{i,k-2}^{-1}p,
\begin{aligned} &\text{otherwise, where $ k $ is the largest integer (of the same parity as $ i$)} \\
&\text{such that $ \varphi_i^k(p) = \varphi_i(p)$}
\end{aligned}
\end{cases}
\end{align*}
The following result is due to Kashiwara \cite[Proposition 3.1]{Kash}.

\begin{Theorem}
\label{Thm: B normal crystal}
$\B $ is a normal crystal.
\end{Theorem}

\subsection{Monomial product crystals}
\label{Tsection}
For any $ c \in \Z $ and $ i \in I $ of the same parity, the monomial $ y_{i,c} $ is clearly highest weight and we can consider the monomial subcrystal $ \B(\varpi_i,c) $ generated by $ y_{i,c} $.  Since $ y_{i,c} $ has weight $ \varpi_i $, we see that $ \B(\varpi_i,c) \cong \B(\varpi_i) $.

The fundamental monomial crystals for different $ c $ all look the same, they differ simply by translating the variables.  In fact, for any complex number $ c \in \C $, we can consider $ \B(\varpi_i,c) $, the crystal obtained by translating $ y_{i,k} \mapsto y_{i,k+c-c_0} $ all variables appearing in all monomials in $ \B(\varpi_i, c_0) $ (for $ c_0 $ an integer of the same parity as $ i $).  This crystal $ \B(\varpi_i, c) $ does not sit inside $ \B $, but we will need to use it on occasion in this paper.

Given a dominant coweight $\lambda $ and an integral set of parameters $ \bR $ of weight $\lambda $ as above, we define the product monomial crystal $ \B(\lambda, \bR) $ by
$$
\B(\lambda, \bR) = \prod_{i \in I, c \in R_i} \B(\varpi_i, c)
$$
In other words, for each parameter $ c \in R_i$, we form its monomial crystal $ \B(\varpi_i, c) $ and then take the product of all monomials appearing in all these crystals.  The order of the product is irrelevant because taking products of monomials is commutative.  The definition of the product monomial crystal appears in some special cases in the work of Kim and Shin \cite{KS}.

The following result will be proved as a consequence of the link with graded quiver varieties.  We do not know a combinatorial proof.
\begin{Theorem} \label{th:subcrystal}
$\B(\lambda, \bR) $ is a subcrystal of $ \B $.  In particular it is a normal crystal.  Moreover, there exists embeddings $ \B(\lambda) \subseteq \B(\lambda, \bR) \subseteq \otimes_i \B(\varpi_i)^{\otimes \lambda_i} $.
\end{Theorem}

Thus $ \B(\lambda, \bR) $ is a crystal which depends on the set of parameters $ \bR$ and lies between the crystal of the irreducible representation and the crystal of the corresponding tensor product of fundamental representations.  Exactly how $ \B(\lambda, \bR) $ depends on $\bR $ is somewhat mysterious, though we will soon discuss some results in this direction.

\subsection{Collections of multisets and monomials}
\label{subsection: Collections of multisets and monomials}
Given a collection of multisets $ \bS = ( S_i )_{i \in I} $, we can define
$$ y_\bS = \prod_{i \in I, k \in S_i} y_{i,k}, \quad z_\bS = \prod_{i \in I, k \in S_i} z_{i,k}.$$

From the definition of the monomial crystal, it is easy to see that every monomial $ p $ in $ \B(\lambda, \bR) $ is of the form
\begin{equation} \label{eq:CSmonomial} p = y_{\bR} z_\bS^{-1} = \prod_{i,k \in R_i} y_{i,k} \prod_{i,k \in S_i} \frac{\prod_{ j \con i} y_{j,k+1}}{y_{i,k} y_{i,k+2}}
\end{equation}
 for some collection of multisets $ \bS $.  Thus an alternative combinatorics for labelling elements of the monomial crystal are these collections of multisets $ \bS $.

\begin{Remark}
\label{Remark: uniqueness of monomial factorization}
For $p\in \B(\lambda,\bR)$, $\bS$ is uniquely determined.  In fact for any tuples of multisets $\bS$ and  $\bS'$, $z_\bS = z_{\bS'}$ implies $\bS = \bS'$ .
\end{Remark}

\begin{Lemma}
\label{Lemma: eq:contain}
Suppose that $ y_\bR z_\bS^{-1} \in \B(\lambda, \bR) $.  Then for all $ i$,
\begin{equation} \label{eq:contain}
S_i + 2 \subset \Big( R_i \cup \bigcup_{j \con i} ( S_j + 1 ) \Big)
\end{equation}
\end{Lemma}

\begin{proof}
Suppose that $p=y_\bR z_\bS^{-1} \in \B(\lambda, \bR)$  and $\tilde{f}_i(p)\neq0$.  Then $\tilde{f}_i(p)=z_{i,k-2}^{-1}p$ for some $k$.  We  first note that
$$
k\in \Big( R_i \cup \bigcup_{j \con i} ( S_j + 1 ) \Big)\smallsetminus (S_i+2).
$$
Indeed, since $\varphi_i(p)\neq 0$ and $k$ is the largest integer such that $\varphi_i^k(p)=\varphi_i(p)$, $y_{i,k}$ must have positive degree in $p$.  Therefore we must have $k\in R_i \cup \bigcup_{j \con i} ( S_j + 1 )$.  If the multiplicity of $k$ in $S_i+2$ is greater than or equal to the multiplicity of $k$ in $R_i \cup \bigcup_{j \con i} ( S_j + 1 )$, then $y_{i,k}$ would have nonpositive degree in $p$, a contradiction.

Now we prove the lemma in the case where $\lambda$ is a fundamental weight.  We proceed by induction on $|\bS|$.  If $|\bS|=0$ then (\ref{eq:contain}) is  tautologically true.  Suppose the (\ref{eq:contain}) is true for all $y_\bR z_\bS^{-1}\in \B(\lambda, \bR)$ such that $|S|=\ell$.  Choose $p=y_\bR z_{\bS'}^{-1}$ with $|\bS'|=\ell+1$.  Since $ \lambda $ is a fundamental weight, $ \B(\lambda, \bR) \cong B(\lambda) $  is an irreducible crystal.  Thus, $p=\tilde{f}_i(q)$ for some $i$ and $q=y_{i,c} z_\bS^{-1}\in \B(\lambda, \bR)$ with $|S|=\ell$.   Therefore $p=z_{i,k-2}^{-1}q$ for some $k$.  Then $S'_j = S_j $ for $ j\ne i $ and $ S'_i = S_i \cup\{k-2\}$.  By the inductive hypothesis and the previous paragraph, we see that the condition (\ref{eq:contain}) holds for $ \bS' $.  Thus by induction, the lemma is true when $ \lambda $ is fundamental.

Finally we prove the lemma for general $\lambda$.  Let  $p=y_\bR z_\bS^{-1} \in \B(\lambda, \bR)$ and write
$$
p=\prod_{j \in I, c\in R_j}p_{j,c},
$$
where $p_{j,c}\in \B(\varpi_j,c)$.  Define $\bR^{(j,c)}=(R^{(j,c)}_i)$ by $R^{(j,c)}_i=\{c\}$ if $i=j$ and $R^{(k,c)}_i=\emptyset$ otherwise.  Then we have collections of multisets $\bS^{(k,c)}=(S^{(k,c)}_i)$ such that
$$
p_{j,c}=y_{\bR^{(j,c)}}z_{\bS^{(j,c)}}^{-1}.
$$
By the previous case we have that for all $k,c$ and for all $i$
$$ S_i^{(j,c)} + 2 \subset \Big( R_i^{(j,c)} \cup \bigcup_{i' \con i} ( S_{i'}^{(j,c)} + 1 ) \Big). $$
Taking the multiset union over all $j,c$ proves the lemma.
 \end{proof}

Thus we can define another sequence of multisets $ \bT = (T_i)_{i \in I} $ by
$$
T_i :=  \Big( R_i \cup \bigcup_{j \con i} ( S_j + 1 ) \Big) \smallsetminus (S_i + 2)
$$
With this notation, it is easy to see that equation (\ref{eq:CSmonomial}) gives
$$
y_{\bR} z_\bS^{-1} = y_\bT y_\bS^{-1}
$$

There is a nice criterion for the highest-weight elements of the
monomial crystal in terms of $ \bS, \bT $.  Recall that an element of
a crystal is called {\bf highest-weight} if it is annihilated by all $\tilde{e}_i $.

\begin{Lemma} \label{le:highweight}
A monomial $ y_\bT y_\bS^{-1} \in \B(\lambda, \bR) $ is highest-weight if and only if for all $ i$, there exists an injective weakly decreasing map $ S_i \rightarrow T_i $.
\end{Lemma}

\begin{proof}
We note that $ \tilde{e}_i(y_\bT y_\bS^{-1}) = 0 $ if and only if $ \varepsilon_i(y_\bT y_\bS^{-1}) = 0 $ for all $ i$.  From the definition of $ \varepsilon_i $, we see that $ \varepsilon_i(y_\bT y_\bS^{-1}) = 0 $ if and only for all $ k$, we have
$$
| \{ l \in S_i : l \le k \} | \le | \{ l \in T_i : l \le k \} |
$$
This condition is clearly equivalent to the existence of an injective weakly decreasing map $ S_i \rightarrow T_i $.
\end{proof}

\subsection{Genericity and singularity}
We have an obvious upper bound on the size of the set $ \B(\lambda, \bR) $.  Since every element of $ \B(\lambda, \bR) $ is the product of monomials from the $ \B(\varpi_i, c) $, we deduce that
$$
| \B(\lambda, \bR) | \le \prod_{i \in I, c \in R_i} | \B(\varpi_i, c) |
$$
If we have equality, then we say that $ \bR $ is a \textbf{generic} integral set of parameters.  Otherwise, we say that $ \bR  $ is \textbf{singular}.

Suppose that we order all elements of $ \bR $ as $ c_1 \le \dots \le c_N $ where $ N = \sum \lambda_i $.  Let $ i_1, \dots, i_N $ be such that $ c_p \in R_{i_p} $ for all $ p $.  Let $ M = \max_i \rho^\vee(\varpi_i) $.  We say that $ \bR $ is a \textbf{well-spaced} set of parameters if $ c_p - c_{p-1} > 4M $ for all $ p $.

The following result generalizes Theorem 2.3 from \cite{KS}.

\begin{Proposition}
Suppose that $ \bR $ is a well-spaced set of parameters.  Then $ \bR $ is generic.  Moreover, the natural map of sets
$$ \bigotimes_{p=1}^N \B(\varpi_{i_p}, c_p) \rightarrow \B(\lambda, \bR) $$
is a crystal isomorphism.
\end{Proposition}

\begin{proof}
Suppose we have a variable $ y_{j,b} $ which occurs in a monomial in $ B(\varpi_{i_p}, c_p) $.  Then from the definition of the monomial crystal, we can see that $ c_p - 4M \le b \le c_p $.
(Indeed, $\rho^\vee(\varpi_i-w_0\varpi_i)=2\rho^\vee(\varpi_i)$ is the greatest number of times we can apply lowering operators to $y_{i,c}$, and each time we apply a lowering operator we at most lower the smallest ``b-value'' in the variables $y_{j,b}$ by 2.)  Thus any variable occurring in some $ B(\varpi_{i_p}, c_p) $ is bigger than any variable occurring in  $ B(\varpi_{i_{p-1}}, c_{p-1}) $.  This immediately shows that the map is injective and thus $ \bR $ is generic.

The fact that the map is a crystal morphism follows from comparing the rule for tensor product of crystals with the definition of monomial crystal.
\end{proof}

%

\begin{Remark}
Whenever $ \bR $ is a generic integral set of parameters, then $
\B(\lambda, \bR) $ is isomorphic to $ \otimes_i \B(\varpi_i)^{\otimes \lambda_i} $, since we have an embedding by Theorem \ref{th:subcrystal} and the two crystals are the same size by the definition of generic.  However, unless $ \bR $ is well-spaced, it does not seem easy to construct an isomorphism.
\end{Remark}

Note that $ y_{\bR} = \prod_{i \in I, c \in R_i} y_{i,c} $ is a highest weight element of weight $ \lambda $ in $ \B(\lambda, \bR)$.  Thus we always have a crystal containment $ \B(\lambda) \subset \B(\lambda, \bR) $.  When this containment is an equality, we say that $ \bR $ is \textbf{maximally singular}.

Define an integral set of parameters $\bR$ by $ R_i = \{ 0^{\lambda_i} \} $ for $ i \in I_{\bar{0}} $ and $ R_i = \{-1^{\lambda_i} \} $ for $ i \in I_{\bar{1}} $.  The following result will be proven in Section \ref{se:earlier}.
\begin{Proposition} \label{pr:MaxSing}
The above choice of $ \bR $ is a maximally singular set of parameters.
\end{Proposition}
It would be interesting to determine in general precise characterizations of generic and/or maximally singular sets of parameters.

\subsection{Example of $ N \varpi_1$}
Here is an instructive example.  Suppose that $ \g = \pgl_n $ and that $ \lambda = N \varpi_1 $.  Then a set of parameters $ \bR $ is just a single multiset of size $ N $.

An $n$-step flag in $ \bR $ is a sequence $V_{\bullet}=( \emptyset = V_0 \subseteq V_1 \subseteq \cdots \subseteq V_n = \bR )$ of multisubsets.  Here containment of multisubsets is defined in the obvious way.  Let $ F_n(\bR) $ denote the set all $n$-step flags in $ \bR $.

Set $\bR = \{c_1^{t_1}, \dots, c_q^{t_q}\}$, i.e. $c_i$ occurs $t_i$ times in $\bR$.  The set $ F_n(\bR) $ carries a $ \pgl_n $ crystal structure.  To describe it, for $V_{\bullet} \in F_n(\bR)$ let $V_i(c)$ be the multiplicity of $c$ in $V_i$.

To define $\tilde{e}_i(V_{\bullet})$ consider the sequence
$$
(-\cdots-+\cdots+,...,-\cdots-+\cdots+),
$$
where in the $j$-th entry there is a block of  $V_{i+1}(c_j)-V_i(c_j)$ minus signs followed by $V_i(c_j)-V_{i-1}(c_j)$ plus signs.  This sequence consists of $q$ such blocks.  Cancel all the pairs of $+-$ that occur, and call the resulting sequence the $i$-signature of $V_{\bullet}$.  Find the rightmost minus sign in the $i$-signature (if none exist then $\tilde{e}_i(V_{\bullet})=0$), and suppose it occurs in the $j^{th}$ entry.  Then $\tilde{e}_i(V_{\bullet})$ differs from $V_{\bullet}$ only in the $i^{th}$ multisubset, where $V_i$ becomes $V_i \cup \{c_j\}$.

The operator $\tilde{f}_i$ on $F_n(\bR)$ is defined similarly.  It is an easy exercise to show that
$$
F_n(\bR) \cong \B(t_1 \varpi_1) \otimes \cdots \otimes \B(t_q \varpi_1).
$$

\begin{Proposition}
\label{Prop:crys-isom}
There is an isomorphism of $ \mathfrak{sl}_n $-crystals $ F_n(\bR) \cong \B(N \varpi_1, \bR) $.
\end{Proposition}

\begin{proof}
Let $f(c,i)=\frac{y_{i,c-i+1}}{y_{i-1,c-i}}$ (with the convention that $ y_0 = y_n = 1$).  Note that as a set $\B(\varpi_1, c)=\{f(c,i):1\leq i \leq n \}$.  To the element $f(c,i)$ associate the flag $V_{\bullet}(c,i)\in F_n(\{c\})$, where the first occurrence of $c$ is in the $i^{th}$ step.
Then to define the crystal isomorphism, to an element $\prod_{c\in\bR}f(c,i_c)\in \B(N \varpi_1, \bR)$ we associate the flag $\bigcup_{c\in\bR}V_{\bullet}(c,i_c)\in F_n(\bR)$, where here union of $n$-step flags is defined componentwise in the obvious way.

Alternatively,  note that $f(c,i)=y_{1,c}z_{\bS(c,i)}^{-1}$ where
$$
\bS(c,i)=\left(\{c-2\} \{c-3\}, \dots, \{c-i\}\right).
$$
Therefore a monomial in $\B(\la,\bR)$ is equal to $y_\bR z_{\bS}^{-1}$ where $\bS=\bigcup_{c\in \bR} \bS(c,i_c)$ and $0\leq i_c \leq n-1$, and hence $\bS=(S_i)$ with
$$
S_i=\bigcup_{i_c\geq i}\{c-i-1\}.
$$
Note that $\bR\supset S_1+2 \supset S_2+3 \supset \cdots$.  Then to $y_{\bR}z_{\bS}^{-1}$ associate the flag $V_{\bullet}=(V_i)$, where
\begin{equation*}
V_i=\bR\setminus(S_i+i+1). \qedhere
\end{equation*}

\end{proof}

In particular, we see that $ \bR$ is generic iff all $ t_i $ equal 1 (i.e. if $ \bR $ is a set) and is maximally singular if $ q = 1 $ (i.e. if $ \bR $ is just a single number with multiplicity).

\subsection{Explicit description in the minuscule case}
Our goal now is to give an explicit description of the monomials occurring in $ \B(\varpi_i, c) $ whenever $ \varpi_i $ is a minuscule coweight.  This will give an explicit description of $ \B(\lambda, \bR) $ whenever $ \lambda $ is a sum of minuscule coweights.

For this section, assume that $ \varpi_i $ is minuscule.  Recall that since $ \varpi_i $ is minuscule, we have $ \langle \gamma, \alpha_j \rangle
\in \{ 1, 0, -1 \} $ for all $ \gamma \in W \varpi_i $ and $j \in I $.  Also recall that the set of weights of $ \B(\varpi_i) $ is $ W \varpi_i $.

Because $ \B(\varpi_i, c) $ is isomorphic to $ \B(\varpi_i) $, we see that the only weights occurring in $ \B(\varpi_i, c) $ are $ W \varpi_i $ and there is a unique monomial with each of these weights.  For $ \gamma \in W \varpi_i$, let $ y_{\gamma, c} $ denote the weight $ \gamma $ monomial in $ \B(\varpi_i, c) $.

If $ \beta $ is a positive root, then we write $ h(\beta) = \langle \rho^\vee, \beta \rangle $ and if $ \beta $ is any root, we write $ h(\pm \beta) $ to denote $ h(\beta) $ if $\beta $ is a positive root and $ h(-\beta) $ if $ \beta $ is a negative root.

\begin{Proposition} \label{pr:explicit}
We have
$$
y_{\gamma, c} = \prod_j y^{\langle \gamma, \alpha_j \rangle}_{j, c - h(\pm w^{-1} \alpha_j) + \langle \gamma, \alpha_j \rangle}
$$
where $ w $ is chosen to be minimal such that $ w \varpi_i = \gamma $.
\end{Proposition}

\begin{proof}
We proceed by induction on $ \gamma $ using the crystal structure.  As the base case, we take the highest weight element $ \gamma = \varpi_i $.  In this case, the formula clearly holds.

Assume that the formula holds for $ \gamma $.  Choose $ p $ such that $ \langle \gamma, \alpha_p \rangle = 1 $.  Note that because $ y_{\gamma, c} $ is the unique element of weight $ \gamma $ in the crystal, we must have $ \tilde{f}_p y_{\gamma, c} = y_{s_p \gamma, c} $.  By the definition of the crystal structure
\begin{equation} \label{eq:Yspgam}
\tilde{f}_p y_{\gamma, c} = z_{p, k-2} y_{\gamma, c} = \frac{\prod_{q \con p} y_{q,k-1}}{y_{p,k-2} y_{p, k}} \prod_j y^{\langle \gamma, \alpha_j \rangle}_{j, c - h(\pm w^{-1} \alpha_j) + \langle \gamma, \alpha_j \rangle}
\end{equation}
where $ k = c - h(w^{-1} \alpha_p) + 1$

We will now check that $ y_{s_p \gamma, c} $ is given by the formula in the statement of the Proposition.  Choose $ j \in I $. Let us consider the occurrences of $ y_{j, b} $ in $ y_{s_p \gamma, c} $.

\textbf{Case 1, $j = p $:}  In this case, we see by (\ref{eq:Yspgam}) that the factor of the form $ y_{p,b} $ occurring in $y_{s_p \gamma, c} $ is $y_{p,k-2}^{-1} $.  Since
$$
k-2 = c - h(w^{-1} \alpha_p) - 1 = c - h(-w^{-1} s_p \alpha_p) + \langle s_p \gamma, \alpha_p \rangle
$$
so the formula is correct for factors of the form $ y_{p,b} $.

\textbf{Case 2, $a_{pj} = 0 $:}  In this case, the factors of the form $ y_{j, b} $ in $ y_{\gamma,c} $ are the same as those in $ y_{s_p \gamma, c} $ and it is easy to see that this is true in the right hand side of the formula.  Thus the formula is correct for factors of the form $ y_{j,b} $.

\textbf{Case 3, $ j \con p $:} In this case, we see that since
$$ \langle s_p \gamma, \alpha_j \rangle = \langle \gamma - \alpha_p, \alpha_j \rangle  = \langle \gamma, \alpha_j \rangle + 1$$
we see that $ \langle \gamma, \alpha_j \rangle $ is either $ 0 $ or $ -1 $.  Let us split into two cases.

\textbf{Case 3a, $ \langle \gamma, \alpha_j \rangle = 0 $:} In this case, we see that by (\ref{eq:Yspgam}), the factor of the form $ y_{j, b} $ occurring in $ y_{s_p\gamma, c} $ is $ y_{j, k-1} $.  Now, by Lemma \ref{le:simple}, we see that $ w^{-1} \alpha_j $ is a simple root, and thus
$$
h(w^{-1} s_p \alpha_j) = h(w^{-1} \alpha_p + w^{-1} \alpha_j) = h(w^{-1} \alpha_p) - 1
$$
and thus
$$
k-1 = c - h(w^{-1} s_p \alpha_j) + 1
$$
which shows that the formula is correct for factors of the form $y_{j,b} $.

\textbf{Case 3b, $\langle \gamma, \alpha_j \rangle = -1 $:}  In this case, the right hand side of the formula shows that there should be no factor of $ y_{j,b} $ in $ y_{s_p \gamma, c} $.  This is also what we compute using (\ref{eq:Yspgam}), since the factor $ y_{j,k-1} $ cancels the factor $  y_{j, c - h(-w^{-1} \alpha_j) -1}$.  To see this, we need to show that $ k -1 = c - h(-w^{-1} \alpha_j)-1$ which is equivalent to showing that
$$
h(-w^{-1} \alpha_j) = h(w^{-1} \alpha_p) - 1
$$
By Lemma \ref{le:simple}, $w^{-1} s_p \alpha_j $ is a simple root and thus
$$
w^{-1} s_p \alpha_j = w^{-1} \alpha_j + w^{-1} \alpha_p
$$
implies the previous equation.
\end{proof}

\begin{Lemma} \label{le:simple}
Fix $ w \in W, i, k \in I $ and let $ \beta = w \alpha_k $.  Assume $  s_j w > w $ for all $ j \ne i $ and that $ \langle \varpi_i, \beta \rangle = 0 $. Then $ \beta $ is a simple root.
\end{Lemma}

\begin{proof}
For any $ j \ne i $, $ s_j w > w $ and thus  $ \alpha_j \in w \Delta_+$.  Hence $ w^{-1} \alpha_j $ is a positive root.

Now, since $ \langle \varpi_i, \beta \rangle = 0 $, we can write $ \beta = \sum_{j \ne i} n_j \alpha_j $ with either all $ n_j \ge 0 $ or all $n_j \le 0$.  Applying $ w^{-1} $, we see that
$$
\alpha_k = w^{-1} \beta = \sum_{j \ne i} n_j w^{-1} \alpha_j.
$$
Applying $ \rho^\vee $ to both sides, we see that $ 1 = \sum_{j \ne i} n_j \langle \rho^\vee, w^{-1} \alpha_j \rangle $.  Since all $ w^{-1} \alpha_j $ are positive roots, $ \langle \rho^\vee, w^{-1} \alpha_j \rangle > 0 $ for all $ j$.  Thus, there exists $ r $, such that $ n_r = 1 $ and $ n_j = 0 $ for $ j \ne r $.  Thus we conclude that $ \beta = \alpha_r $ is a simple root.
\end{proof}

\section{Monomials as highest weights}

\subsection{The Yangian and the shifted Yangian}

\begin{Definition}
\label{Definition: Yangian}
The Yangian $Y$ is the $\C$-algebra with generators $E_i^{(r)}, F_i^{(r)}, H_i^{(r)}$ for $i\in I$, $r\in \Z_{>0}$, and relations
\begin{align*}
[H_i^{(s)}, H_j^{(s)}] &= 0,  \\
[E_i^{(r)}, F_j^{(s)}] &=  \delta_{ij} H_i^{(r+s-1)}, \\
[H_i^{(1)}, E_j^{(s)}] &=  a_{ij} E_j^{(s)}, \\
[H_i^{(r+1)},E_j^{(s)}] - [H_i^{(r)}, E_j^{(s+1)}] &= \frac{a_{ij}}{2} (H_i^{(r)} E_j^{(s)} + E_j^{(s)} H_i^{(r)}) , \\
[H_i^{(1)}, F_j^{(s)}] &= -a_{ij} F_j^{(s)}, \\
[H_i^{(r+1)},F_j^{(s)}] - [H_i^{(r)}, F_j^{(s+1)}] &= -\frac{a_{ij}}{2} (H_i^{(r)} F_j^{(s)} + F_j^{(s)} H_i^{(r)}) , \\
[E_i^{(r+1)}, E_j^{(s)}] - [E_i^{(r)}, E_j^{(s+1)}] &= \frac{a_{ij}}{2} (E_i^{(r)} E_j^{(s)} + E_j^{(s)} E_i^{(r)}), \\
[F_i^{(r+1)}, F_j^{(s)}] - [F_i^{(r)}, F_j^{(s+1)}] &= -\frac{a_{ij}}{2} (F_i^{(r)} F_j^{(s)} + F_j^{(s)} F_i^{(r)}),\\
i \neq j, N = 1 - a_{ij} \Rightarrow
\operatorname{sym} &[E_i^{(r_1)}, [E_i^{(r_2)}, \cdots [E_i^{(r_N)}, E_j^{(s)}]\cdots]] = 0 \\
i \neq j, N = 1 - a_{ij} \Rightarrow
\operatorname{sym} &[F_i^{(r_1)}, [F_i^{(r_2)}, \cdots [F_i^{(r_N)}, F_j^{(s)}]\cdots]] = 0
\end{align*}
where $\operatorname{sym}$ denotes symmetrization over the indices $r_1,\ldots, r_N$.
\end{Definition}

There is an embedding $U\g \hookrightarrow Y$, as the subalgebra generated by the modes $X^{(1)}$.  We will use the notation $Y^>, Y^0, Y^<$ to denote the (unital) subalgebras of $Y$ generated by the $E_i^{(r)}$, the $H_i^{(r)}$, and the $F_i^{(r)}$, respectively.  Denote $Y^\geq = Y^> Y^0$ and $Y^\leq = Y^< Y^0$.

Define the \textbf{shifted Yangian} $ Y_\mu $ to be the subalgebra of $ Y $ generated by all $ E_i^{(s)} $, all $ H_i^{(s)} $ and $ F_i^{(s)} $ for $ s > \mu_i $.

Following \cite{KWWY}, we define a PBW basis for $Y$ as follows. Fix any order on the Dynkin diagram.  Then, for each positive root $\alpha$ we define $\check{\alpha}$ to be the smallest simple root, such that $\hat{\alpha} = \alpha - \check{\alpha}$ is a postive root.  Inductively, we define
$$ E_\alpha^{(r)} = [E_{\hat{\alpha}}^{(r)}, E_{\check{\alpha}}^{(1)}], \ \ F_\alpha^{(r)} = [F_{\hat{\alpha}}^{(r)}, F_{\check{\alpha}}^{(1)}] $$
This can be made compatible with $Y_\mu \subset Y$: for $s\leq \langle \mu, \alpha\rangle$ we define $F_\alpha^{(s)}$ as above, while for $s>\langle \mu, \alpha\rangle$ we define
$$ F_\alpha^{(s)} = [F_{\hat{\alpha}}^{(s-\langle \mu, \check{\alpha}\rangle)}, F_{\check{\alpha}}^{(\langle \mu, \check{\alpha}\rangle +1)} ] $$

\begin{Remark}
$Y$ and $Y_\mu$ are the filtered versions of the Yangians appearing in  \cite{KWWY}; the filtration on $Y$ is defined by $\deg X^{(r)} = r$, for $X = E_\alpha, F_\alpha, H_i$.  The filtration on $Y_\mu \subset Y$ is defined by restriction.
\end{Remark}

\begin{Proposition}[\mbox{\cite[Proposition 3.11]{KWWY}}] \label{Prop: PBW}\mbox{}
\begin{enumerate}
\item Ordered monomials in $E_\alpha^{(r)}, H_i^{(r)}, F_\alpha^{(r)}$, for $r > 0$, form a basis for $Y$,
\item Ordered monomials in $E_\alpha^{(r)}, H_i^{(r)}, F_\alpha^{(s)}$, for $r > 0, s>\langle \mu, \alpha\rangle$, form a basis for $Y_\mu$.
\end{enumerate}
\end{Proposition}

$Y$ is a Hopf algebra, see for example Chapter 12 in \cite{CP2}. The coproduct on $Y$ is defined by
\begin{align*}
\Delta(X^{(1)}) & = X^{(1)}\otimes 1 + 1\otimes X^{(1)}, \\
\Delta(H_i^{(2)}) &= H_i^{(2)}\otimes 1 + H_i^{(1)}\otimes H_i^{(1)} + 1\otimes H_i^{(2)} + \sum_{\beta >0} c_\beta F_\beta^{(1)} \otimes E_\beta^{(1)}
\end{align*}
for some constants $c_\beta$.  Moreover $Y_\mu$ is a left coideal in $Y$, by Lemma \ref{Lemma: coideal} below.

\subsection{The truncated shifted Yangians}
\label{subsection: The truncated shifted Yangians}

Let $ \mu $ be a dominant coweight such that $ \mu \le \lambda $.  Thus we can write $ \lambda - \mu = \sum m_i \alpha_i^\vee $ for $ m_i \in \N $.

Given an integral set of parameters $ \bR = \{ R_i \}_{i \in I} $ of weight $\lambda$, we define power series $ r_i(u) $ by the formula
\begin{equation} \label{eq:rfromc}
r_i(u) = u^{-\lambda_i} \prod_{c \in R_i}\left( u - \tfrac{1}{2}c\right) \frac{ \prod_{j \con i} (1- \frac{1}{2} u^{-1})^{m_j}}{(1- u^{-1})^{m_i}}
\end{equation}

Then we use these to define elements $ A_i^{(s)} \in Y_\mu $ for $ i \in I, s \ge 1 $ by the formula
\begin{equation} \label{eq: def of A}
H_i(u) = r_i(u) \frac{ \prod_{j \con i} A_j(u - \frac{1}{2})}{A_i(u) A_i(u - 1) }
\end{equation}
where
$$ H_i(u) = 1+ \sum_{s>0} H_i^{(s)} u^{-s}, \quad A_i(u) = 1 + \sum_{s>0} A_i^{(s)} u^{-s}$$
These elements $ A_i^{(s)} $ are uniquely defined, by Lemma 2.1 in \cite{GKLO}.

Finally, we define $ I^\lambda_\mu $ for the 2-sided ideal of $ Y_\mu $ generated by $ A_i^{(s)} $ for $ i \in I, s > m_i $ and we define
$ Y^\lambda_\mu(\bR) = Y_\mu / I^\lambda_\mu $.  Note that the elements $ A_i^{(s)} $ depends on $ \bR $, as does the ideal $ I^\lambda_\mu $.

\begin{Remark}
Throughout this paper, we will make frequent use of formal series: many calculations involving Yangians are simplified by working with series in $Y((u^{-1}))$.
\end{Remark}

\subsection{Verma modules for shifted Yangians}
The algebra $ Y_\mu $ has a well-known highest weight theory.  Let $ J = \{ J_i(u) \}_{i\in I} $ be a collection of power series $ J_i(u) \in 1 + u^{-1}\mathbb C[[u^{-1}]] $.
Let $M_\mu(J) $ be the universal $ Y_\mu $-module generated by a vector $ \one $ such that
\begin{equation}\label{Verma-equation}
 H_i(u) \one = J_i(u) \one , \quad E_i^{(s)} \one = 0.
\end{equation}
As usual, these modules satisfy a universal property: they represent the functor sending a $ Y_\mu $-module $M$ to the vector subspace $$\left\{m\in M : H_i(u) m = J_i(u) m, E_i^{(s)} m = 0\right\}$$
of highest weight vectors of weight $J$.

\subsection{Finite-dimensional quotients of Vermas for shifted Yangians}
The following theorem generalizes Corollary 7.10 from \cite{BK}.
\begin{Theorem} \label{th:finitedimquotient}
$M_\mu(J) $ has a finite-dimensional quotient if and only if there exist monic polynomials $ P_i, Q_i \in \mathbb C[u] $ such that the degree of $ Q_i$ is $ \mu_i $ and
$$
J_i(u) = \frac{P_i(u+1)}{P_i(u)}  \frac{Q_i(u)}{u^{\mu_i}}
$$
\end{Theorem}

\begin{Remark} \label{rmk:finitedimquotient}
The polynomials $P_i, Q_i$ are unique if we impose that $\operatorname{gcd}(P_i, Q_i) = 1$.
\end{Remark}

\begin{Lemma} \label{Lemma: coideal}
The comultiplication $\Delta:Y\rightarrow Y\otimes Y$ restricts to a map
$$ \Delta: Y_\mu \rightarrow Y \otimes Y_\mu$$
\end{Lemma}
This lemma is generalized extensively in \cite[Section 4]{FKPRW}.
\begin{proof}
It suffices to show that the generators of $Y_\mu$ have coproducts in $Y \otimes Y_\mu$.

As in \cite{CP}, the element $S_i = H_i^{(2)} - \tfrac{1}{2}(H_i^{(1)})^2$ satisfies $[S_i, E_i^{(r)}] = 2 E_i^{(r+1)}$ and
$$ \Delta(S_i) = S_i\otimes 1 + 1\otimes S_i + \sum_{\beta>0} c_\beta F_\beta^{(1)}\otimes E_\beta^{(1)} $$
for some constants $c_\beta$.  Using this a simple induction using the fact that $\Delta$ is an algebra homomorphism shows that
$$\Delta(E_i^{(r)}) \in  E_i^{(r)} \otimes 1 + Y^\leq\otimes Y^>$$
Similarly $\Delta(F_i^{(r)}) \in 1\otimes F_i^{(r)}+Y^< \otimes Y^\geq$.  Note that $[F_i^{(s)}, Y^>] \subset Y^\geq$ for any $s$, and so $\Delta(H_i^{(r)}) = [\Delta(E_i^{(r)}), \Delta(F_i^{(1)})] \in Y^\leq \otimes Y^\geq$.
\end{proof}

Let $P_i, Q_i$ be as in the statement of the theorem.  Consider the Verma module $M_Q$ for $Y_\mu$, with generator $\mathbf{1}_Q$, such that
$$ H_i(u)\mathbf{1}_Q = \frac{Q_i(u)}{u^{\mu_i}} \mathbf{1}_Q = \left(1+Q_i^{(1)}u^{-1}+\ldots+Q_i^{(\mu_i)} u^{-\mu_i}\right) \mathbf{1}_Q $$
Consider also the Verma module $M_P$ for the full Yangian $Y$, corresponding to the Drinfeld polynomials $P_i(u)$.  By \cite{D} we know that $M_P$ has a finite dimensional quotient as a $Y$--module, and hence also as a module for $Y_\mu\subset Y$.  By the previous Lemma, we may form the $Y_\mu$--module
$$ M_{P,Q} = M_P \otimes M_Q $$

\begin{Proposition}\label{drinfeld poly prod}
As modules for $Y_\mu$,
\begin{enumerate}
\item $M_Q$ has a 1-dimensional quotient $\C \mathbf{1}_Q$
\item $M_{P,Q}$ has $M_P\otimes \C \mathbf{1}_Q $ as a subquotient
\item $M_{P,Q}$ has a finite dimensional quotient
\end{enumerate}
\end{Proposition}
\begin{proof}
For (1) note that
$$ E_i^{(r)} F_j^{(s)} \mathbf{1}_Q =  \delta_{ij} H_i^{(s+r-1)} \mathbf{1}_Q =0$$
for $r\geq 1$ and $s>\mu_i$, so the $Y_\mu$--submodule generated by $\textrm{span}_{\C}\{F_i^{(s)} \mathbf{1}_Q : i\in I, s>\mu_i \}$ is proper, and the corresponding quotient is $\C \mathbf{1}_Q$.  (2) and (3) follow from (1) and the previous lemma.
\end{proof}

\begin{proof}[Proof of Theorem \ref{th:finitedimquotient}]
If $P_i, Q_i$ exist as in the statement of the Theorem, then $M_\mu(J)$ has a finite dimensional quotient $M_{P,Q}$ by the previous Proposition.

Conversely if $N$ is a finite dimensional quotient of $M_\mu(J)$, then for any fixed $i\in I$ consider the cyclic module $N_i = Y_{\mu, i} \cdot \mathbf{1}\subset N$, where
$$Y_{\mu,i} = \langle H_i^{(r)}, E_i^{(r)}, F_i^{(s)}: r>0, s>\mu_i \rangle $$
is the $i$-th ``root shifted Yangian''. Note that $Y_{\mu,i} \cong Y_{\mu_i}(sl_2)$.  Since $N_i$ is finite dimensional, Remark 7.11 from \cite{BK} applies and we obtain the existence of $P_i,Q_i$ as claimed.
\end{proof}

\subsection{Verma modules for truncated shifted Yangians} \label{Verma modules for truncated shifted Yangians}

Given $J $ as above, we can form $M^\lambda_\mu(J, \bR) $ the corresponding Verma module for $ Y^\lambda_\mu(\bR) $.  More precisely, we define
$$
M^\lambda_\mu(J, \bR) := Y^\lambda_\mu(\bR) \otimes_{Y_\mu} M_\mu(J)
$$

By Frobenius reciprocity, these modules have the same universal property in the category of $Y^\lambda_\mu(\bR)$-modules that $M_\mu(J)$ did in the category of all $Y_\mu$-modules: there is a unique map $M_\mu^\lambda(J) \to M$ sending $\one \mapsto m\in M$ if and only if $H_i(u) m = J_i(u) m$ and $E_i^{(s)} m = 0$.

The situation for $ Y^\lambda_\mu(\bR) $ is quite different in a crucial way, though.   For most choices of $J $, we will get $ M^\lambda_\mu(J,\bR) = 0 $.  In other words, for most $J$, there are no non-zero highest weight vectors of weight $J $ in any $ Y^\lambda_\mu(\bR)$-module.  The basic question that we will attempt to answer is the following.
\begin{Question}
For which $ J $ is $ M^\lambda_\mu(J,\bR) $ non-zero?
\end{Question}

If $ M^\lambda_\mu(J,\bR) $ is non-zero, then we say that $ J $ is a highest weight for $ Y^\lambda_\mu(\bR) $.  We write $ H_\mu^\lambda(\bR) $ for the set of highest weights of $ Y^\lambda_\mu(\bR) $.

Let us be more precise. Denote the polynomial ring
\begin{equation*}
 \Cartan = \C[H_i^{(s)}: i\in I, s\geq 1]
\end{equation*}  
We can identify $\Cartan \subset Y_\mu$ as the subalgebra generated by the elements of the same name.  With this in mind, we remark that we also have $\Cartan = \C[A_i^{(s)}: i\in I, s\geq 1]$.  We define a projection $\Pi\colon Y_\mu\to  \Cartan $ induced by the PBW basis ordered so that $F$'s come before $H$'s come before $E$'s.

Note that $ J $ as above can be thought of as a map from $ \Cartan$ to $ \C $, taking $ H_i^{(s)} $ to the coefficient of $ u^{-s} $ in $ J_i $.
\begin{Proposition}
\label{Prop: B algebra and Vermas}
  The Verma module $M^\la_\mu(J)$ is non-zero if and only if the weight $J$ kills $\Pi(I^\la_\mu)$.
\end{Proposition}
\begin{proof}
First note that
$$
M^\lambda_\mu(J,\bR) \ne 0 \ \text{ iff } \ I^\lambda_\mu M_\mu(J) \ne M_\mu(J) \ \text{ iff } \ \one \notin I^\lambda_\mu M_\mu(J)
$$
Now, we have a decomposition $ M_\mu(J) = \bigoplus_{k \in \N} M_\mu(J)_{-k} $ according to the eigenvalues for the action of $ \rho^\vee $ and we have that $ M_\mu(J)_0 = \C \one $.  Let $ \psi : M_\mu(J) \rightarrow \C \one $ be the projection induced by this direct sum decomposition.   Thus, $ M^\lambda_\mu(J,\bR) \ne 0 $ if and only if  $ \psi(I^\lambda_\mu M_\mu(J)) = 0 $.  Finally, note that since $ I^\lambda_\mu $ is a 2-sided ideal and $ M_\mu(J) $ is generated by $ \one$, we have $ I^\lambda_\mu M_\mu(J) = I^\lambda_\mu \one$.

Now, for any element $ a \in Y_\mu $, we have $ \psi(a \one) = J(\Pi(a)) \one $.  Thus, we see that $ \psi(I^\lambda_\mu \one) = 0 $ if and only if $J $ kills $\Pi(I^\lambda_\mu) $.  Combining with the above paragraph yields the desired result.
\end{proof}

Thus, an object of key interest is the quotient
$$B=\Cartan/\Pi(I^\la_\mu),$$ 
since the maximal ideals of this
algebra are in bijection with $J$ such that $M^\la_\mu(J)\neq 0$.

Let $ A = \oplus_{k \in \Z} A_k $ be a $ \Z$-graded algebra.  In
    \cite[\S 5.1]{BLPWgco}, the {\bf $B$-algebra} of $ A $ is defined to be
$$ B(A) = A_0 / \sum_{k>0} A_{-k} A_k $$
In the case of $ Y_\mu^\la(\bR)$, we will define a $ \Z$-grading by the action of $\rho^\vee$ (so all $ E_i^{(s)} $ have degree 1 and all $ F_i^{(s)} $ have degree $ -1$).
\begin{Proposition} \label{pr:2Balgebras}
  The quotient $B$ is canonically isomorphic to the $B$-algebra of $Y^\lambda_\mu$.
\end{Proposition}
\begin{proof}
Let $Y_{\mu,k}$ be the space of elements of weight $k$ under the action of $\rho^\vee$.
  We can write the $B$-algebra as the quotient of  $Y_{\mu,0}$ by  $\sum_{k>0} Y_{\mu,-k} Y_{\mu,k} \cap I^\la_\mu$.  If we quotient by the former first,
  we precisely kill the kernel of $\Pi$ and we just get
  $\Cartan$, with the image of $I^\la_\mu$ being
  $\Pi(I^\la_\mu)$.  If we quotient by the latter first, we get the
  $B$-algebra as defined above.
\end{proof}
The same argument shows that the modules $M^\la_\mu(J, \bR)$ are the
{\bf standard modules} $\Delta_{J}$ as defined in \cite[\S
5.2]{BLPWgco}.  That is, $M^\la_\mu(J,\bR)$ is the parabolic induction
$$Y_\mu^\la(\bR)\otimes_{Y_\mu^\la(\bR)^\ge}\C_J,$$ where $\C_J$ is
the corresponding 1-dimensional representation of $Y_\mu^\la(\bR)^\ge$ factoring through $B$.

By Theorem 4.8 in \cite{KWWY} (cf. Theorem \ref{th:quantize} below), the algebra $Y^\lambda_\mu(\bR) $ is the quantization of a scheme supported on the affine Grassmannian slice $\Grlmbar$.  Thus the argument in the proof of Proposition 5.1 of \cite{BLPWgco} applies to deduce the following result.

\begin{Proposition}
The algebra $ B $ is finite-dimensional.
\end{Proposition}
Thus, we conclude the following basic result.
\begin{Corollary}
 For $\bR$ fixed,  there are only finitely many $J$ such that $M^\la_\mu(J,\bR)\neq 0$.
\end{Corollary}

\subsection{From highest weights to monomial crystals}
\label{subsec:highest}
Let $ J \in H_\mu^\lambda(\bR) $.  Note that each $ J_i $ is a rational function of $ u $ because of the formula relating $ H_j(u)$ and $ A_i(u) $ and because $A_i^{(s)} \one = 0$ for $s>m_i$.  Thus each $ J_i $ can be written as a product of linear factors and their inverses.  That is, there is a collection of multiplicities $a_{i,k}\in \Z$ for each $k\in \C$ such that
 $$
J_i(u) = u^{-\mu_i} \prod_k (u - \tfrac{1}{2} k)^{a_{i,k}}.
$$
We will prove later, in Proposition \ref{prop:integrality and parity}, that if $a_{i,k} \neq 0$ then we must have $k\in \Z$ and $i\in I_{\overline{k}}$ (see Section \ref{subsection:Notation}). Assuming this result, we can define $ y(J) = \prod_{i,k} y_{i,k}^{a_{i,k}} \in \B$ to be the Nakajima monomial obtained by converting any factor $ u - \tfrac{1}{2}b $ occurring in $ u^{\mu_i} J_i(u) $ into $ y_{i,b} $.
This leads us to the main conjecture of the paper, which is proved in type A as Corollary \ref{Cor:typeA}. (This is a reformulation of Conjecture \ref{co:intro} from the introduction.)

\begin{Conjecture} \label{co:main}
The map $ J \mapsto y(J) $ is a bijection between $ H_\mu^\lambda(\bR) $ and $ \B(\lambda, \bR)_\mu$, the $\mu$-weight set in $\B(\lambda,\bR)$.
\end{Conjecture}

We can reformulate the map $ J \mapsto y(J) $ as follows.  Suppose we have a Verma module $ M^\lambda_\mu(J,\bR) $ for $ Y^\lambda_\mu(\bR) $.  Let us consider the action of $ A_i^{(s)} $ on $ \one$.  Since $ A_i^{(s)} = 0 $ for $ s > m_i $, we can write
$$
A_i(u) \one = \prod_{k \in S_i} (1 - \tfrac{1}{2} k u^{-1} )  \one
$$
for some multiset $S_i $.
Thus we get a collection of multisets $ \bS = ( S_i )_{i \in I} $ which determine $J $ by
$$
J_i(u) = u^{-\mu_i}\prod_{c \in R_i} (u - \tfrac{1}{2}c) \frac{ \prod_{j \con i} \prod_{k \in S_j}  (u - \frac{1}{2}k - \frac{1}{2})}{\prod_{k \in S_i} (u - \frac{1}{2}k)(u - \frac{1}{2}k - 1)}
$$
From this, it is easy to see that $ y(J) = y_\bR z_\bS^{-1}$.

\begin{Remark}
 Throughout this paper, we assume that $ \bR $ is integral.  This is done in order to get a natural crystal structure on $ \B(\lambda, \bR) $ and in order to make contact with quiver varieties.

 However, the crystal $ \B(\varpi_i, c) $ is well-defined for any $ c \in \C $.  Thus $ \B(\lambda, \bR) $ can still be defined as a set (with weight function) for any complex parameters $ \bR $.  It's most natural to think of this set as a crystal for a sum $\mathfrak{g}\oplus\cdots \oplus \mathfrak{g}$ with one factor for each coset of $\C/2\Z$ which contains one of the parameters.  Conjecture \ref{co:main} thus makes perfect sense for complex parameters as well.  In fact our proof in type A (Corollary \ref{Cor:typeA}) applies equally well to prove the conjecture in this generality.
\end{Remark}

\subsection{Finite-dimensional simples} \label{se:finitedim}
Let $ J \in H^\lambda_\mu(\bR) $, so that $ M^\lambda_\mu(J,\bR) $ is a non-zero Verma module for $Y^\lambda_\mu(\bR) $.
As in Section \ref{Tsection}, let us write $ y(J) = y_\bT y_\bS^{-1} $ for two sequences of multisets $ \bS, \bT $.

\begin{Proposition}\label{fd <-> weakly decreasing}
$M^\lambda_\mu(J,\bR) $ has a finite-dimensional simple quotient if and only if there exists a weakly decreasing injective map $ S_i \rightarrow T_i $ for all $i $.
\end{Proposition}

Thus, if we assume Conjecture \ref{co:main}, we can combine the previous Proposition and Lemma \ref{le:highweight} to obtain the following result.

\begin{Corollary}
There is a bijection between finite-dimensional simple $ Y^\lambda_\mu(\bR) $ modules and highest weight elements of $ \B(\lambda, \bR)_\mu $.
\end{Corollary}

To prove Proposition \ref{fd <-> weakly decreasing} we need the following lemma.

\begin{Lemma}\label{chain decomp}
Let $a_1, \dots, a_m, b_1, \dots, b_m $ be complex numbers which all lie in the same $ \Z $ coset.  Then there exists a polynomial $P\in \C[u]$ such that
$$ \frac{(u-b_1)\cdots (u-b_m)}{(u-a_1)\cdots (u-a_m)} = \frac{P(u+1)}{P(u)}$$
if and only if there exists a weakly decreasing bijection $\{a_1,\ldots,a_m\} \rightarrow \{b_1,\ldots,b_m\}$.
\end{Lemma}
\begin{proof}
Suppose there exists a weakly decreasing bijection, so $a_k = b_k+n_k$ for some $n_k\in \mathbb{Z}_{\geq 0}$. Then
$$ \frac{u-b_k}{u-b_k-n_k} = \frac{(u-b_k)(u-b_k-1)\cdots(u-b-n_k+1)}{(u-b_k-1)\cdots(u-b_k-n_k+1)(u-b-n_k)} = \frac{P_k(u+1)}{P_k(u)} $$
for $P_k(u) = (u-b_k-1)\cdots(u-b_k-n_k)$.  Take $P(u) = P_1(u)\cdots P_m(u)$.

Conversely, given $P(u)$ we have
$$ (u-b_1)\ldots (u-b_m) P(u) = P(u+1) (u-a_1)\cdots (u-a_m)$$
Choose a maximal length string of roots $r+1,\ldots,r+n$ of $P(u)$.  Then we must have $r= b_i$ and $r+n = a_j$ for some $i$ and $j$.  Dividing $(u-r)\cdots(u-r-n)$  from both sides, we proceed inductively until we exhaust all factors.
\end{proof}
\begin{proof}[Proof of Proposition \ref{fd <-> weakly decreasing}]
We may write $J_i(u) = \frac{\prod_{k\in T_i}(u-\tfrac{1}{2}k)}{u^{\mu_i} \prod_{k\in S_i}(u-\tfrac{1}{2}k)}$.  Note that $M_\mu^\lambda(J)$ has a finite-dimensional quotient if and only if $M_\mu(J)$ does, so we may apply Theorem \ref{th:finitedimquotient}.

If there is a weakly decreasing injection $\phi_i: S_i\rightarrow T_i$, we define $$Q_i(u) = \prod_{k \in T_i \setminus \phi_i(S_i)} (u-\tfrac{1}{2}k)$$ and $P_i(u)$ via the lemma (note that the sizes are correct, i.e. $| T_i \setminus \phi_i(S_i) | = \mu_i$).  Conversely if we are given the $P_i$ and $Q_i$, we may assume $(P_i, Q_i)=1$ by Remark \ref{rmk:finitedimquotient}.  Then the roots of $Q_i$ must lie in $\tfrac{1}{2}T_i$; dividing by $Q_i$ we may again appeal to the lemma.
\end{proof}

\section{Relationship to the geometry of affine Grassmannian slices} \label{se:resolutions}
For the purposes of this section, we assume that $ \lambda $ is a sum of \emph{minuscule} coweights.  Our main goal in this section is to show that the representation theory of $ Y^\lambda_\mu(\bR) $ behaves as if it were the global section of a sheaf of algebras quantizing a symplectic resolution.

\subsection{The affine Grassmannian slices}
Let $ G $ be the adjoint group with Lie algebra $ \g $.  Let $ \lambda, \mu $ be dominant coweights as before.  They give rise to elements $ t^\lambda, t^\mu \in G[t,t^{-1}] $.

Let $ \Gr = G((t))/G[[t]] $ be the affine Grassmannian of $ G $.  Let $ \Gr^\lambda = G[[t]]t^\lambda $ and $ \Gr_\mu = G_1[t^{-1}] t^{w_0 \mu} $, where $ G_1[t^{-1}] $ is the kernel of the evaluation map $ G[t^{-1}] \rightarrow G $.

Let $ \Grlmbar = \overline{\Gr^\lambda} \cap \Gr_\mu $ be the affine Grassamannian slice.  This is a transverse slice to $ \Gr^\mu $ inside of $ \overline{\Gr^\lambda} $.

In \cite{KWWY}, we explained that $ \Grlmbar $ carries a natural Poisson structure.  Moreover, we proved the following result (Theorem 4.8 of \cite{KWWY}).

\begin{Theorem} \label{th:quantize}
There is a map of graded Poisson algebras $\gr(Y^\lambda_\mu(\bR)) \rightarrow \C[\Grlmbar] $ which is an isomorphism modulo the nilradical of the left hand side.
\end{Theorem}

In other words, we proved that $ Y^\lambda_\mu(\bR) $ is a filtered deformation of the coordinate ring of a scheme supported on $ \Grlmbar$.  In \cite{KWWY}, we conjectured that $ \gr(Y^\lambda_\mu(\bR)) $ is in fact reduced and thus it gives a quantization of $ \Grlmbar $.

\begin{Remark}
In fact, in \cite{KWWY}, we worked with a slightly different definition of $ Y^\lambda_\mu(\bR) $, defined as the image of $Y_\mu $ under a map defined by Gerasimov-Kharchev-Lebedev-Oblezin \cite{GKLO}.  Let us temporarily denote that algebra by $ \hat{Y}^\lambda_\mu(\bR) $.  Examining the proof of Theorem 4.8 of \cite{KWWY}, we can see that Theorem \ref{th:quantize} does hold for our $ Y^\lambda_\mu(\bR) $.  Moreover, Theorem 4.10.2 of \cite{KWWY} shows that if Conjecture 2.20 of \cite{KWWY} holds, then the natural map $ Y^\lambda_\mu(\bR) \rightarrow \hat{Y}^\lambda_\mu(\bR)$ is an isomorphism.
\end{Remark}

Fix a sequence $ (i_1, \dots, i_N) $ so that $ \lambda = \sum_{p=1}^N \varpi_{i_p} $.  Following \cite[\S 2.3]{KWWY}, this sequence defines a symplectic resolution $ \widetilde{\Gr^\lambda_\mu} \rightarrow \Grlmbar $, where we define
$$
\widetilde{\Gr^\lambda} = \Gr^{\varpi_{i_1}} \tilde{\times} \cdots \tilde{\times} \Gr^{\varpi_{i_N}}  \text{ and } \widetilde{\Gr^\lambda_\mu} = m^{-1}(\Gr_\mu)
$$
Here $$  \Gr^{\varpi_{i_1}} \tilde{\times} \cdots \tilde{\times} \Gr^{\varpi_{i_N}} := \{ ([g_1], \dots, [g_N]) \in \Gr^N : g_{k-1}^{-1}g_k \in \Gr^{\varpi_{i_k}} \text{ for $ k = 1, \dots, N$} \} $$ 
and $m : \widetilde{\Gr^\lambda} \rightarrow \overline{\Gr^\lambda} $ is the convolution morphism which maps $([g_1], \dots, [g_N]) $ to $ [g_N]$.

An integral set of parameters $ \bR $ defines a $T$-equivariant line bundle $ \O(\bR) $ on $ \widetilde{\Gr^\lambda_\mu}$.

\subsection{Torus fixed points}
The action of the torus $ T \subset G $ on $ \Grlmbar $ extends to the symplectic resolution.

Suppose we have a sequence $ \bgam := (\gamma_1, \dots, \gamma_N) $, such that $ \gamma_p \in W \varpi_{i_p} $ for all $ p $ and such that $ \gamma_1 + \dots + \gamma_N = \mu $. Then we can define a point
$$
t^\bgam = (t^{\gamma_1}, t^{\gamma_1 + \gamma_2}, \dots, t^{\gamma_1 + \dots + \gamma_N}) \in \widetilde{\Gr^\lambda_\mu}
$$

For $ \gamma = w \varpi_i \in W\varpi_i $, let us define $ \Delta(\gamma) = w \Delta^i_+ \cap \Delta_- $, where $ \Delta^i_+ $ denotes those positive roots with a non-zero coefficient of $ \alpha_i $ ($\Delta(\gamma)$ does not depend on $ w $) and let $\chi(\gamma) = \sum_{\beta \in \Delta(\gamma)} \beta $.

We begin with a combinatorial fact about this definition.
\begin{Lemma} \label{le:CombinatorialFact}
 $\chi(\gamma) = w\rho - \rho $ where $ w $ is the minimal element such that $ \gamma = w \varpi_i $.
\end{Lemma}

\begin{proof}
 First note that
$$ w\rho - \rho = \frac{1}{2}( \sum_{\beta \in \Delta_+} w \beta - \sum_{\beta \in \Delta_+} \beta) = \sum_{\beta \in \Delta_+, w \beta \in \Delta_-} w \beta $$
Next, note that for $ j \ne i $, the minimality of $ w $ means that $ w \alpha_j \in \Delta_+ $.  Thus for all $ \beta \in \Delta_+ \smallsetminus \Delta_+^i $, we have $ w \beta \in \Delta_+$.  Combining these observations, the result follows.
\end{proof}

The following geometric fact is our motivation for introducing the set $ \Delta(\gamma) $.  (Recall that  the elements of $\bR$ are $c_1\leq \cdots \leq c_N$.)
\begin{Lemma} \label{le:fixedpoint}
Let $ \bgam $ be as above.
\begin{enumerate}
\item $t^\bgam $ is a $T$-fixed point in $ \widetilde{\Gr^\lambda_\mu} $ and all $T$-fixed points are of this form.
\item The $ T $ action on the line $ \O(\bR)_{t^\bgam} $ is by the weight $ \sum_{p=1}^N (c_p -1) \gamma_p $.
\item The multiset of negative $ T $ weights on $ T_{t^\bgam} \widetilde{\Gr^\lambda_\mu} $ is $ \cup_{p=1}^N \Delta(\gamma_p) $.
\end{enumerate}
\end{Lemma}

\begin{proof}
(1) is standard and (2) follows from the definition of the line bundle.

For (3), we observe that for $ \varpi_i $ minuscule, $ Gr^{\varpi_i} $ is a partial flag variety and that the $ T$-weights on $ T_\gamma Gr^{\varpi_i} $ are $ w \Delta^i_+ $ (for $ \gamma \in W \varpi_i $).  Thus the $ T $-weights on $ T_{t^\bgam} \widetilde{\Gr^\lambda} $ are
$$
\bigcup_{p = 1}^N w_p \Delta_+^{i_p}
$$
where $ \gamma_p = w_p \varpi_{i_p} $.

Thus the negative $ T$-weights on $ T_{t^\bgam} \widetilde{\Gr^\lambda_\mu} $ are a subset of $\cup_{p=1}^N \Delta(\gamma_p) $.

On the other hand
$$
\dim \Gr^\lambda_\mu = 2 \rho^\vee(\lambda - \mu) = 2 \sum_p \rho^\vee(\varpi_{i_p} - \gamma_p) = 2 \sum_p | \Delta(\gamma_p)|
$$
Since $ T $ acts Hamiltonianly on $ \widetilde{\Gr^\lambda_\mu} $ with isolated fixed points, the dimension of the negative $T$-weight space on $ T_{t^\bgam} \Gr^\lambda_\mu $ is half the dimension of $ \widetilde{\Gr^\lambda_\mu}$.  Thus the result follows.
\end{proof}

\subsection{From Verma modules to torus fixed points}
Now assume that $ \bR $ is a generic integral set of parameters.  In this case every element of $ \B(\lambda, \bR)_\mu $ can be written uniquely as
$$
y_{\bgam, \bR} := \prod_{p=1}^N y_{\gamma_p, c_p}
$$
where $ \bgam $ is a sequence as in the previous section and $ y_{\gamma, c} $ denotes the monomial given by Proposition \ref{pr:explicit}.

Consider a Verma module $ M^\lambda_\mu(J,\bR) $, such that $ y(J) = y_\bR z_\bS^{-1} $.  It is generated by a highest weight vector $ \one $.  The vector $ \one $ is a weight vector for the abelian Lie algebra $ \mathfrak t $ spanned by the $H_i^{(1)} $.  We identify $ \mathfrak t $ with the Lie algebra of $ T $ by identifying $ H_i^{(1)} $ with the usual element $ H_i $.  Thus we can regard $ \one $ as an weight vector for $ T $.  It is easy to see that it has weight $ \tau(J) := \sum_{i, k \in S_i} \frac{1}{2}k \alpha_i $ (since $ A_i^{(1)} $ represents a negative fundamental coweight under the above identification).

Assuming Conjecture \ref{co:main}, we have the following result.
\begin{Theorem}\label{th:fixed-bijection}
The bijection $ H^\lambda_\mu(\bR) \rightarrow \B(\lambda, \bR)_\mu $ gives us a bijection between Verma modules for $ Y^\lambda_\mu(\bR) $ and torus fixed points in $ \widetilde{\Gr^\lambda_\mu} $ such that if $ M^\lambda_\mu(J,\bR) $ corresponds to $ t^\bgam$, then $\tau(J) $ is the weight of $ T $ acting on $( \det (T_{\bgam} \widetilde{\Gr^\lambda_\mu})_- \otimes \O(\bR)_{t^\bgam}^* \otimes \O(\bR)_{t^\blam} )^{\otimes \frac{1}{2}}$.
\end{Theorem}

Here in the statement of the theorem $ \blam = (\varpi_{i_1}, \varpi_{i_1} + \varpi_{i_2}, \dots, \lambda) $.

\begin{proof}
Let $ i \in I $ be minuscule, let $ c \in \Z $ of the same parity as $ i $, and let $ \gamma \in W \varpi_i $.  Let us write $ y_{\gamma, c} = y_{i,c} z_{\bS(\gamma)}^{-1} $ for some $ \bS(\gamma) $.  Define $ \tau(\gamma) = \sum_{j, k \in S_j(\gamma)} \frac{k}{2} \alpha_j $.

From Lemma \ref{le:fixedpoint}, we see that by additivity of $T$-weights, in order to prove the Theorem it suffices to show that
\begin{equation} \label{eq:taugamma}
2\tau(\gamma)  = (c-1)(\varpi_i- \gamma) + \chi(\gamma).
\end{equation}

In order to prove (\ref{eq:taugamma}), we proceed by induction on $\gamma$ as in the proof of Proposition \ref{pr:explicit}.  First note that when $ \gamma = \varpi_i $, then both sides of (\ref{eq:taugamma}) are 0.  So now fix some $ \gamma $ for which (\ref{eq:taugamma}) holds and choose $ p $ such that $ \langle \gamma, \alpha_p \rangle = 1$.  We will now prove (\ref{eq:taugamma}) for $ s_p \gamma $.

From the proof of Proposition \ref{pr:explicit}, we see that $ y_{s_p \gamma, c} = z_{p,c - h(w^{-1} \alpha_p) -1} y_{\gamma,c}$ and thus
$$
2\tau(s_p \gamma) = 2\tau(\gamma) + (c - h(w^{-1} \alpha_p) - 1) \alpha_p
$$

On the other hand, $ \langle \gamma, \alpha_p \rangle = 1$ implies that $ w^{-1} \alpha_p \in \Delta_+^i $ and thus we see that $ \Delta(s_p \gamma)  = s_p \Delta(\gamma) \cup \{ -\alpha_p \} $.  Moreover $ s_p \gamma = \gamma - \alpha_p $.  So we have
$$
(c-1)(\varpi_i- s_p \gamma) + \chi(s_p\gamma) = (c-1)(\varpi_i - \gamma) + (c-1)\alpha_p + \chi(\gamma) - \langle \chi(\gamma), \alpha_p) \rangle \alpha_p - \alpha_p
$$

Since (\ref{eq:taugamma}) holds for $\gamma$, we see that (\ref{eq:taugamma}) holds for $s_p \gamma $ if
$$
h(w^{-1} \alpha_p) = \langle \chi(\gamma), \alpha_p \rangle + 1
$$

Finally, note that this last equation is equivalent to
$$
\langle w\rho - \chi(\gamma), \alpha_p \rangle = 1
$$
which follows from Lemma \ref{le:CombinatorialFact}.
\end{proof}

\subsection{Relation to geometric category $ \O $}

This identification between Verma modules and torus fixed points is quite natural from the perspective of the work
\cite{BLPWgco} of the third author and his collaborators.

Recall that by the work of Bezrukavnikov-Kaledin \cite{BeKa}, a smooth symplectic variety equipped with a conic $\C^*$-action, such as $\widetilde{\Gr^\lambda_\mu} $, has an associated universal family of quantizations.  In \cite[Conjecture 4.11]{KWWY}, we gave a precise version of the following conjecture.
\begin{Conjecture}
  $Y^\la_\mu(\bR)$ is the universal family of quantizations.
\end{Conjecture}
Assuming this conjecture, the paper \cite{BLPWgco} gives the following result, which is proven by geometric methods.
\begin{Theorem}
  For a Zariski dense set of $\bR$, the category $\O$ over $Y^\lambda_\mu(\bR)$ is a highest
  weight category whose simple objects are in canonical bijection with
  the torus fixed points on $\widetilde{\Gr^\lambda_\mu} $.
\end{Theorem}
In the present paper, we are working with algebraic/combinatorial methods.  Conjecture \ref{co:main} and Theorem \ref{th:fixed-bijection} gives such a bijection between simple objects and torus fixed points.

\section{Description of the $B$-algebra} \label{se:Balgebra}

\subsection{Outline of strategy}

Our goal in this section will be to describe generators for the ideal of the $ B$-algebra for $Y_\mu^\lambda(\bR)$.   By Proposition \ref{pr:2Balgebras}, the $B$--algebra is canonically isomorphic to 
$$  B = \Cartan/ \Pi(I_\mu^\lambda), $$
where $ I_\mu^\lambda \subset Y_\mu $ is the ideal of $ Y_\mu^\lambda(\bR) $.

To compute the $B$-algebra, we can allow ourselves more flexibility by working with the larger left ideal $L_\mu^\lambda = Y \cdot I_\mu^\lambda $ in $Y$. This is justified by the following result:

\begin{Lemma}  \label{lem:Compute with L}
The $B$-algebra for $Y_\mu^\lambda(\bR)$ can be computed using $L_\mu^\lambda$:
$$ B = \Cartan / \Pi(L_\mu^\lambda) $$
\end{Lemma}
\begin{proof}
We need to show that $ \Pi( I_\mu^\lambda ) \supset \Pi( L_\mu^\lambda)$.  By Proposition \ref{Prop: PBW}, the PBW bases for $Y$ and $Y_\mu$ are compatible: any $x\in Y$ can be written uniquely as a sum $x = \sum_{a} F_a x_a$, where each $x_a \in Y_\mu$ and $F_a$ is a monomial in the generators $F_\alpha^{(r)}$ with $r\leq \langle \mu, \alpha\rangle$.

For $x\in L_\mu^\lambda$, we have all $x_a\in I_\mu^\lambda$.  But under $\Pi$, any summand where $F_a$ is non-trivial is sent to zero. Hence $\Pi(x) \in \Pi(I_\mu^\lambda)$.
\end{proof}

Here is an outline of this section.
\begin{enumerate}
\item Define elements $ T_{\beta, \gamma}(u) $ of the Yangian $ Y $ which lift generalized minors $ \Delta_{\beta, \gamma}(u) $.
\item Give a presentation for $L_\mu^\lambda$ as a left $Y$-ideal
$$ L_\mu^\lambda = Y S_\mu^\lambda $$
where $S_\mu^\lambda$ is an explicit set of elements related to the lifted minors $T_{\beta,\gamma}(u)$.  In this step, we have partial results for general type, but complete results only in type A.
\item In type A, use the generators of $ L_\mu^\lambda $ to give generators of the ideal defining the $ B$-algebra.
\end{enumerate}

\subsection{Generalized minors}

We recall the Poisson structure on generalized minors as in \cite{KWWY}.  For a dominant weight $\tau$ and $\beta, \gamma\in V(\tau)$, consider the function
$$ \Delta_{\beta, \gamma} : G \rightarrow \C, \quad g \mapsto \langle \beta, g\gamma\rangle $$
where $\langle \cdot, \cdot\rangle$ denotes the Shapovalov form on $V(\tau)$, defined with respect to a highest weight vector $v_\tau$.  Recall its definition: letting $\omega$ denote the Chevalley involution of $\g$ defined by $e_i \mapsto -f_i$, $f_i\mapsto - e_i$, and $h \mapsto -h$  for $i\in I$ and $h\in\h$, the Shapovalov form on $V(\tau)$ is uniquely determined by requiring
$$
\langle v_\tau, v_\tau \rangle = 1, \quad \langle x \beta, \gamma \rangle  = - \langle \beta, \omega(x) \gamma\rangle,
$$
for all $x\in \g, \beta$ and $\gamma \in V(\tau)$.

Using $\Delta_{\beta,\gamma}$ we define the function $\Delta_{\beta,\gamma}^{(r)}$ on $G_1[[t^{-1}]]$, for $r\in \Z_{\geq 0}$, whose value on $g$ is the coefficient of $t^{-r}$ in $\langle \beta, g\gamma\rangle$.  Since $G_1[[t^{-1}]]$ is a pro-unipotent group, these functions generate $\O(G_1[[t^{-1}]])$.  We form the formal series
$$ \Delta_{\beta, \gamma}(u) = \sum_{r\geq 0} \Delta_{\beta, \gamma}^{(r)} u^{-r} $$

There is a Poisson structure on $G_1[[t^{-1}]]$ and on $\Gr$, coming from the Manin triple $(\g((t^{-1})),\g[t],t^{-1}\g[[t^{-1}]])$ induced by the non-degenerate bilinear form $(f(t), g(t)) = - \text{res}_0 (f(t),g(t))_\g$.  Here $(\cdot,\cdot)_\g$ is a non-degenerate symmetric bilinear form on $\g$, extended by $\C((t^{-1}))$-linearity. Choose dual bases $\{J_a\}, \{J^a\}$ for $\g$ with respect to $(\cdot,\cdot)_\g$.

\begin{Proposition}[Proposition 2.13, \cite{KWWY}] \label{Poisson bracket def}
The Poisson bracket on $\O(G_1[[t^{-1}]])$ is given by
$$ (u-v)\left\{ \Delta_{\beta_1,\gamma_1}(u), \Delta_{\beta_2,\gamma_2}(v)\right\} = \sum_a \left( \Delta_{J_a \beta_1, \gamma_1}(u) \Delta_{J^a \beta_2, \gamma_2}(v) - \Delta_{ \beta_1,J_a\gamma_1}(u) \Delta_{ \beta_2,J^a \gamma_2}(v) \right) $$
\end{Proposition}

There are two subtleties in this result: First, the right-hand side has the opposite sign compared to Proposition 2.13 of \cite{KWWY}, corresponding to the opposite sign $r$--matrix.  Denoting the Casimir $2$-tensor by $C = \sum_a J_a \otimes J^a$, the choice $r(u,v) = C / (u-v)$ from \cite{KWWY} corresponds to the canonical Lie bialgebra inclusion $\g[t]\hookrightarrow \g((t^{-1}))$, which is quantized by the Drinfeld Yangian.  However, it induces the {\em opposite} Lie bialgebra structure on the dual space $\g[t]^\ast = t^{-1}\g[[t^{-1}]]$ by Proposition 1.4.2 in \cite{CP}.  Since Drinfeld-Gavarini duality puts the correct (non-opposite) Lie bialgebra structure on the dual, the sign in \cite{KWWY} was wrong: we should have used the negative $-r(u,v)$, which is the Poisson structure quantized by $Y$. We have taken the corrected sign here.

Second, in \cite{KWWY} the generalized minors were defined with $\beta\in V(\tau)^\ast$, whereas here we have used the Shapovalov form to identify $\beta \in V(\tau)^\ast \cong V(\tau)$.  Thus in our current conventions where $\beta_i, \gamma_i \in V(\tau_i)$, translating from \cite{KWWY} we more naturally see the sum
$$
\sum_a \Delta_{ \omega(J_a) \beta_1, \gamma_1}(u) \Delta_{\omega(J^a) \beta_2,\gamma_2}(v)
$$
But this agrees with Proposition \ref{Poisson bracket def}, since in fact $\sum_a \omega(J_a)\otimes \omega(J^a) = \sum_a J_a \otimes J^a$: both are $\g$--invariant elements of $\g\otimes \g$ so must be proportional, and comparing coefficients of e.g.~$e_i\otimes f_i$ we see they are equal.

\begin{Lemma}
There is an action of $\g$ on $\O(G_1[[t^{-1}]])$ defined by
$$ e_i  \mapsto \left\{ \Delta_{f_i v_{\varpi_i}, v_{\varpi_i}}^{(1)}, -\right\}, \quad
f_i \mapsto \left\{ \Delta_{v_{\varpi_i},f_i v_{\varpi_i}}^{(1)}, -\right\} $$
In particular, the Cartan weight of $\Delta_{\beta,\gamma}^{(r)}\in \O(G_1[[t^{-1}]])$ with respect to this action is $\operatorname{wt}(\gamma-\beta)$.
\end{Lemma}
\begin{proof}
There is an embedding $\g \subset Y$, and so a Hamiltonian action of $\g$ on the classical limit $\O(G_1[[t^{-1}]])$ of $Y$.  The explicit images follow from Theorem 3.9 in \cite{KWWY}. 
\end{proof}

Explicitly, using Proposition \ref{Poisson bracket def} we have
\begin{align*}
\left\{ \Delta_{f_i v_{\varpi_i},v_{\varpi_i}}^{(1)}, \Delta_{\beta,\gamma}(u)\right\} &= \Delta_{f_i \beta, \gamma}(u) - \Delta_{\beta,e_i \gamma}(u) \\
\left\{ \Delta_{v_{\varpi_i},f_i v_{\varpi_i}}^{(1)}, \Delta_{\beta,\gamma}(u)\right\} &= \Delta_{e_i \beta, \gamma}(u) - \Delta_{ \beta, f_i\gamma}(u)
\end{align*}

In addition to the Chevalley involution $\omega$, consider the involution $\iota$ of $\g$ defined by $e_i\mapsto - e_i$, $f_i\mapsto -f_i$, and $h\mapsto h$. From the above explicit formulas for the $\g$ action on minors,
\begin{Corollary} \label{Cor:cyclic}
Consider the cyclic representation $V_i$ of $U\g$ generated by the vector $\Delta_{\varpi_i,\varpi_i}(u) \in \O(G_1[[t^{-1}]])[[u^{-1}]]$.  Then there is an isomorphism
\begin{align*}
V_i & \stackrel{\sim}{\longrightarrow}  (\iota\circ \varpi)^\ast\Big( V(\varpi_i) \Big) \otimes \iota^\ast \Big(V(\varpi_i)\Big)  \\
\Delta_{\beta,\gamma}(u) & \longmapsto \beta \otimes \gamma
\end{align*}
In particular, $V_i \cong V(\varpi_i^\ast)\otimes V(\varpi_i)$.
\end{Corollary}

\subsection{Lifting generalized minors}
As in Section \ref{subsection: The truncated shifted Yangians}, given a set of parameters $\bR = \{R_i\}_{i\in I}$ we define elements $A_i^{(s)}$ in $Y_\mu$.    Since $Y_\mu \subset Y$, let us momentarily think of the $A_i^{(s)} \in Y$.  Then following \cite{GKLO}, we may consider also elements of $Y$ defined by
\[B_i(u) = A_i(u) E_i(u)\quad C_i(u) = F_i(u) A_i(u)\]\[ D_i(u) = H_i(u) A_i(u)  + F_i(u) A_i(u) E_i(u),\]

\begin{Proposition}[Proposition 2.1, \cite{GKLO}]
\label{Prop: GKLO relations}
For any $i\neq j$ we have the relations
\begin{align*}
[A_i(u), A_j(v)] &= 0, \\
[A_i(u), B_j(v)] &= [A_i(u), C_j(v)] = 0, \\
[B_i(u), B_i(v)] &= [C_i(u), C_i(v)] = 0,\\
[B_i(u), C_j(v)] &= 0,\\
(u-v) [A_i(u), B_i(v)] &= B_i(u) A_i(v) - B_i(v) A_i(u), \\
(u-v) [A_i(u), C_i(v)] &= A_i(u) C_i(v) - A_i(v) C_i(v), \\
(u-v)[B_i(u), C_i(v)] &= A_i(u) D_i(v) - A_i(v) D_i(u), \\
(u-v)[B_i(u), D_i(v)] &= B_i(u) D_i(v) - B_i(v) D_i(u), \\
(u-v)[C_i(u), D_i(v)] &= B_i(u) C_i(v) - B_i(v) C_i(u), \\
(u-v)[A_i(u), D_i(v)] &= B_i(u) C_i(v) - B_i(v) C_i(u)
\end{align*}
\end{Proposition}
\begin{proof}
We note that the series considered here differ from those in \cite{GKLO} by multiplication by constant series.  Indeed, there exist unique $s_i(u) \in 1+ u^{-1}\C[[u^{-1}]]$ such that
$$ r_i(u) = \frac{s_i(u) s_i(u-1)}{\prod_{j \con i} s_j(u-\tfrac{1}{2})}$$
with $r_i(u)$ as defined in (\ref{eq:rfromc}).  Then the series $s_i(u)^{-1} X_i(u)$ for $X = A, B, C, D$ are exactly those defined in \cite{GKLO}.  With this observation, the claim is immediate from Proposition 2.1 in \cite{GKLO}.
\end{proof}

\begin{Corollary}
\label{Cor: GKLO relations}
For any $i\in I$ we have
\begin{align*}
(u-v)[A_i(u), E_i(v)] &= A_i(u) \left( E_i(u) - E_i(v)\right), \\
(u-v)[A_i(u), F_i(v)] &= \left(F_i(v) - F_i(u)\right) A_i(u)
\end{align*}
\end{Corollary}

In \cite{KWWY}, we related these series to functions on the group $G_1[[t^{-1}]]$:
\begin{Theorem}[Theorem 3.9, \cite{KWWY}] \label{Thm:KWWY lifts}
There is an isomorphism of Poisson algebras $\gr Y \rightarrow \O(G_1[[t^{-1}]])$, which sends
\begin{align*}
A_i(u) &\mapsto \Delta_{v_{\varpi_i},v_{\varpi_i}}(u) \\
B_i(u) &\mapsto \Delta_{f_i v_{\varpi_i}, v_{\varpi_i}}(u) \\
C_i(u) &\mapsto \Delta_{v_{\varpi_i}, f_i v_{\varpi_i}}(u)
\end{align*}
\end{Theorem}

Our goal for this section is to extend the lifts provided by this theorem to all generalized minors $\Delta_{\beta,\gamma}(u)$, via an analog of Corollary \ref{Cor:cyclic}. We recall that there is an embedding $U\g \hookrightarrow Y$ defined by
$$e_i \mapsto E_i^{(1)} = B_i^{(1)}, \quad f_i \mapsto F_i^{(1)} = C_i^{(1)} $$

\begin{Proposition} \label{cyclic}
The cyclic $U\g$-module generated by the vector $A_i(u) \in Y[[u^{-1}]]$ under the adjoint action is isomorphic to $V(\varpi_i^\ast) \otimes V(\varpi_i)$, via the map extending $A_i(u) \mapsto v_{-\varpi_i} \otimes v_{\varpi_i}$.
\end{Proposition}

\begin{Corollary}  \label{unique lifts}
There are unique lifts $T_{\beta, \gamma}(u) \in Y[[u^{-1}]]$ corresponding to the minors $\Delta_{\beta, \gamma}(u)$, for $\beta,\gamma \in V(\varpi_i)$, satisfying $T_{v_{\varpi_i},v_{\varpi_i}}(u) = A_i(u)$ and
\begin{align*}
[B_j^{(1)}, T_{\beta, \gamma}(u) ] = T_{f_j \beta,\gamma}(u) - T_{\beta, e_j\gamma}(u) \\
[C_j^{(1)}, T_{\beta, \gamma}(u)] = T_{e_j \beta, \gamma}(u) - T_{\beta, f_j\gamma}(u)
\end{align*}
for all $i, j\in I$.
\end{Corollary}

We call these elements \textbf{lifted minors}. Note that $ T_{f_i v_{\varpi_i}, v_{\varpi_i}}(u) = B_i(u) $, $ T_{v_{\varpi_i}, f_i v_{\varpi_i}}(u) = C_i(u) $, and $T_{f_i v_{\varpi_i}, f_i v_{\varpi_i}}(u) = D_i(u)$.  Additional examples appear in Section \ref{Section: lifted minor examples}.

We will denote $T_{\beta,\gamma}(u) = \sum_{r\geq 0} T_{\beta,\gamma}^{(r)} u^{-r}$.  If $ \beta $ and $ \gamma $ are both weight vectors, then the weight of $T_{\beta,\gamma}^{(r)}$ with the respect to the action of $\h\subset \g$ is $\operatorname{wt}(\gamma-\beta)$.

\begin{Remark}
In Section \ref{section: type A, quantum determinants}, we will show for $\g = \mathfrak{sl}_n$ that these lifted minors are close to quantum minors from the RTT presentation of $Y(\mathfrak{gl}_n)$.  For general $\g$, we expect that our lifted minors are related to generators in the RTT presentation of $Y$ given by Wendlandt \cite{Wend}.  More precisely, there is an embedding of $\g$--representations $V(\varpi_i) \subset V_{i,a}$, where $V_{i,a}$ is the $i$th fundamental representation of $Y$ with parameter $a$ (i.e.~root of Drinfeld polynomial). Corresponding to $V_{i,a}$ there is an RTT presentation of $Y$, and we expect our lifted minors are related to those RTT generators where both components are from the subspace $V(\varpi_i)$.
\end{Remark}

To prove the Proposition, we will use an explicit presentation of $V(\varpi_i^\ast)\otimes V(\varpi_i)$.  Consider the $U(\g)$ module
$$ N = \mathrm{Ind}_{U(\h)}^{U(\g)} \C_{\mathrm{triv}} $$
where $\C_{\mathrm{triv}} = \C v_0$ is the trivial $\h$-module.  The following result seems to be known to experts.
\begin{Lemma} There is an isomorphism of $\g$-modules
$$ V(\varpi_i^\ast)\otimes V(\varpi_i) \cong N / \left\langle e_i^2 v_0, e_j v_0, f_i^2 v_0, f_j v_0 \Big| j\neq i \right\rangle $$
extending $ v_{-\varpi_i}\otimes v_{\varpi_i} \mapsto v_0 $.
\end{Lemma}
\begin{proof}
Using the BGG resolution, we have two short exact sequences:
$$ 0 \rightarrow \text{Im} \Big( \bigoplus_{j\in I} M(s_j\cdot \varpi_i) \Big) \rightarrow M(\varpi_i) \rightarrow V(\varpi_i) \rightarrow 0 $$
$$ 0 \rightarrow \text{Im} \Big( \bigoplus_{j\in I} M_\textrm{low} (- s_j\cdot \varpi_i) \Big) \rightarrow M_\textrm{low}(-\varpi_i) \rightarrow V(\varpi_i^\ast)\rightarrow 0 $$
where $M_\textrm{low}(\lambda)$ denotes the lowest weight Verma for $U(\g)$ of weight $\lambda\in \h^\ast$.  Hence
$$ 0 \rightarrow \begin{matrix} \text{Im}\left(\bigoplus M_\textrm{low}(-s_j \cdot \varpi_i)\right) \otimes M(\varpi_i) \\ + \\ M_\textrm{low}(-\varpi_i)\otimes \text{Im}\left(\bigoplus M(s_j\cdot \varpi_i)\right) \end{matrix} \rightarrow M_\textrm{low}(-\varpi_i)\otimes M(\varpi_i) \rightarrow V(\varpi_i^\ast)\otimes V(\varpi_i) \rightarrow 0 $$
The claim follows by observing that, for any $\lambda\in\h^\ast$,
$$M_\textrm{low}(-\lambda)\otimes M(\lambda) \cong N, \ \ v_{-\lambda}\otimes v_\lambda \mapsto v_0  $$
\end{proof}
\begin{proof}[Proof of Proposition \ref{cyclic}]
By Proposition \ref{Prop: GKLO relations}, $[\h, A_i(u)] = 0$, $[B_j^{(1)}, A_i(u)] = [C_j^{(1)}, A_i(u)] = 0$ for $j \neq i$, and
$$ [B_i^{(1)}, [B_i^{(1)}, A_i(u)]] = [B_i^{(1)}, B_i(u)] = 0, \quad [C_i^{(1)}, [C_i^{(1)}, A_i(u)]] = - [C_i^{(1)}, C_i(u)] = 0 $$
Therefore by the previous Lemma, the cyclic module $U(\g)\cdot A_i(u)$ admits a surjection from $V(\varpi_i^\ast)\otimes V(\varpi_i)$.  But the dimension of $U(\g)\cdot A_i(u)$ is at least that of $U(\g)\cdot \Delta_{v_{\varpi_i},v_{\varpi_i}}(u)$.  Indeed, for any fixed $s$, the module $U(\g)\cdot A_i^{(s)}$ is contained in the filtered piece $Y_{\leq s}$ of $Y$.  The filtered pieces are $U\g$-invariant, so there is a surjective morphism $U(\g)\cdot A_i^{(s)} \rightarrow U(\g)\cdot \Delta_{v_{\varpi_i},v_{\varpi_i}}^{(s)}$ of $U(\g)$-modules defined by $x\cdot A_i^{(s)} \mapsto x\cdot \Delta_{v_{\varpi_i},v_{\varpi_i}}^{(s)}$.
\end{proof}

\subsection{Interaction with higher generators}

Recall that, starting from the elements $A_i^{(s)} \in Y_\mu$, we defined lifted minors $T_{\beta,\gamma}^{(s)} \in Y$, for any $i\in I$ and $\beta,\gamma \in V(\varpi_i)$.  These elements are defined using the action of $U(\g)$, and it is not clear what relations they will have with the higher level generators $E_i^{(s)}, F_i^{(s)}, H_i^{(s)} \in Y$.  Most important to us will be the interaction with the elements $F_i^{(s)}$ for $s>\mu_i$, since in studying $Y_\mu$ we are limited to these modes.

First some notation.  For a pair $(\bi,\bs)$ consisting of tuples $\bi = (i_1,\ldots, i_d) \in I^d$ and $\bs = (s_1,\ldots, s_d)\in \Z_{>0}^d$ of any length $d\geq 0$, we will denote
$$ F_\bi^{(\bs)} = F_{i_1}^{(s_1)} \cdots F_{i_d}^{(s_d)} \in Y $$
Given a second pair $(\bi', \bs')$ with $\bi' = (i'_1,\ldots, i'_{d'})$ and  $\bs' = (s'_1,\ldots, s'_{d'})$, we will write $(\bi',\bs') < (\bi,\bs)$ if $d' < d$ and  there exists an increasing sequence $1\leq a_1 < \ldots < a_{d'} \leq d$, such that
$$ i'_\ell = i_{a_\ell}, \quad s'_\ell \geq s_{a_\ell}, \qquad \forall 1\leq \ell \leq d' $$
In other words, $F_{\bi'}^{(\bs')}$ is essentially a subword of $F_\bi^{(\bs)}$, but with some exponents possibly increased.

Let $(\bi, \bs)$ be given.  We will denote
$$f_\bi = f_{i_d} \cdots f_{i_1} \in U(\g)$$
Note that the order is reversed.  We will also denote the analogous composition of Kashiwara operators $\tilde{f}_\bi = \tilde{f}_{i_d}\cdots \tilde{f}_{i_1}$. For any $i\in I$ consider the monomial crystal $\B(\varpi_i,0)$, see Section \ref{Tsection}.  If $\tilde{f}_\bi(y_{i,0}) \neq 0$, define $ k_1, \dots, k_d $ by
$$ y_{i,0} \stackrel{\tilde{f}_{i_1}}{\longrightarrow} z_{i_1, k_1-2}^{-1} y_{i,0} \stackrel{\tilde{f}_{i_2}}{\longrightarrow} z_{i_2, k_2-2}^{-1} z_{i_1, k_1-2}^{-1} y_{i,0} \stackrel{\tilde{f}_{i_3}}{\longrightarrow} \cdots \stackrel{\tilde{f}_{i_d}}{\longrightarrow} z_{i_d, k_d-2}^{-1}\cdots y_{i_1,k_1-2}^{-1} y_{i,0} $$

\begin{Definition}
\label{def: G poly}
Let $G_{i, \bi}^\bs(u) \in \C[u]$ be the polynomial
$$ G_{i, \bi}^\bs(u) = \left\{\begin{matrix} u^{m_i} (u+\tfrac{1}{2} k_1)^{s_1-1}\cdots (u+\tfrac{1}{2} k_d)^{s_d -1}, & \text{ if } \tilde{f}_{\bi}(y_{i,0}) \neq 0 , \\  0, & \text{ if } \tilde{f}_{\bi}(y_{i,0}) = 0. \end{matrix} \right.$$
for $k_1,\ldots, k_d$ as above.
\end{Definition}

For any formal series $X(u) = \sum_{s\in\Z}X^{(s)} u^{-s} \in Y[[u,u^{-1}]]$, let us denote the principal part by $\underline{X(u)} = \sum_{s>0} X^{(s)} u^{-s}$.  The following basic fact will be useful later.

\begin{Lemma} \label{Lemma: properties of principal part}
Suppose $X(u) \in Y[[u, u^{-1}]]$.  For any polynomial $f(u)\in \C[u]$, the coefficients of $\underline{ f(u) X(u)}$ are linear combinations of the coefficients of $\underline{X(u)}$.
\end{Lemma}

We will also denote
$$ Y\cdot X(u) = \operatorname{span}_\C \left\{ y X^{(s)} : y\in Y, s\in \Z \right\}[[u^{-1}]]$$
In other words, $Y\cdot X(u) = L[[u^{-1}]] $, where $L\subset Y$ is the left ideal generated by the coefficients of $X(u)$.

Recall from Corollary \ref{unique lifts} that we have 
$$ [\cdots [ [A_i(u), F_{i_1}^{(1)}], F_{i_2}^{(1)}],\cdots, F_{i_d}^{(1)}] = T_{v_{\varpi_i}, f_{i_d}\cdots f_{i_1}v_{\varpi_i}}(u) = T_{v_{\varpi_i}, f_{\bi} v_{\varpi_i}}(u) $$
This leads to the formula
$$
A_i(u) F_\bi^{(1, \dots, 1)} = T_{v_{\varpi_i}, f_\bi v_{\varpi_i}}(u) + \sum_{(\bi',\bs') < (\bi, \bs)} Y\cdot  T_{v_{\varpi_i}, f_{\bi'} v_{\varpi_i}}(u)
$$
We will need to generalize this formula, where $ (1, \dots, 1) $ is replaced by an arbitrary sequence $ \bs $:

\begin{Proposition}
\label{Prop: higher F relations}
Suppose that $\varpi_i$ is a minuscule coweight.  Then for any $(\bi, \bs)$, we have the equality
\begin{equation}
\label{Equation: higher F relations}
\underline{u^{m_i} A_i(u)} F_\bi^{(\bs)} = \underline{G_{i,\bi}^\bs(u) T_{v_{\varpi_i}, f_\bi v_{\varpi_i}}(u) } + \sum_{(\bi',\bs') < (\bi, \bs)} Y\cdot \underline{G_{i,\bi'}^{\bs'}(u) T_{v_{\varpi_i}, f_{\bi'} v_{\varpi_i}}(u)},
\end{equation}
as cosets in $Y[[u^{-1}]]$.
\end{Proposition}
(In this proposition, we include the case where $ f_\bi v_{\varpi_i} = 0 $.  In this case, $ f_\bi v_{\varpi_i} = 0 $, so even though $ G_{i,\bi}^\bs(u) $ is not defined, by convention we have $ G_{i,\bi}^\bs(u) T_{v_{\varpi_i}, f_\bi v_{\varpi_i}}(u) = 0 $.)

Note that Proposition \ref{Prop: higher F relations} is equivalent to a similar statement where equation (\ref{Equation: higher F relations}) is replaced by
\begin{equation}
\label{Equation: higher F relations 2}
\underline{G_{i,\bi}^\bs(u) T_{v_{\varpi_i}, f_\bi v_{\varpi_i}}(u)}= \underline{u^{m_i} A_i(u)} F_\bi^{(\bs)} + \sum_{(\bi',\bs') < (\bi,\bs)} Y\cdot \underline{u^{m_i} A_i(u)} F_{\bi'}^{(\bs')}
\end{equation}
This follows from upper-triangularity and induction on the length $d$ of $(\bi, \bs)$.

\begin{proof}[Proof of the proposition]
The proof will proceed by induction on $d$, which is the length of $\bi$ and $\bs$.

\underline{(i) The case $d=1$}: Using Corollary \ref{Cor: GKLO relations}, for any $s>0$ we have the relation
$$ [A_i(u), F_i^{(s+1)}] = \Big( \sum_{r>0} F_i^{(s+r)} u^{-r}\Big) A_i(u) $$
Now, by Corollary \ref{unique lifts}, $T_{v_{\varpi_i},v_{\varpi_i}}(u) = A_i(u)$ and $T_{v_{\varpi_i}, f_iv_{\varpi_i}}(u) = [A_i(u), F_i^{(1)}] = F_i(u) A_i(u)$.  Using this, rewrite the above relation:
$$ [A_i(u), F_i^{(s+1)}] = u^s  T_{v_{\varpi_i}, f_i v_{\varpi_i}}(u) - (F_i^{(1)} u^{s-1} + F_i^{(2)} u^{s-2} + \ldots + F_i^{(s)}) T_{v_{\varpi_i},v_{\varpi_i}}(u) $$
Rearranging terms, multiplying by $u^{m_i}$, and taking principal parts, it follows that
\begin{align*}
 \underline{ u^{m_i} A_i(u)} F_i^{(s+1)} &= \underline{u^{s+m_i} T_{v_{\varpi_i}, f_i v_{\varpi_i}}(u) } -  \underline{( F_i^{(1)} u^{s-1} + \ldots + F_i^{(s)} -  F_i^{(s+1)}) u^{m_i} T_{v_{\varpi_i}, v_{\varpi_i}}(u) } = \\
& = \underline{u^{s+m_i} T_{v_{\varpi_i}, f_iv_{\varpi_i}}(u)} + Y\cdot \underline{u^{m_i} T_{v_{\varpi_i},v_{\varpi_i}}(u)}
\end{align*}
Since $\tilde{f}_i(y_{i,0}) = z_{i,-2}^{-1} y_{i,0}$ in $\B(\varpi_i,0)$, so for $\bi=(i)$ and $\bs = (s+1)$ we have $G_{i,\bi}^{\bs}(u) = u^{m_i+s}$.  Hence the claim holds in the case $\bi = (i)$.  If $\bi = (j)$ where $j\neq i$, then $[A_i(u), F_j^{(s+1)}] = 0$ and $f_j v_{\varpi_i} = 0$, so the claim also holds.

\underline{(ii) The inductive step in $d$}: Assume the claim holds for all sequences $(\bi', \bs')$ of length $ \leq d$.  We will prove the claim for $(\mathbf{j}, \mathbf{r})$ where $\mathbf{j} = (i_1,\ldots, i_d, j)$, $\mathbf{r} = (s_1,\ldots, s_d, r)$, by induction on $r$. We continue to denote $\bi = (i_1,\ldots,i_d)$ and $\bs = (s_1,\ldots, s_d)$.

When $r=1$, we multiply equation (\ref{Equation: higher F relations}) on the right by $ F_j^{(1)} $ to give
$$
\underline{u^{m_i} A_i(u)} F_\mathbf{j}^{(\mathbf{r})} = \underline{G_{i,\bi}^\bs(u) T_{v_{\varpi_i}, f_\bi v_{\varpi_i}}(u)F_j^{(1)} } + \sum_{(\bi',\bs') < (\bi, \bs)} Y\cdot \underline{G_{i,\bi'}^{\bs'}(u) T_{v_{\varpi_i}, f_{\bi'} v_{\varpi_i}}(u) F_j^{(1)}}
$$
By Corollary \ref{unique lifts}, $$T_{v_{\varpi_i}, f_\bi v_{\varpi_i}}(u) F_j^{(1)} = T_{v_{\varpi_i}, f_{\mathbf{j}} v_{\varpi_i} }(u) + F_j^{(1)} T_{v_{\varpi_i}, f_\bi v_{\varpi_i} }(u)$$  
Also, $G_{i, \mathbf j}^{\mathbf{r}} (u) = G_{i, \bi}^{\bs}(u)$ since $r=1$.  Similar formulas hold for all summands $(\bi', \bs')$, proving the claim in the $r=1$ case.

For the inductive step in $r$, we consider two cases.

\underline{(a) The case $f_{\mathbf{j}} v_{\varpi_i} = 0$}:  In this case, we prove a stronger version of equation (\ref{Equation: higher F relations}):
\begin{equation} \label{eq: higher F relations3}
\underline{ u^{m_i} A_i(u)} F_\bi^{(\bs)} F_j^{(r)} = \sum_{(\bi',\bs') < (\bi,\bs)} Y\cdot \underline{G_{i,\bi'}^{\bs'}(u) T_{v_{\varpi_i}, f_{\bi'}v_{\varpi_i}}(u)} 
\end{equation}
Again the case $r=1$ holds using Corollary \ref{unique lifts}.  To prove the inductive step, recall the element $S_i$ from Lemma \ref{Lemma: coideal}, which satisfies $[S_i, F_j^{(r)}] = -a_{ij} F_j^{(r+1)}$.  Assuming the above formula holds up to $r$, we wish to conclude the case $r+1$.  Bracketing both sides of the formula for $r$ by $S_j$, we get
$$
\underline{u^{m_i} A_i(u)} [S_j, F_\bi^{(\bs)}] F_j^{(r)} - 2 \underline{u^{m_i} A_i(u)} F_\bi^{(\bs)} F_j^{(r+1)} =  \sum_{(\bi',\bs') < (\bi,\bs)} Y\cdot [S_j, \underline{ G_{i,\bi'}^{\bs'}(u) T_{v_{\varpi_i}, f_{\bi'}v_{\varpi_i}}(u)}]
$$
since $[S_j, A_i(u)] = 0$.  The second term on the left is the one we want to express.   Since $[S_j, F_\bi^{(\bs)}]$ is a linear combination of terms $F_{\bi}^{(\bs')}$ where $s'_\ell = s_\ell +1$ for some $\ell$, the inductive assumption applies to the first term on the left hand side.  For the right hand side we claim that
$$
[S_j, \underline{ G_{i,\bi'}^{\bs'}(u) T_{v_{\varpi_i}, f_{\bi'}v_{\varpi_i}}(u)}] = \sum_{(\bi'',\bs'')<(\bi',\bs')} Y\cdot \underline{ G_{i,\bi''}^{\bs''}(u) T_{v_{\varpi_i}, f_{\bi''} v_{\varpi_i} }(u)}
$$
Indeed, this follows by the inductive assumption on $d$ and equation (\ref{Equation: higher F relations 2}), since the action of $[S_j, \cdot]$ on products $F_{\bi''}^{(\bs'')}$ shifts exponents.

\underline{(b) The case $f_{\mathbf{j}} v_{\varpi_i} \ne 0$}:
Choose $p$ maximal such that $i_p = j$ or $i_p \con j$.  Thus
$$ F_{\mathbf j}^{(\mathbf r)} = F_\bi^{(\bs)} F_j^{(r)} = F_{i_1}^{(s_1)} \cdots F_{i_p}^{(s_p)} F_j^{(r)} F_{i_{p+1}}^{(s_{p+1})}\cdots F_{i_d}^{(s_d)}$$

Suppose that $ i_p = j $.  Since $\varpi_i$ is miniscule and  $f_{i_p}\cdots f_{i_1} v_{\varpi_i}\neq 0$, then $f_j f_{i_p}\cdots f_{i_1}v_{\varpi_i} = 0$, putting us back in case (a). Thus we may suppose that $ i_p \con j $.

We will use the following identity, which follows from Definition \ref{Definition: Yangian}:
\begin{equation}\label{eq: ipconj}
[F_{i_p}^{(s)}, F_j^{(r)}] = [F_{i_p}^{(s+r-1)}, F_j^{(1)}] - \tfrac{1}{2} \sum_{n=1}^{r-1} (F_{i_p}^{(s+r-n-1)} F_j^{(n)} + F_j^{(n)} F_{i_p}^{(s+r-n-1)} )
\end{equation}
Now, by the inductive assumption on $d$,
$$ \underline{ u^{m_i} A_i(u)} F_{i_1}^{(s_1)}\cdots F_{i_{p-1}}^{(s_{p-1})} = \underline{ G_{i, (i_1,\ldots,i_{p-1})}^{(s_1,\ldots,s_{p-1})}(u) T_{v_{\varpi_i}, f_{i_{p-1}}\cdots f_{i_1} v_{\varpi_i}}(u) } + \ldots $$
Since $f_j f_{i_p}f_{i_{p-1}}\cdots f_{i_1} v_{\varpi_i} \ne 0$ and $\varpi_i$ is miniscule, we must have
\begin{equation} \label{eq: ipconj 2}
f_{i_p} f_j f_{i_{p-1}}\cdots f_{i_1} v_{\varpi_i} = 0
\end{equation}
Consider equation (\ref{eq: ipconj}) for  $s=s_p$, and left-multiply both sides by $\underline{u^{m_i} A_i(u)} F_{i_1}^{(s_1)} \cdots F_{i_{p-1}}^{(s_{p-1})}$.  Then by (\ref{eq: ipconj 2}), case $(a)$ applies to all summands corresponding to terms $F_j^{(a)} F_{i_p}^{(b)}$ from (\ref{eq: ipconj}); we may replace these summands using the stronger form (\ref{eq: higher F relations3}). Therefore, modulo lower terms $ \underline{ u^{m_i} A_i(u)} F_{\mathbf{j}}^{(\mathbf{r})}$ is equal to
$$  \underline{ u^{m_i} A_i(u)} F_{i_1}^{(s_1)}\cdots F_{i_{p-1}}^{(s_{p-1})}\Big(F_{i_p}^{(s_p+r-1)} F_j^{(1)} -\tfrac{1}{2} \sum_{n=1}^{r-1} F_{i_p}^{(s_p+r-n-1)} F_j^{(n)} \Big) F_{i_{p+1}}^{(s_{p+1})} \cdots F_{i_d}^{(s_d)} $$
By the definition of $p$, we can commute all factors $F_j^{(n)}$ to the far right.  After doing so, we apply the inductive assumption, valid since all exponents $n< r$.  Modulo lower terms, we get the principal part of

$$  \Big( (u+\tfrac{1}{2}k_p)^{s_p+r-2} - \tfrac{1}{2}\sum_{n=1}^{r-1} (u+\tfrac{1}{2}k_p)^{s_p+r-n-2} (u+\tfrac{1}{2}k)^{n-1}\Big) \prod_{q\neq p} (u+\tfrac{1}{2} k_q)^{s_q-1} T_{v_{\varpi_i}, f_{\mathbf{j}}v_{\varpi_i}}(u)  $$
where $k $ is defined by $\tilde{f}_j \tilde{f}_{i_d}\cdots \tilde{f}_{i_1} (y_{i,0}) = z_{j,k-2}^{-1} z_{i_d,k_d-2}^{-1}\cdots z_{i_1, k_1-2}^{-1} y_{i,0}$.  Now, by the definition of $p$ it follows that
$$ k = k_p - 1 $$
By the $v= u +\tfrac{1}{2}k_p$, $c = \tfrac{1}{2}$ case of the polynomial identity
$$ (v-c)^{r-1} = v^{r-1} - c \sum_{n=1}^r v^{r-n-1}(v-c)^{n-1}, $$
the above is equal to the principal part of
$$ (u + \tfrac{1}{2}k_p - \tfrac{1}{2})^{r-1} \prod_q (u+\tfrac{1}{2} k_q)^{s_q -1} T_{v_{\varpi_i}, f_{\mathbf{j}} v_{\varpi_i}}(u),$$
This is precisely $\underline{G_{i, \mathbf j}^{\mathbf s}(u) T_{v_{\varpi_i}, f_{\mathbf j} v_{\varpi_i}} (u) }$, proving the claim.
\end{proof}

\begin{Lemma}
\label{lemma: G poly ind}
Assume that $(\bi, \bs)$ is such that $i_a = i_b$ implies $s_a = s_b$, for any $a,b$.  Then $G_{i,\bi}^\bs(u)$ only depends on $\tilde{f}_\bi(y_{i,0})$, i.e. it is independent of the path taken from $y_{i,0}$ to $\tilde{f}_\bi(y_{i,0})$ in the crystal.
\end{Lemma}
\begin{proof}
Write $\tilde{f}_\bi(y_{i,0}) = y_{i,0} z_{\mathbf{U}}^{-1}$.  Such a decomposition is unique, c.f. Remark \ref{Remark: uniqueness of monomial factorization}.  From the assumption on $\mathbf{s}$, it is easy to see that $G_{i,\bi}^\bs$ only depends on $\mathbf{U}$:
$$ G_{i, \bi}^\bs(u) = u^{m_i} \prod_{(j,k-2)\in \mathbf{U}} (u+\tfrac{1}{2} k)^{s_j} $$
where $s_j$ is defined to be $s_a$ for any $i_a =j$.
\end{proof}

\begin{Definition}
\label{def: G poly 2}
Let $\varpi_i$ be  miniscule and $\gamma \in V(\varpi_i)$ a nonzero weight vector.  Write $\gamma = c f_\bi v_{\varpi_i}$ for some $\bi = (i_1,\ldots,i_d)$ and $c\in \C^\ast$, and set
$$ \bs = (\mu_{i_1} +1, \mu_{i_2} +1,\ldots, \mu_{i_d}+1 ) $$
We define $ G_\gamma(u) = G_{i, \bi}^\bs(u) $.  This definition is independent of $\bi$, by the previous Lemma.
\end{Definition}

Note that if $ \mu = 0$, then $ G_\gamma(u) = u^{m_i} $.  In general, $ G_\gamma(u) $ is a polynomial of degree $ \langle \lambda, \varpi_i \rangle - \langle \mu, \gamma \rangle $.

\begin{Lemma}
\label{Lemma: G poly divides}
Let $\varpi_i$ be miniscule.  Suppose that $\gamma'$ lies above $\gamma$ in $V(\varpi_i)$, i.e. $\gamma'\neq 0$ is proportional to $e_{j_\ell}\cdots e_{j_1} \gamma$ for some $i_1,\ldots,i_d$.  Then $G_{\gamma'}(u)$ divides $G_\gamma(u)$.
\end{Lemma}
\begin{proof}
Proportionality does not affect the definition of $G_\gamma(u)$, so we may assume all scalars are 1.  We can write $\gamma' = f_{j_d}\cdots f_{j_{\ell+1}} v_i$ for some $d> \ell$ and indices $j_{\ell+1},\ldots,j_d$, and hence
$$ \gamma =  f_{j_1}\cdots f_{j_\ell} \gamma' = f_{j_1}\cdots f_{j_\ell} f_{j_d}\cdots f_{j_{\ell+1}} v_i $$
Indeed if $e_j v \neq 0$ for a weight vector $v\in V(\varpi_i)$, then $f_j e_j v = v$ since $\varpi_i$ is miniscule.  So $\gamma = f_\bi(v_i)$ for $\bi = (i_1,\ldots,i_d) := (j_{\ell+1},\ldots, j_d, j_\ell,\ldots, j_1) $, and $G_\gamma(u) = G_{i,\bi}^\bs(u)$ with $\bs$ as defined above.  Meanwhile, $G_{\gamma'}(u) = G_{i, \bi'}^{\bs'}(u)$ where $\bi' = (i_1,\ldots, i_{d-\ell})$ and $\bs' = (s_1,\ldots, s_{d-\ell})$, so $G_{\gamma'}(u)$ divides $G_\gamma(u)$.
\end{proof}

\subsection{Partial description of the left ideal}
We now are in a position to give a partial description of $L_\mu^\lambda$.

\begin{Definition}
We define
\begin{enumerate}
\item $ S_\mu^\lambda$ to be the set of all coefficients of all series $\underline{ G_\gamma(u) T_{\beta, \gamma}(u)}$, for $i\in I$ miniscule and $\beta,\gamma\in V(\varpi_i)$ weight vectors,
\item $(S_\mu^\lambda)^\leq$ to be the set of all coefficients of all series $\underline{G_\gamma(u) T_{v_{\varpi_i},\gamma}(u)}$, for $i\in I$ miniscule and $\gamma\in V(\varpi_i)$ a weight vector.
\end{enumerate}
\end{Definition}

From equations (\ref{Equation: higher F relations}) and (\ref{Equation: higher F relations 2}) together with Lemma \ref{Lemma: properties of principal part}, we deduce:

\begin{Corollary}
\label{Cor: right ideal Borel}
There is an equality of $(Y,Y_\mu^\leq)$-bimodules:
$$ Y \{ A_i^{(s)}: \varpi_i \text{ miniscule}, s>m_i\} Y_\mu^{\leq} = Y (S_\mu^\lambda)^\leq $$
\end{Corollary}

\begin{Conjecture}
\label{Conj: L is A-invariant}
The subspace $Y S_\mu^\lambda$ is invariant under the right action of $\C[A_i^{(r)}]$.
\end{Conjecture}

\begin{Theorem}
\label{Thm: L = YS}
If Conjecture \ref{Conj: L is A-invariant} holds, then
$$ Y \{A_i^{(s)}: \varpi_i \text{ miniscule}, s>m_i\} Y_\mu = Y S_\mu^\lambda $$
\end{Theorem}
\begin{proof}
By Corollary \ref{unique lifts}, we have
$$ [E_i^{(1)}, \underline{G_\gamma(u) T_{\beta, \gamma}(u)} ] = \underline{G_\gamma(u) T_{f_i \beta,\gamma}(u)} - \underline{G_\gamma(u) T_{\beta,e_i \gamma}(u)} $$
Note that $e_i \gamma$ lies above $\gamma$ in the sense of Lemma \ref{Lemma: G poly divides}, hence $G_{e_i \gamma}(u)$ divides $G_\gamma(u)$.  By Lemma \ref{Lemma: properties of principal part} the coefficients of
$$\underline{G_\gamma(u) T_{\beta,e_i \gamma}(u) }$$
are linear combinations of the coefficients of $\underline{G_{e_i \gamma}(u) T_{\beta, e_i \gamma}(u)}$. It follows that $Y S_\mu^\lambda$ is invariant under right multiplication by the elements $E_i^{(1)}$.  The higher modes $E_i^{(r)}$ are obtained from $E_i^{(1)}$ by multiplying by elements of $\C[A_i^{(s)}]$, so if the conjecture holds we conclude that $Y S_\mu^\lambda$ is right invariant under $Y_\mu^\geq$.

It remains to show that $\underline{G_\gamma(u) T_{\beta,\gamma}(u)} F_k^{(s)} \in Y S_\mu^\lambda$ for all $s>\mu_k$, which we prove by induction on $\operatorname{ht}(\varpi_i - \operatorname{wt} \beta)$.  The base case when $\beta = v_{\varpi_i}$ follows from Corollary \ref{Cor: right ideal Borel}.  For the inductive step, write $\beta = \sum_j f_j \beta_j$.  Then by Corollary \ref{unique lifts},
$$ T_{\beta, \gamma} (u) = \sum_j \left( [E_j^{(1)}, T_{\beta_j,\gamma}(u)] + T_{\beta_j, e_j\gamma}(u) \right) $$
Therefore,
$$ T_{\beta,\gamma}(u) F_k^{(s)} = \sum_j \left( E_j^{(1)} T_{\beta_j,\gamma}(u) + T_{\beta_j,e_j \gamma}(u)\right) F_k^{(s)} + \sum_j T_{\beta_j,\gamma}(u) E_j^{(1)} F_k^{(s)}$$
Multiply both sides by $G_\gamma(u)$, and take principal parts. Recall also that $G_{e_j\gamma}(u)$ divides $G_\gamma(u)$.  For the first sum on the right-hand side, we can apply the inductive assumption since $\operatorname{ht}(\varpi_i-\operatorname{wt} \beta_j) < \operatorname{ht}(\varpi_i -\operatorname{wt}\beta)$.  For the second sum, apply the identity
$$ E_j^{(1)} F_k^{(s)} = F_k^{(s)} E_j^{(1)} + \delta_{jk} H_k^{(s)}$$
We know from the above that $Y_\mu^\geq$ preserves $Y S_\mu^\lambda$, so we are again in the position to apply the inductive assumption.
\end{proof}

If we combine Lemma \ref{lem:Compute with L} with Theorem \ref{Thm: L = YS}, then we see that the ideal of the $B$-algebra of $ Y^\lambda_\mu(\bR)$ contains $ \Pi(YS^\lambda_\mu) $.  We have equality if $ \mathfrak g $ is of type A, as we will prove below in Corollary \ref{CorollaryL=YS}.  

Now we will see that to compute $ \Pi(YS^\lambda_\mu) $ we only need ``principal'' minors.

\begin{Proposition} \label{prop:Generators of B}
  $\Pi(YS^\lambda_\mu)$ is generated as an ideal of $\Cartan$ by the coefficients of
$\Pi (\underline{ G_\gamma(u) T_{\gamma,\gamma}(u)}) $, running over all minuscule $i\in I$  and weight vectors $\gamma\in V(\varpi_i)$.
\end{Proposition}

\begin{proof}
We will first rule out several cases using a PBW basis in ``FHE'' order.  We then prove the claim by induction.

$\Pi(Y S_\mu^\lambda)$ is spanned by coefficients of series of the form $y =\Pi( x\cdot \underline{G_\gamma(u) T_{\beta, \gamma}(u)} )$ where $x\in Y$ and $\beta,\gamma \in V(\varpi_i)$ are weight vectors.  Note that $\nu = \operatorname{wt}\gamma -\operatorname{wt}\beta$ is the weight of $T_{\beta,\gamma}(u)$ with respect to $\h$, which lies in the root lattice.

If $\operatorname{ht} \nu >0$ then $T_{\beta,\gamma}^{(r)}$ contains factors $E_j^{(s)}$ on the right when written in PBW form, so $y=0$.  Assuming $\operatorname{ht} \nu \leq 0$, it suffices to consider $x = FHE$ a PBW monomial of weight $-\nu$.  If $F\neq 1$, then $HE\cdot T_{\beta,\gamma}^{(r)}$ contains factors $E_j^{(s)}$ on the right when written in PBW form, so $y = 0$.  Since $\Pi$ is equivariant with respect to left multiplication by $H$, we may assume $H = 1$.

Working with generators $B_i^{(s)}$ instead of $E_i^{(s)}$, we are reduced to showing that all coefficients of any series
$$ \Pi \left( B_{j_1}^{(s_1)} \cdots B_{j_k}^{(s_k)} \underline{G_\gamma(u) T_{\beta ,\gamma}(u)} \right) $$
lie in the ideal as claimed, where
$$\nu = \operatorname{wt}\gamma - \operatorname{wt}\beta=-\alpha_{j_1}-\ldots-\alpha_{j_k}$$
We will do this by induction on filtered degree $d =s_1+\ldots+s_k $, with an inner induction on $k= -\operatorname{ht}\nu $.  Note that when all $s_p = 1$, we may push each $B_{i_p}^{(s_p)}$ to the right using Corollary \ref{unique lifts}.  Applying $\Pi$, the result is a sum of terms $\underline{G_\gamma(u) T_{\gamma',\gamma'}(u)}$ with $\gamma'$ above $\gamma$ in $V(\varpi_i)$.  Since $G_{\gamma'}(u)$ divides $G_{\gamma}(u)$, the claim follows.

From Proposition \ref{Poisson bracket def} we have that
\begin{multline*} (u-v) \{ \Delta_{f_j v_{\varpi_j},v_{\varpi_j}}(u), \Delta_{\beta,\gamma}(v)\} = c \Delta_{\beta,\gamma}(v) \Delta_{f_j v_{\varpi_j}, v_{\varpi_j}}(u) + \Delta_{f_j \beta, \gamma}(v) \Delta_{v_{\varpi_j},v_{\varpi_j}}(u) \\
+\sum_{\alpha\in \Delta^+}\Big(\Delta_{e_\alpha \beta,\gamma}(v)\Delta_{f_\alpha f_j v_{\varpi_j}, v_{\varpi_j}}(u) - \Delta_{f_j v_{\varpi_j}, f_\alpha v_{\varpi_j}}(u)\Delta_{\beta, e_\alpha \gamma}(v)\Big)
\end{multline*}
where $ c = (s_j\varpi_j,\operatorname{wt} \beta) -(\varpi_j,\operatorname{wt} \gamma)$.  Multiplying by $u^{-1} / (1- u^{-1} v)$ and taking the coefficient of $u^{-s}$ on both sides expresses  $\{\Delta_{f_j v_{\varpi_j},v_{\varpi_j}}^{(s)}, \Delta_{\beta,\gamma}(v)\}$ as a sum, where exponents of terms coming from series in $u$ decrease in degree (i.e. $s$ decreases).  This continues to hold if we multiply both sides of the identity by $G_\gamma(v)$ and take principal parts in $v$.

Because we have a filtered deformation, the lifted version of this identity is true modulo lower filtered terms.  All summands lifting those written above lie in $Y S_\mu^\lambda$, and therefore the lower filtered terms also lie in $Y S_\mu^\lambda$.  So, the inductive hypothesis applies to these lower terms.   Note also that the first and third terms map to zero under $\Pi$, while the second and fourth terms have smaller $k$.
\end{proof}

\subsection{Type A: Connection with quantum determinants}
\label{section: type A, quantum determinants}

In the case $\g = \mathfrak{sl}_n$ we can give another description of the lifted minors of the previous section: they correspond with the quantum determinants of the RTT presentation.

To begin, we recall the definition of the Yangian $Y(\mathfrak{gl}_n)$ (see \cite{Molev}, \cite{BK}): it is the associative algebra with generators $t_{i,j}^{(r)}$ for $1\leq i, j \leq n$ and $r\geq 1$, and relations
$$ [t_{i,j}^{(r+1)}, t_{k,l}^{(s)}] - [t_{i,j}^{(r)}, t_{k,l}^{(s+1)} ] = t_{k,j}^{(r)} t_{i,l}^{(s)} - t_{k,j}^{(s)} t_{i,l}^{(r)} $$
where we define $t_{ij}^{(0)} =\delta_{i,j}$.  These relations can also be encoded in series form:
\begin{equation} \label{eq: Y(gl_n) relations}
(u-v)[t_{i,j}(u), t_{k,l}(v)] = t_{k,j}(u) t_{i,l}(v) - t_{k,j}(v) t_{i,l}(u)
\end{equation}
where $t_{i,j}(u) = \delta_{ij} + \sum_{r\geq 1} t_{i,j}^{(r)} u^{-r} $.

Given two $i$-tuples $\mathbf{a} = (a_1,\ldots,a_i)$ and $\mathbf{b} = (b_1,\ldots,b_i)$ of elements of $\{ 1,\ldots,n\}$, define the quantum determinant (see \cite{Molev}, \cite[\S 8]{BK2})
$$ Q_{\mathbf{a},\mathbf{b}}(u) = \sum_{\sigma\in S_i} (-1)^\sigma t_{a_{\sigma(1)}, b_1}(u-i+1) t_{a_{\sigma(2)},b_2}(u-i+2)\cdots t_{a_{\sigma(i)}, b_i}(u) $$
Considering the left action of the symmetric group $S_i$ on $i$-tuples, for $\pi \in S_i$ we have
$$ Q_{\pi \mathbf{a},\mathbf{b}}(u) = (-1)^\pi Q_{\mathbf{a},\mathbf{b}}(u) = Q_{\mathbf{a}, \pi \mathbf{b}}(u)$$
In particular if the elements from $\mathbf{a}$ are not pairwise distinct then $Q_{\mathbf{a},\mathbf{b}}(u) = 0$, and similarly for $\mathbf{b}$.

Therefore the definition of the quantum determinant factors through the map taking $i$-tuples to vectors in $ \bigwedge^i \C^n$,
$$\mathbf{a} = (a_1,\ldots, a_i) \longmapsto v_{a_1} \wedge \cdots \wedge v_{a_i} $$
where $ v_1, \dots, v_n $ is the usual basis for $ \C^n$.  Extending the definition by bilinearity, we define $Q_{\beta,\gamma}(u)$ for $\beta, \gamma \in\bigwedge^i \C^n$.  We equip $\bigwedge^i \C^n$ with the usual action of $gl_n$.

\begin{Lemma} \label{gl_n action on quantum determinants}
Under the action of $U(\mathfrak{gl}_n) = \langle t_{i,j}^{(1)} | 1\leq i,j \leq n \rangle \subset Y(\mathfrak{gl}_n)$, we have
$$ [t_{i,j}^{(1)}, Q_{\beta,\gamma}(u)] = Q_{e_{ij} \beta,\gamma}(u) - Q_{\beta, e_{ji}\gamma}(u) $$
\end{Lemma}
\begin{proof}
It is sufficient to consider the case where $\beta = v_{a_1}\wedge \cdots \wedge v_{a_\ell}$ and $\gamma = v_{b_1}\wedge \cdots \wedge v_{b_\ell}$ for some $\mathbf{a}$, $\mathbf{b}$.  Since $e_{ij} v_a = \delta_{aj} v_i$, formula (\ref{eq: Y(gl_n) relations}) yields
$$ [t_{i,j}^{(1)}, t_{a,b}(u)] = \delta_{aj} t_{i,b}(u) - \delta_{ib} t_{a,j}(u) $$
proving the case where $\mathbf{a} =(a)$ and $\mathbf{b} = (b)$ have length 1.

For general $\mathbf{a}$ and $\mathbf{b}$ of length $\ell$, we apply the $\ell =1$ case and the definition of $Q_{\mathbf{a}, \mathbf{b}}(u)$ to compute
\begin{align*}
[t_{i,j}^{(1)}, Q_{\mathbf{a},\mathbf{b}}(u)] &= \sum_{r=1}^\ell \sum_{\sigma\in S_\ell} (-1)^\sigma \delta_{j, a_{\sigma(r)}} t_{a_{\sigma(1)}, b_1}(u-\ell+1)\cdots t_{i, b_r}(u-\ell+r) \cdots t_{a_{\sigma(\ell)},b_\ell}(u) \\
& - \sum_{r=1}^\ell \sum_{\sigma\in S_\ell} (-1)^\sigma \delta_{i, a_r} t_{a_{\sigma(1)}, b_1}(u-\ell+1)\cdots t_{a_{\sigma(r)}, j}(u-\ell+r) \cdots t_{a_{\sigma(\ell)},b_\ell}(u)
\end{align*}
If there exists $p$ (necessarily unique) such that $j = a_p$, then the first sum can be written as
$$
\sum_{r=1}^\ell \sum_{\substack{\sigma\in S_\ell, \\ \sigma(r) = p}} (-1)^\sigma t_{a_{\sigma(1)}, b_1}(u-\ell+1)\cdots t_{i, b_r}(u-\ell+r) \cdots t_{a_{\sigma(\ell)},b_\ell}(u)
$$
which is equal to $Q_{e_{ij}\beta,\gamma}(u)$.  If no such $p$ exists, then the sum is zero as is $e_{ij} \beta$.

Simiarly, if there exists $p$ such that $i = b_p$ then the second sum can be written as
$$
\sum_{\sigma\in S_\ell}(-1)^\sigma t_{a_{\sigma(1)}, b_1}(u-\ell+1)\cdots t_{a_{\sigma(r)},j}(u-\ell+r) \cdots t_{a_{\sigma(\ell)},b_\ell}(u)
$$
which equals $Q_{\beta,e_{ji} \gamma}(u)$. If no such $p$ exists then the sum is zero and so is $e_{ji} \gamma$.
\end{proof}

To bring our notation in line with the previous sections we also make the identification $V(\varpi_i) \cong \bigwedge^i \C^n$ of $\mathfrak{sl}_n$ representations, so that the weight space corresponding to a weight in $W \varpi_i$ is identified with the span of a basis element $v_{a_1}\wedge \cdots \wedge v_{a_i}$.  For example $v_{\varpi_i}$ corresponds to $v_1\wedge\cdots \wedge v_i$, while $f_i v_{\varpi_i}$ corresponds to $v_1\wedge \cdots \wedge v_{i-1}\wedge v_{i+1} $.

The following theorem is well-known, see for example \cite[Rk. 3.1.8]{Molev}, or \cite[\S 8]{BK2}.

\begin{Theorem}
There is an embedding $\widetilde{\phi}: Y(\mathfrak{sl}_n) \rightarrow Y(\mathfrak{gl}_n)$, defined by
\begin{align*}
\widetilde{\phi}(H_i(u)) &= \frac{Q_{v_{\varpi_{i-1}},v_{\varpi_{i-1}}}(u+\tfrac{i-1}{2}) Q_{v_{\varpi_{i+1}},v_{\varpi_{i+1}}}(u+\tfrac{i+1}{2})}{Q_{v_{\varpi_i},v_{\varpi_i}}(u+\tfrac{i-1}{2}) Q_{v_{\varpi_i},v_{\varpi_i}}(u+\tfrac{i+1}{2})}, \\
\widetilde{\phi}(E_i(u)) &= Q_{f_i v_{\varpi_i},v_{\varpi_i}}(u+\tfrac{i-1}{2}) Q_{v_{\varpi_i},v_{\varpi_i}}(u+\tfrac{i-1}{2})^{-1}, \\
\widetilde{\phi}(F_i(u)) &= Q_{v_{\varpi_i},v_{\varpi_i}}(u+\tfrac{i-1}{2})^{-1} Q_{v_{\varpi_i}, f_i v_{\varpi_i}}(u+\tfrac{i-1}{2})
\end{align*}
Moreover, the center of $Y(\mathfrak{gl}_n)$ is freely generated by the coefficients of $Q_{v_{\varpi_n},v_{\varpi_n}}(u)$, and there is an isomorphism of algebras
$$ Y(\mathfrak{gl}_n) \cong Y(\mathfrak{sl}_n) \otimes Z(Y(\mathfrak{gl}_n)) $$
\end{Theorem}

Consider the trivial central character $\chi_0$ for $Y(\mathfrak{gl}_n)$, and the central quotient $Y(\mathfrak{gl}_n)/\chi_0$.  Define $\phi$ to be the composition
$$ Y(\mathfrak{sl}_n) \stackrel{\widetilde{\phi}}{\longrightarrow} Y(\mathfrak{gl}_n) \twoheadrightarrow Y(\mathfrak{gl}_n) / \chi_0 $$

Define series $s_1(u),\ldots, s_{n-1}(u) \in 1+ u^{-1} \C[[u^{-1}]]$ by the property
$$ r_i(u) = \frac{s_i(u) s_i(u-1)}{s_{i-1}(u-\tfrac{1}{2}) s_{i+1}(u-\tfrac{1}{2})}, \ \ i=1,\ldots,n-1, $$
viewing $s_0(u) = s_n(u) = 1$.  Here $ r_i(u) $ is the power series defined using $ R_i $ in equation (\ref{eq:rfromc}).  It follows from Lemma 2.1 in \cite{GKLO} that these series are uniquely defined.

\begin{Corollary} \label{Cor: type A quantum determinants}
$\phi:Y(\mathfrak{sl}_n) \rightarrow Y(\mathfrak{gl}_n) / \chi_0$ is an isomorphism.  Moreover,
$$ \phi( T_{\beta,\gamma}(u) ) = s_i(u) Q_{\beta,\gamma}(u+\tfrac{i+1}{2}) $$
for all $\beta, \gamma \in V(\varpi_i)$.
\end{Corollary}
\begin{proof}
$\phi$ is an isomorphism by the previous Theorem.  For the second claim, we first note that
$$ \phi(H_i(u)) = r_i(u) \frac{\phi(A_{i-1}(u-\tfrac{1}{2})) \phi(A_{i+1}(u-\tfrac{1}{2})) }{\phi(A_i(u)) \phi(A_i(u-1)) } $$
By the uniqueness of factorization from Lemma 2.1 in \cite{GKLO}, it follows that $\phi(A_i(u)) = s_i(u) Q_{v_{\varpi_i},v_{\varpi_i}}(u+\tfrac{i+1}{2})$.  Since $\phi(E_i^{(1)}) = t_{i+1,i}^{(1)}$ and $\phi(F_i^{(1)}) = t_{i,i+1}^{(1)}$, the claim follows by applying Corollary \ref{unique lifts} and Lemma \ref{gl_n action on quantum determinants}.
\end{proof}

\begin{Lemma} \label{Lemma: $B$-algebra Y(gln)}
$\phi$ induces an isomorphism of the $B$-algebras for $Y(\mathfrak{sl}_n)$ and $Y(\mathfrak{gl}_n)/\chi_0$:
$$\Cartan  \cong \C[t_{i,i}^{(r)} :1\leq i\leq n,r\geq 1] / \langle t_{1,1}(u-n+1)\cdots t_{n,n}(u) - 1 \rangle $$
In particular,
$$H_i(u)  \mapsto \frac{t_{i+1,i+1}(u+\tfrac{i+1}{2})}{t_{i,i}(u+\tfrac{i+1}{2})} $$
\end{Lemma}
\begin{proof}
On the lowest-mode generators, $\phi$ is the map $U(\mathfrak{sl}_n) \stackrel{\con}{\rightarrow} U(\mathfrak{gl}_n)/\langle I -1 \rangle $ defined by
$$ e_i = E_i^{(1)} \mapsto e_{i+1, i} = t_{i+1,i}^{(1)}, \ \ f_i = F_i^{(1)} \mapsto e_{i,i+1} = t_{i,i+1}^{(1)} $$
Considering the effect on $\rho^\vee$ weights, we get the above isomorphism.  The image of the $\rho^\vee$-weight zero series $Q_{\mathbf{a},\mathbf{a}}(u)$ in the $B$-algebra of $Y(\mathfrak{gl}_n)/\chi_0$ is
$$ \sum_{\sigma\in S_i} (-1)^\sigma t_{a_{\sigma(1)}, a_1}(u-i+1) t_{a_{\sigma(2)},a_2}(u-i+2)\cdots t_{a_{\sigma(i)}, a_i}(u) \mapsto t_{a_1,a_1}(u-i+1) \cdots t_{a_i,a_i}(u) $$
\end{proof}

\subsection{Type A: The $B$-algebra}

In this section, we assume that $\g = \mathfrak{sl}_n$.  In this case, we will prove Conjecture \ref{Conj: L is A-invariant}.  Since all fundamental weights are miniscule in type A, by Theorem \ref{Thm: L = YS} it will follow that
$$ L_\mu^\lambda = Y S_\mu^\lambda $$
Finally, we will use this presentation of $L_\mu^\lambda$ to give an explicit description of the $B$-algebra for $Y_\mu^\lambda(\bR)$.

We will make use of the following explicit formula, given in Example 1.15.8 in \cite{Molev}:
\begin{Lemma}
\label{Lemma: Molev formula}
In $Y(gl_n)$, for $\mathbf{a},\mathbf{b} \subset \{1,\ldots,n\}$ of size $i$ and $\mathbf{c},\mathbf{d}\subset \{1,\ldots,n\}$ of size $\ell$,
$$ [Q_{\mathbf{a},\mathbf{b}}(u), Q_{\mathbf{c},\mathbf{d}}(v)] = $$
$$ =\sum_{p=1}^{\min\{i,\ell\}} \frac{(-1)^{p-1} p!}{(u-v-i+1)\cdots (u-v-i+p)} \sum_{\substack{i_1 < \cdots <i_p \\ j_1 < \cdots < j_p}} \left( Q_{\mathbf{a}', \mathbf{b}}(u) Q_{\mathbf{c}',\mathbf{d}}(v) -  Q_{\mathbf{c},\mathbf{d}'}(v) Q_{\mathbf{a}, \mathbf{b}'}(u) \right)$$
where primes indicate that elements have been exchanged according to the indices $i_1<\cdots< i_p$ and $j_1<\cdots< j_p$:
$$\mathbf{a}' = \{a_1,\ldots, c_{j_1},\ldots, c_{j_p},\ldots, a_i\} \ \text{ and } \ \mathbf{c}' = \{ c_1,\ldots, a_{i_1},\ldots, a_{i_p},\ldots, c_\ell\}, $$
$$\mathbf{b}' = \{b_1,\ldots, d_{j_1},\ldots, d_{j_p},\ldots, b_i\} \ \text{ and } \ \mathbf{d}' = \{ d_1,\ldots, b_{i_1},\ldots, b_{i_p},\ldots, d_\ell\}$$
\end{Lemma}

\begin{Proposition} The subspace $Y S_\mu^\lambda$ is invariant under the right action of $(Y_\mu)^{\geq}$.  In particular, Conjecture \ref{Conj: L is A-invariant} holds.
\end{Proposition}
\begin{proof}
We will prove that for any weight vectors $\beta, \gamma \in V(\varpi_i)$,
$$ \underline{G_\gamma(u) T_{\beta,\gamma}(u)} T_{\delta, v_{\varpi_\ell}}^{(r)} \in Y S_\mu^\lambda $$
for all $\delta\in V(\varpi_\ell)$ and $r\geq 1$. Note that the elements $T_{\delta,v_{\varpi_\ell}}^{(r)}$ generate $Y_\mu^\geq$.

We will establish this by induction on $r$.  When $r=1$ the elements $T_{\delta,v_{\varpi_i}}^{(1)}$ lie in the Borel subalgebra $\mathfrak{b}\subset \g \subset Y$, so we can directly apply Corollary \ref{unique lifts}.  Since $\gamma$ remains fixed under this action, the claim holds.

For the inductive step, we will use Lemma \ref{Lemma: Molev formula}.  Transporting this identity under the isomorphism $\phi$ from Corollary \ref{Cor: type A quantum determinants}, we get
$$ [ T_{\beta,\gamma}(u), T_{\delta,v_{\varpi_\ell}}(v)]  =\sum_{p=1}^{\min\{i,\ell\}} \frac{(-1)^{p-1} p!}{(u-v+\tfrac{i-\ell}{2}-i+1)\cdots (u-v+\tfrac{i-\ell}{2}-i+p)} \times $$
$$ \times\sum \left( T_{\beta',\gamma}(u) T_{\delta',v_{\varpi_\ell}}(v) - T_{\delta,v_{\varpi_\ell'}}(v) T_{\beta,\gamma'}(u) \right) $$
with primes indicating exchanges as per the Lemma. Expand each of the rational functions $(-1)^{p-1} p! / \cdots$ in the domain $\C[u][[v^{-1}]]$.  Multiply both sides by $G_\gamma(u)$, and take the principal part in $u$.  We wish to extract the coefficient of $v^{-r}$ on both sides of the resulting identity.

Consider a term $T_{\beta',\gamma}(u) T_{\delta',v_{\varpi_\ell}}(v)$.  When we extract the coefficient of $v^{-r}$, since we have multiplied by an element of $\C[u][[v^{-1}]]$ on the right-hand side, we get a sum of products $\underline{G_\gamma(u) T_{\beta',\gamma}(u) } T_{\delta',\varpi_\ell}^{(s)}$ where $s< r$.  By the inductive hypothesis, these summands all lie in $Y S_\mu^\lambda$.

Consider now a term $T_{\delta,\varpi_\ell'}(v) T_{\beta,\gamma'}(u)$.  Since $\varpi_\ell$ is the highest weight, the element $\gamma'$ is a weight vector lying above $\gamma$ in $V(\varpi_i)$, in the sense of Lemma \ref{Lemma: G poly divides}.  So $G_{\gamma'}(u)$ divides $G_\gamma(u)$, and therefore these summands already lie in $Y S_\mu^\lambda$.
\end{proof}

By Theorem \ref{Thm: L = YS}, it follows that:

\begin{Corollary}
\label{CorollaryL=YS}
For $\g$ of type A, we have $L_\mu^\lambda = Y S_\mu^\lambda$.
\end{Corollary}

Applying Proposition \ref{prop:Generators of B}, we can now give an explicit description of the $B$-algebra from Section \ref{Verma modules for truncated shifted Yangians}.
\begin{Corollary}\label{B algebra}
The $B$-algebra of $ Y^\lambda_\mu(\bR) $ is the quotient of the ring $\Cartan$ by the ideal generated by the coefficients of
$\Pi (\underline{G_\gamma(u) T_{\gamma,\gamma}(u)}) $, running over all $i\in I$  and weight vectors $\gamma\in V(\varpi_i)$.
\end{Corollary}

Next we will describe precisely the images $\Pi(\underline{G_\gamma(u) T_{\gamma,\gamma}(u)})$, also providing a connection to monomial crystals.
\begin{Proposition} \label{multiset conditions}
Let $i\in I$ and suppose $\gamma = f_{i_d}\cdots f_{i_1} v_{\varpi_i} \in V(\varpi_i)$ is non-zero. Consider the corresponding element
$$
q_\gamma:=\tilde{f}_{i_d}\cdots \tilde{f}_{i_1}(y_{i,0}) = y_{i,0} z_{\mathbf{U}}^{-1}
$$
of the monomial crystal $ \B(\varpi_i, 0)$.  Then
$$\Pi\left( T_{\gamma,\gamma}(u) \right) = A_i(u) \prod_{(j,k-2) \in \mathbf{U}} H_j(u+\tfrac{1}{2}k )$$
If $q_\gamma = \prod_{j,k} y_{j,k}^{b_{j,k}}$, then
$$ \Pi(G_\gamma(u) T_{\gamma,\gamma}(u)) = \Big( \prod_{j,k} \prod_{c\in R_j} (u-\tfrac{1}{2}c+\tfrac{1}{2}k)^{U_j(k-2) }\Big)   \prod_{j,k} \left((u+\tfrac{1}{2}k)^{m_j} A_j(u+\tfrac{1}{2}k) \right)^{b_{j,k}} $$
\end{Proposition}
\begin{proof}
The proof of the first part will proceed by induction on $\operatorname{ht}(\varpi_i - \operatorname{wt} \gamma)$, by verifying that the images of both sides are equal under the isomorphism of Lemma \ref{Lemma: $B$-algebra Y(gln)}.

The case $\varpi_i = \operatorname{wt} \gamma$ is straightforward: in this case $\gamma = v_{\varpi_i}$, so $q_\gamma = y_{i,0}$, $G_\gamma(u) = u^{m_i}$, and $T_{\gamma,\gamma}(u) = A_i(u)$.

Suppose now that $\gamma $ corresponds to the the vector $v_{a_1}\wedge \cdots \wedge v_{a_i} $ under the isomorphism $V(\varpi_i) \cong \bigwedge^i \C^n$. Then $\phi(T_{\gamma,\gamma}(u)) = s_i(u) Q_{\mathbf{a}, \mathbf{a}}(u+\tfrac{i+1}{2})$ by Corollary \ref{Cor: type A quantum determinants}, and the image of $T_{\gamma,\gamma}(u)$ in the B -algebra of $Y(gl_n)/\chi_0$ is
$$ s_i(u) t_{a_1,a_1}(u+\tfrac{i+1}{2}-i+1) \cdots t_{a_r,a_r}(u+\tfrac{i+1}{2}-i+r)\cdots t_{a_i,a_i}(u+\tfrac{i+1}{2})$$
Recall that the image of $H_j(u)$ is
$$\frac{t_{j+1,j+1}(u+\tfrac{j+1}{2})}{t_{j,j}(u+\tfrac{j+1}{2})} $$

Now assume $f_j(\gamma) \neq 0$, so in particular $j = a_r$ for some $r$.  To complete the induction, it suffices to show that the Kashiwara operator $\tilde{f}_j$ acts on $\tilde{f}_{i_d}\cdots \tilde{f}_{i_1}(y_{i,0})\in \B(\varpi_i, 0)$ as multiplication by $z_{j,k-2}^{-1}$ where $k = -i -a_r +2r$.  This follows from Lemma \ref{lemma: embedding of sl_n crystals} below.

The prove the second part, use the following formula which combines (\ref{eq: def of A}) and (\ref{eq:rfromc}):
$$
H_j(u) = \frac{\prod_{c\in R_j} (u- \tfrac{1}{2}c)}{u^{\mu_j}} \frac{\prod_{\ell \con j} (u-\tfrac{1}{2})^{m_\ell} A_\ell(u-\tfrac{1}{2})}{u^{m_j} A_j(u) (u-1)^{m_j} A_j(u-1) }
$$
In the expression for $\Pi(T_{\gamma,\gamma}(u))$ from the first part, replace all $H_j(u+\tfrac{1}{2}k)$ which appear by applying this formula.  The corresponding product of factors $(u+\tfrac{1}{2}k)^{\mu_j}$ in the denominator is exactly $u^{-m_i} G_\gamma(u)$.  The product of $A$'s mimics the definition
$$z_{j,k-2}^{-1} = \frac{\prod_{\ell\con j} y_{\ell,k-1}}{y_{j,k} y_{j,k-2}} $$
Since this formula for $z_{j,k-2}^{-1}$ links the two expressions given for $q_\gamma$, the claim follows.
\end{proof}

\begin{Lemma} \label{lemma: embedding of sl_n crystals}
There is an containment of $\mathfrak{sl}_n$ monomial crystals
$$ \B(\varpi_i, 0) \subset \B(\varpi_1, -i+1) \B(\varpi_1, -i+3)\cdots \B(\varpi_1, i-1) $$
The element of $\B(\varpi_i,0)$ corresponding to the weight space $\C \cdot v_{a_1}\wedge \cdots \wedge v_{a_i}\subset \bigwedge^i \C^n$, where $a_1 < \ldots < a_i$, can be expressed as
$$ (\tilde{f}_{a_1-1}\cdots \tilde{f}_1)(y_{1,-i+1}) \cdot (\tilde{f}_{a_2-1}\cdots \tilde{f}_1)(y_{1,-i+3})\cdots (\tilde{f}_{a_i-1}\cdots \tilde{f}_1)(y_{1,i-1}) =$$
$$ = \frac{y_{a_1, -i-a_1+2}}{y_{a_1-1,-i-a_1+1}} \cdots \frac{y_{a_r, -i - a_r +2r}}{y_{a_r-1,-i-a_r+2r-1}} \cdots \frac{y_{a_i,i-a_i}}{y_{a_i-1,i-a_i-1}} $$
\end{Lemma}
\begin{proof}
There is an equality of monomials
\begin{align*}
y_{i,0} &= y_{1,-i+1} \frac{y_{2,-i+2}}{y_{1,-i+1}}\frac{y_{3,-i+3}}{y_{2,-i+2}}\cdots \frac{y_{i,0}}{y_{i-1,-1}} = \\
 &= y_{1,-i+1}\cdot \tilde{f}_1(y_{1,-i+3}) \cdot (\tilde{f}_2 \tilde{f}_1)(y_{1,-i+5}) \cdots (\tilde{f}_{i-1}\cdots \tilde{f}_1)(y_{1,i-1}),
\end{align*}
Since $y_{i,0}$ generates $\B(\varpi_i,0)$ under the $\tilde{f}_j$, this proves that there is a containment of crystals.  The second claim follows by a straightforward induction.
\end{proof}

\subsection{Example: Lifted minors for $\g = \mathfrak{sl}_3$}
\label{Section: lifted minor examples}
Let us fix dominant coweights $\lambda \geq \mu$ for $\mathfrak{sl}_3$.  In general, to effectively compute the elements $T_{\beta,\gamma}(u)$ it is helpful to use the following relations:
\begin{Lemma} For any $i,j \in I$,
\begin{enumerate}
\item $[B_i^{(1)}, H_j(u)] = - a_{ij} H_j(u) E_i(u +\tfrac{1}{2}a_{ij}) $ and $[H_j(u), C_i^{(1)}] = - a_{ij} F_i(u+\tfrac{1}{2} a_{ij}) H_j(u)$,
\item $[B_i^{(1)}, A_i(u)] =  B_i(u) = A_i(u) E_i(u) $ and $[A_j(u), C_i^{(1)}] = C_i(u) = F_i(u) A_i(u)$
\item $[E_i(u), C_i^{(1)}] = [B_i^{(1)}, F_i(u)]= H_i(u) - 1$
\end{enumerate}
\end{Lemma}

Here are several examples of lifted minors for $\g = \mathfrak{sl}_3$, written in PBW form, calculated inductively using Corollary \ref{unique lifts} and the previous Lemma:
\begin{align*}
T_{v_{\varpi_1}, v_{\varpi_1}}(u) & = A_1(u), \\
T_{f_1 v_{\varpi_1}, v_{\varpi_1}}(u) & = A_1(u) E_1(u), \\
T_{f_1 v_{\varpi_1}, f_1 v_{\varpi_1}}(u) &= H_1(u) A_1(u) + F_1(u) A_1(u) E_1(u), \\
T_{f_2f_1 v_{\varpi_1}, f_1 v_{\varpi_1}}(u) &= H_1(u) A_1(u) E_2(u-\tfrac{1}{2}) + F_1(u) A_1(u) [E_2^{(1)}, E_1(u)], \\
T_{f_2 f_1 v_{\varpi_1}, f_2 f_1 v_{\varpi_1}}(u) &= H_2(u-\tfrac{1}{2}) H_1(u) A_1(u) + F_2(u-\tfrac{1}{2}) H_1(u) A_1(u) E_2(u-\tfrac{1}{2})+ \\ &+ [F_1(u), F_2^{(1)}] A_1(u) [E_2^{(1)}, E_1(u)]
\end{align*}

Observe that the images under $\Pi$ are
\begin{align*}
\Pi(T_{v_{\varpi_1}, v_{\varpi_1}}(u)) &= A_1(u), \\
\Pi( T_{f_1 v_{\varpi_1}, f_1 v_{\varpi_1}}(u) ) &= H_1(u) A_1(u),  \\
\Pi( T_{f_2 f_1 v_{\varpi_1}, f_2 f_1 v_{\varpi_1}}(u)) &= H_2(u-\tfrac{1}{2}) H_1(u) A_1(u),
\end{align*}
in agreement with Proposition \ref{multiset conditions}, as the monomial crystal $\B(\varpi_1, 0)$ is
$$
\xymatrix{
y_{1,0} \ar[rr]^{ z_{1,-2}^{-1}}&& \frac{y_{2,-1}}{y_{1,-2}} \ar[rr]^{z_{2,-3}^{-1}} && \frac{1}{y_{2,-3}}
}
$$
From the crystal we also see that
$$G_{v_{\varpi_1}}(u) = u^{m_1}, \ \ G_{f_1 v_{\varpi_1}}(u) = u^{\mu_1 + m_1}, \ \ G_{f_2 f_1 v_{\varpi_1}}(u) = (u-\tfrac{1}{2})^{\mu_2} u^{\mu_1 + m_1} $$
By symmetry, analogous formulas hold for lifted minors for $V(\varpi_2)$.

In particular, the B-algebra for $Y_\mu^\lambda(\bR)$ is the quotient of $\C[H_i^{(s)}: i = 1,2, \ s\geq 1]$ by the principal parts of the series
\begin{align*}
u^{m_i} A_i(u) &, \\  u^{\mu_i+m_i} H_i(u) A_i(u) &= \prod_{c\in R_i}(u-\tfrac{1}{2} c)\frac{ (u-\tfrac{1}{2})^{m_j} A_j(u-\tfrac{1}{2})}{(u-1)^{m_i} A_i(u-1)}, \\ (u-\tfrac{1}{2})^{\mu_j} u^{\mu_i+m_i} H_j(u-\tfrac{1}{2}) H_i(u) A_i(u) &= \prod_{c_i\in R_i}(u-\tfrac{1}{2} c_i) \prod_{c_j \in R_j}(u-\tfrac{1}{2}c_j- \tfrac{1}{2}) \frac{1}{(u-\tfrac{3}{2})^{m_j} A_j(u-\tfrac{3}{2})}
\end{align*}
from both labellings $\{i,j\} = \{1,2\}$.

\section{Description of the product monomial crystal} \label{se:ProofConjecture}

In this section we give a combinatorial characterization of the product monomial crystal, using fundamental monomial crystals.  We then combine this characterization with the results of the previous section to prove the main conjecture in type A.

\subsection{Monomial data and regularity}

Recall the monomial crystal $\B$ from Section \ref{subsection:Definitions}.  Throughout this section, fix an integral collection of parameters $ \bR $.

\begin{Definition}
A $\bR$--{\em monomial datum} is an element $p\in \B$ of the form $p = y_\bR z_\bS^{-1}$, where $\bS = (S_i)_{i\in I}$ is an integral collections of multisets (in the sense of Section \ref{subsection:Definitions}).
\end{Definition}

We will now describe the product monomial crystal combinatorially, using conditions indexed by elements of fundamental monomial crystals $\B(\varpi_i, n) $, where $ i $ and $ n $ have opposite parity.

We fix some notation: we reserve the letter $p$ to denote monomial data
$$p = y_\bR z_\bS^{-1} = \prod_{i,k} y_{i,k}^{a_{i,k}}\in \B,$$
while we will reserve the letter $q$ to denote an element of a fundamental crystal with opposite parity condition
$$ q = y_{i,n} z_{\bU}^{-1} = \prod_{j,k} y_{j,k}^{b_{j,k}} \in \B(\varpi_i, n) $$
For a multiset $S$, let $S(k)$ denote the multiplicity of the element $k$ in $S$.
\begin{Definition}
For $p, q$ as above we set
$$ E_q(p) := \sum_{j,k} U_j(k) R_j(k+1) + \sum_{j,k} b_{j,k} S_j(k-1) $$
\end{Definition}
\begin{Definition}
\label{Def:reg}
We call a $\bR$--monomial data $p$ {\em regular} if $E_q(p)\geq 0$ for all $q\in \B(\varpi_i, n)$, where $i\in I_{\overline{n+1}}$.
\end{Definition}

\subsection{Diagrams and partitions}
\label{Sec:diagrams}

In this section we will give an alternate diagrammatic formulation of monomial data and regularity, for the case that $\g$ is of type A.

Fix a monomial datum $p = y_\bR z_\bS^{-1}$, and consider a square grid whose vertices are at the points $(i,k) \in I \times \Z$ such that $ i $ and $k $ have the same parity.  Draw $R_i(k)$ circles around the vertex $(i,k)$, and place the number $S_i(k)$ in the square whose base vertex is $(i,k)$. See Figure \ref{fig:TypeA-monom-data}.

In this notation, elements of the product monomial crystal correspond precisely to tuples of partitions: those data which can be decomposed into a union of partitions with vertices at the circled points, such that the label of each square is equal to the number of partitions containing it. See Figure \ref{fig:TypeA-monom-data-partitions}.  This decomposition need not be unique.

\begin{figure}
\begin{tikzpicture}[scale=0.5]

\draw (0.5,5.5)--(2.5,7.5);
\draw (0.5,3.5)--(4.5,7.5);
\draw (0.5,1.5)--(6.5,7.5);

\draw (1.5,0.5)--(8.5,7.5);
\draw (3.5,0.5)--(8.5,5.5);
\draw (5.5,0.5)--(8.5,3.5);
\draw (7.5,0.5)--(8.5,1.5);

\draw (2.5,0.5)--(0.5,2.5);
\draw (4.5,0.5)--(0.5,4.5);
\draw (6.5,0.5)--(0.5,6.5);
\draw (8.5,0.5)--(1.5,7.5);
\draw (8.5,2.5)--(3.5,7.5);
\draw (8.5,4.5)--(5.5,7.5);
\draw (8.5,6.5)--(7.5,7.5);

\draw (3,6) circle(0.2cm);
\draw (3,6) circle(0.3cm);
\draw (3,4) circle(0.2cm);
\draw (5,6) circle(0.2cm);
\draw (7,6) circle(0.2cm);
\draw (6,5) circle(0.2cm);
\draw (6,3) circle(0.2cm);

\draw node at (2,4) {1};
\draw node at (3,3) {2};
\draw node at (4,2) {4};

\draw node at (3,5) {2};
\draw node at (4,4) {1};
\draw node at (5,3) {3};

\draw node at (5,5) {1};
\draw node at (6,4) {2};

\draw node at (7,5) {1};

\draw node at (8,4) {1};
\draw node at (7,3) {2};

\draw node at (1,0) {\tiny 1};
\draw node at (2,0) {\tiny 2};
\draw node at (3,0) {\tiny 3};
\draw node at (4,0) {\tiny 4};
\draw node at (5,0) {\tiny 5};
\draw node at (6,0) {\tiny 6};
\draw node at (7,0) {\tiny 7};
\draw node at (8,0) {\tiny 8};

\draw node at (0,1) {\tiny 0};
\draw node at (0,2) {\tiny 1};
\draw node at (0,3) {\tiny 2};
\draw node at (0,4) {\tiny 3};
\draw node at (0,5) {\tiny 4};
\draw node at (0,6) {\tiny 5};
\draw node at (0,7) {\tiny 6};

\end{tikzpicture}

\caption{\label{fig:TypeA-monom-data}
Example monomial data for $A_8$. Here e.g. $\bR_3 = \{3, 5^2\}$ and $\bS_4 = \{0^4,2\}$.  The labels along the bottom are the nodes of the Dynkin diagram, while the integers $k$ are on the vertical axis.
}

\end{figure}
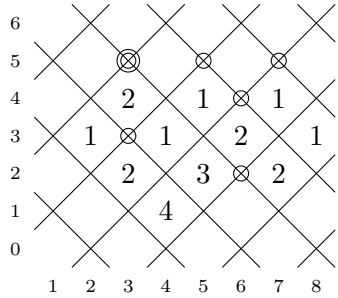

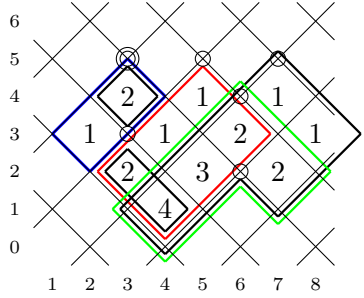
\begin{figure}
\begin{tikzpicture}[scale=0.5]

\draw[line width = 0.03cm] (3,5.8)--(3.8,5);
\draw[line width = 0.03cm] (3,5.8)--(2.2,5);
\draw[line width = 0.03cm] (3,4.2)--(3.8,5);
\draw[line width = 0.03cm] (3,4.2)--(2.2,5);

{\color{blue}
\draw[line width = 0.04cm] (3,6)--(4,5);
\draw[line width = 0.04cm] (4,5)--(2,3);
\draw[line width = 0.04cm] (2,3)--(1,4);
\draw[line width = 0.04cm] (1,4)--(3,6);
}

{\color{red}
\draw[line width = 0.03cm] (5,5.8)--(6.8,4);
\draw[line width = 0.03cm] (6.8,4)--(4,1.2);
\draw[line width = 0.03cm] (4,1.2)--(2.2,3);
\draw[line width = 0.03cm] (2.2,3)--(5,5.8);
}

\draw[line width = 0.03cm] (7,6.2)--(9.2,4);
\draw[line width = 0.03cm] (9.2,4)--(7,1.8);
\draw[line width = 0.03cm] (7,1.8)--(6,2.8);
\draw[line width = 0.03cm] (6,2.8)--(4,0.8);
\draw[line width = 0.03cm] (4,0.8)--(2.8,2);
\draw[line width = 0.03cm] (2.8,2)--(7,6.2);

\draw[line width = 0.03cm] (3,3.6)--(4.6,2);
\draw[line width = 0.03cm] (4.6,2)--(4,1.4);
\draw[line width = 0.03cm] (4,1.4)--(2.4,3);
\draw[line width = 0.03cm] (2.4,3)--(3,3.6);

{\color{green}
\draw[line width = 0.03cm] (6,5.4)--(8.4,3);
\draw[line width = 0.03cm] (8.4,3)--(7,1.6);
\draw[line width = 0.03cm] (7,1.6)--(6,2.6);
\draw[line width = 0.03cm] (6,2.6)--(4,0.6);
\draw[line width = 0.03cm] (4,0.6)--(2.6,2);
\draw[line width = 0.03cm] (2.6,2)--(6,5.4);

}

\draw (0.5,5.5)--(2.5,7.5);
\draw (0.5,3.5)--(4.5,7.5);
\draw (0.5,1.5)--(6.5,7.5);

\draw (1.5,0.5)--(8.5,7.5);
\draw (3.5,0.5)--(8.5,5.5);
\draw (5.5,0.5)--(8.5,3.5);
\draw (7.5,0.5)--(8.5,1.5);

\draw (2.5,0.5)--(0.5,2.5);
\draw (4.5,0.5)--(0.5,4.5);
\draw (6.5,0.5)--(0.5,6.5);
\draw (8.5,0.5)--(1.5,7.5);
\draw (8.5,2.5)--(3.5,7.5);
\draw (8.5,4.5)--(5.5,7.5);
\draw (8.5,6.5)--(7.5,7.5);

\draw (3,6) circle(0.2cm);
\draw (3,6) circle(0.3cm);
\draw (3,4) circle(0.2cm);
\draw (5,6) circle(0.2cm);
\draw (7,6) circle(0.2cm);
\draw (6,5) circle(0.2cm);
\draw (6,3) circle(0.2cm);

\draw node at (2,4) {1};
\draw node at (3,3) {2};
\draw node at (4,2) {4};

\draw node at (3,5) {2};
\draw node at (4,4) {1};
\draw node at (5,3) {3};

\draw node at (5,5) {1};
\draw node at (6,4) {2};

\draw node at (7,5) {1};

\draw node at (8,4) {1};
\draw node at (7,3) {2};

\draw node at (1,0) {\tiny 1};
\draw node at (2,0) {\tiny 2};
\draw node at (3,0) {\tiny 3};
\draw node at (4,0) {\tiny 4};
\draw node at (5,0) {\tiny 5};
\draw node at (6,0) {\tiny 6};
\draw node at (7,0) {\tiny 7};
\draw node at (8,0) {\tiny 8};

\draw node at (0,1) {\tiny 0};
\draw node at (0,2) {\tiny 1};
\draw node at (0,3) {\tiny 2};
\draw node at (0,4) {\tiny 3};
\draw node at (0,5) {\tiny 4};
\draw node at (0,6) {\tiny 5};
\draw node at (0,7) {\tiny 6};

\end{tikzpicture}

\caption{\label{fig:TypeA-monom-data-partitions}
This data decomposes into partitions with vertices at the circled lattice points, allowing multiplicities. This is shown by outlining the relevant partitions, and colors are just visual aids. Saying that this is a decomposition of the monomial data means that the label of each square is the number of partitions that contain it. }

\end{figure}

The number $E_q(p)$ can also be encoded diagrammatically, as in Figure \ref{fig:TypeA-condition}.  Just as elements of the product monomial crystal correspond to tuples of partitions, we can associate a single partition to each element $q \in \B(\varpi_j, n)$, satisfying the opposite parity condition $j \in I_{\overline{n+1}}$.

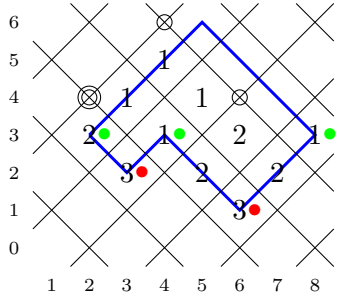
\begin{figure}
\begin{tikzpicture}[scale=0.5]

\draw[line width = 0.04cm, color=blue] (2,4)--(3,3)--(4,4)--(6,2)--(8,4)--(5,7)--(2,4);

\draw (0.5,5.5)--(2.5,7.5);
\draw (0.5,3.5)--(4.5,7.5);
\draw (0.5,1.5)--(6.5,7.5);

\draw (1.5,0.5)--(8.5,7.5);
\draw (3.5,0.5)--(8.5,5.5);
\draw (5.5,0.5)--(8.5,3.5);
\draw (7.5,0.5)--(8.5,1.5);

\draw (2.5,0.5)--(0.5,2.5);
\draw (4.5,0.5)--(0.5,4.5);
\draw (6.5,0.5)--(0.5,6.5);
\draw (8.5,0.5)--(1.5,7.5);
\draw (8.5,2.5)--(3.5,7.5);
\draw (8.5,4.5)--(5.5,7.5);
\draw (8.5,6.5)--(7.5,7.5);

\draw (4,7) circle(0.2cm);
\draw (6,5) circle(0.2cm);
\draw (2,5) circle(0.2cm);
\draw (2,5) circle(0.3cm);

\draw node at (4,6) {1};

\draw node at (3,5) {1};
\draw node at (5,5) {1};

\draw node at (2,4) {2};
\draw node at (4,4) {1};
\draw node at (6,4) {2};
\draw node at (8,4) {1};

\draw node at (3,3) {3};
\draw node at (5,3) {2};
\draw node at (7,3) {2};

\draw node at (6,2) {3};

\draw node at (1,0) {\tiny 1};
\draw node at (2,0) {\tiny 2};
\draw node at (3,0) {\tiny 3};
\draw node at (4,0) {\tiny 4};
\draw node at (5,0) {\tiny 5};
\draw node at (6,0) {\tiny 6};
\draw node at (7,0) {\tiny 7};
\draw node at (8,0) {\tiny 8};

\draw node at (0,1) {\tiny 0};
\draw node at (0,2) {\tiny 1};
\draw node at (0,3) {\tiny 2};
\draw node at (0,4) {\tiny 3};
\draw node at (0,5) {\tiny 4};
\draw node at (0,6) {\tiny 5};
\draw node at (0,7) {\tiny 6};

\draw node at (2.4,4) {{\color{green} $\bullet$}};
\draw node at (4.4,4) {{\color{green} $\bullet$}};
\draw node at (8.4,4) {{\color{green} $\bullet$}};

\draw node at (3.4,3) {{\color{red} $\bullet$}};
\draw node at (6.4,2) {{\color{red} $\bullet$}};

\end{tikzpicture}

\caption{\label{fig:TypeA-condition}
This monomial data is not in the product monomial crystal, since the condition $E_q$ shown in blue fails. That is, if you add up the circles in the partition and the numbers at the inner corners (labeled by green dots), then subtract the numbers at the outer corners (labeled by red dots), you get a negative number: $1 + (2+1+1)-(3+3)=-1.$
\newline The monomial corresponding to the data is $p =y_\bR z_\bS^{-1}$, where $R_2 = \{4^2\}, R_4 = \{6\}, R_6 = \{4\},$ and $S_2 = \{2^2\}, S_3= \{ 1^3, 3\}, S_4 = \{2,4\}, S_5 = \{1^2, 3\}, S_6 = \{0^3, 2^2\}, S_7 = \{1^2\}, S_8 = \{2\}$.
\newline The monomial corresponding to the condition is
$q = y_{5,6} z_\bU^{-1}$,where $\bU$ has $U_3 =\{2\}, U_4 = \{3\}, U_5 = \{2, 4\}, U_6 = \{1,3\}, U_7 = \{2\}$.
}
\end{figure}

\subsection{Product monomial crystal}
\label{subsection: Product monomial crystal}

In this section we will prove the following theorem, which gives our promised characterization of the product monomial crystal.
\begin{Theorem}
\label{Theorem: regular = product monomial crystal}
Consider a $\bR$--monomial datum $p = y_\bR z_\bS^{-1}$.  Then $p$ is regular if and only if it is an element of the product monomial crystal $\B(\lambda,\bR)$, where $\lambda = \sum_i |\bR_i | \varpi_i^\vee$.
\end{Theorem}

We begin with several lemmas.
\begin{Lemma} \label{lem:pqE}
Consider $p, q$ as in the previous section.  Then
\begin{enumerate}
\item $ E_q (z_{i,k}^\pm p) = E_q(p) \mp b_{i,k+1} $,
\item $ E_{z_{i,k}^\pm q} (p) = E_q(p) \mp a_{i,k+1}$.
\end{enumerate}
whenever the above are well-defined.
\end{Lemma}

For the proofs of the next two Lemmas, we use the connection of monomial crystals to quiver varieties; see Section \ref{section: Link to quiver varieties}.  We do not know purely combinatorial proofs.

\begin{Lemma}
\label{lemma: regular kashiwara}
Multiplication by $z_{i,k}$ gives a bijection of sets
$$ \begin{pmatrix} \text{Regular } \bR\text{--monomial data } p \\ \text{with } a_{i,k} <0 \text{ and } a_{i,k+2} \geq 0 \end{pmatrix} \stackrel{\con}{\longrightarrow} \begin{pmatrix} \text{Regular } \bR\text{--monomial data } p' \\ \text{with } a_{i,k}' \leq 0 \text{ and } a_{i,k+2}' >0 \end{pmatrix}$$
In particular, the set of regular $\bR$--monomial data is invariant under the Kashiwara operators $\tilde{e}_i, \tilde{f}_i$.
\end{Lemma}
\begin{proof}
Considering only the conditions on $a_{i,k} $ and $a_{i,k+2}$, this is straightforward.  We must show that the property of being regular $\bR$--monomial data is preserved in both directions.

Suppose that $p = y_\bR z_\bS^{-1}$ has $a_{i,k}<0$ and $a_{i,k+2} \geq 0$. By Lemma \ref{lem:pqE},
$$ E_q(z_{i,k}p) = E_q(p) - b_{i,k+1}. $$
If $b_{i,k+1}\leq 0$ this is non-negative because $p$ is regular. If $b_{i,k+1}>0$ then Lemma \ref{Lemma: other multiplications} applies, so $q' = z_{i,k-1}^{-b_{i,k+1}} q \in \B(\varpi_j,n)$. Then
\begin{equation} \label{eq:rwt}
0\leq E_{q'} (p) = E_q(p) + b_{i,k+1} a_{i,k} \leq  E_q(z_{i,k} p ),\end{equation}
where the last inequality is because $a_{i,k} \leq -1, b_{i,k+1} \geq 1$.  In particular, taking $q = y_{i,k+1}$ we get $ 0 \leq E_{q}(z_{i,k} p) = S_i(k)-1 $, i.e. that $S_i(k) \geq 1$.  This shows that $p' = z_{i,k}p$ is indeed $\bR$--monomial data.  Equation \eqref{eq:rwt} proves regularity.

The proof for the other direction is similar.
\end{proof}

\begin{Lemma}
\label{lemma:positive data}
Consider $\bR$--monomial data $p$.  If $a_{i,k}\geq 0$ for all pairs $(i,k)$, then
\begin{enumerate}
\item $p$ is regular,
\item $p \in \B(\lambda, \bR)$, where $\lambda = \sum_i |R_i| \varpi_i^\vee$.
\end{enumerate}
\end{Lemma}
\begin{proof}
For (1), consider first a highest weight element $q = y_{j,n+1} \in \B(\varpi_j, n)$.  Then $E_q(p) = S_j(n-1) \geq 0$.  Every other element of $q'\in \overline{B}(\varpi_j,n)$ can be reached from $q$ by multiplying by a sequence of $z_{i,k}^{-1}$, and by Lemma \ref{lem:pqE} multiplying by $z_{i,k}^{-1}$ corresponds to adding $a_{i,k+1}$ to $E_q(p)$.  Since all $a_{i,k+1}\geq 0$, we get $E_{q'}(p)\geq 0$ for all $q'$.

For (2), we define a sequence of elements $p(k)\in \B(\lambda, \bR)$ for $k_{\text{min}} \leq k\leq k_{\text{max}}$.  Here $k_{\text{max}}$ denotes the maximal value of $k$ for which some $R_i(k) > 0$, while $k_{\text{min}}$ denotes the minimal value of $k$ for which some $S_i(k) >0$.  By the construction we will have $p(k_{\text{min}}) = p$, proving the claim.

Define $p(k_{\text{max}}) = \prod_{i,k} y_{i,k}^{R_i(k)}$, so $p(k_{\text{max}})\in \B(\lambda,\bR)$ is the element of highest weight.  We define the rest of the sequence iteratively. Assuming $p(k+1)\in \B(\lambda,\bR)$ has been defined, put
$$ p(k) = \Big(\prod_i z_{i,k}^{-S_i(k)} \Big) p(k+1) $$
We claim that $p(k) \in \B(\lambda,\bR)$.   From the iterative definition, the exponent of $y_{i,k+2}$ in $p(k+1)$ is
$$ R_i(k+2) - S_i(k+2) + \sum_{j\con i} S_j(k+1) = a_{i,k+2} + S_i(k)$$
By assumption $a_{i,k+2}\geq 0$, so
$z_{i,k}^{-S_i(k)} p(k+1)\in \B(\lambda,\bR)$ by Lemma \ref{Lemma: other multiplications}.  Multiplication by $z_{i,k}^{-S_i(k)}$ does not change the exponent of $y_{j,k+2}$ for $j\neq i$, so by applying the same argument for all $j$ it follows that $p(k)\in \B(\lambda,\bR)$.

The assumption that all $a_{i,k} \geq 0$ forces $S_i(k) = 0$ for $k\geq k_{\text{max}}$.  With this in mind it is clear that $p(k_{\text{min}}) = p$, proving that $p \in \B(\lambda,\bR)$.

\end{proof}

\begin{proof}[Proof of Theorem \ref{Theorem: regular = product monomial crystal}]
Consider an equivalence relation on the set of regular $\bR$--monomial data, defined by extending the relation from Lemma \ref{lemma: regular kashiwara}: define $p$ and $p'$ to be equivalent if $p'$ can be obtained from $p$ by a series of multiplications by some $z_{i,k}$ (or $z_{i,k}^{-1}$), where at each step we had $a_{i,k}<0$ and $a_{i,k+2}\geq 0$ (resp. $a_{i,k}\leq 0$ and $a_{i,k+2}>0$).  From Corollary \ref{Cor: equiv rel on monomial crystal}, it follows that if some representative of an equivalence class is in $\B(\lambda,\bR)$, then all representatives are also in $\B(\lambda,\bR)$.

Similarly, define an equivalence relation on $\B(\lambda,\bR)$ by extending the relation from Corollary \ref{Cor: equiv rel on monomial crystal}: $p$ and $p'$ are again defined to be equivalent if $p'$ can be obtained from $p$ by a series of multiplications by $z_{i,k}^{\pm 1}$ as above.  By Lemma \ref{lemma: regular kashiwara}, if some representative of an equivalence class is regular, then all representatives are regular.

In both cases, we claim that every equivalence class contains a representative $p^+$ satisfying $a_{i,k}^+\geq 0$ for all $i,k$. Indeed, starting from $p$ a regular $\bR$--monomial data (resp. $p \in \B(\lambda,\bR)$), choose $a_{i,k}<0$ with $k$ maximal (assuming some $a_{i,k}<0$). Then $a_{i,k+2}\geq 0$, so $z_{i,k} p$ is also regular $\bR$--monomial data by Lemma \ref{lemma: regular kashiwara} (resp. $z_{i,k}p \in \B(\lambda,\bR)$ by Corollary \ref{Cor: equiv rel on monomial crystal}).  Iterating this argument, we produce an element $p^+$ in the same equivalence class as claimed.

By Lemma \ref{lemma:positive data}, such an element $p^+$ is both regular and lies in $\B(\lambda, \bR)$.  We conclude that all regular $\bR$--monomial data is in $\B(\lambda,\bR)$, and vice versa.
\end{proof}

\subsection{Monomial data and highest weights}

Fix $\lambda, \mu, \bR$.  We assume that $\lambda$ is dominant and that $\lambda \geq \mu$, but will not require $\mu$ to be dominant (although connections with the Yangian $Y_\mu^\lambda(\bR)$ and $H_\mu^\lambda(\bR)$ as defined in this paper only make sense in the dominant case). Our goal now is to relate the product monomial crystal $\B(\lambda,\bR)_\mu$ to the set of highest weights $H^\lambda_\mu(\bR)$.

We will first define a commutative algebra $\widetilde{B}_\mu^\lambda(\bR)$, and show that $\MaxSpec \widetilde{B}_\mu^\lambda(\bR)$ is in bijection with $\B(\lambda,\bR)_\mu$.  In type A, $\widetilde{B}_\mu^\lambda(\bR)$ is isomorphic to the B-algebra of $Y_\mu^\lambda(\bR)$, and we expect that this holds in other types.  Since the maximal spectrum of the B-algebra is precisely the set of highest weights $H_\mu^\lambda(\bR)$, this will prove Conjecture \ref{co:main} in type A.

To define $\widetilde{B}_\mu^\lambda(\bR)$, we will mimic the presentation for the B-algebra given in Corollary \ref{B algebra} and Proposition \ref{multiset conditions}.  Consider the commutative ring $\Cartan$.  For each element $q \in \B(\varpi_i, 0)$, write $q = y_{i,0} z_{\mathbf{U}}^{-1} = \prod_{j,k} y_{j,k}^{b_{j,k}}$, and define an element of $\Cartan((u^{-1}))$ by
\begin{equation}
\label{eq: def of H_q}
H_q(u) = \Big( \prod_{j,k} \prod_{c\in R_j} (u-\tfrac{1}{2}c+\tfrac{1}{2}k)^{U_j(k-2)} \Big)   \prod_{j,k} \big((u+\tfrac{1}{2}k)^{m_j} A_j(u+\tfrac{1}{2}k) \big)^{b_{j,k}}  
\end{equation}

\begin{Definition}
The algebra $\widetilde{B}_\mu^\lambda(\bR)$ is defined to be the quotient of $\Cartan$ by the ideal generated by all coefficients of the principal parts of all series $\underline{H_q(u)}$, over all $q\in \B(\varpi_i,0)$ and $i\in I$.
\end{Definition}
By Corollary \ref{B algebra} and Proposition \ref{multiset conditions}, we have
\begin{Corollary}
\label{cor: B-algebra in type A}
In type A, $\widetilde{B}_\mu^\lambda(\bR)$ and the B-algebra $B\big(Y_\mu^\lambda(\bR)\big)$ are equal as quotients of $\Cartan$.
\end{Corollary}

To relate $\MaxSpec \widetilde{B}_\mu^\lambda(\bR)$ with $\B(\lambda,\bR)_\mu$, we will make use of the following simple property of principal parts of rational functions:
\begin{Lemma}
\label{Lemma: properties of principal part 2}
Suppose that $X(u) \in \C((u^{-1}))$ is the expansion at $u=\infty$ of a rational function $f(u) / g(u)$, where $f, g\in \C[u]$.  Then
$$ \underline{X(u)} = 0 \ \ \Longleftrightarrow \ \ g \text{ divides } f $$
\end{Lemma}

The series corresponding to the highest weight element $y_{i,0}\in \B(\varpi_i,0)$ is $H_{y_{i,0}}(u) = u^{m_i} A_i(u)$, and so the elements $A_i^{(r)}$ with $r>m_i$ are in the ideal defining $\widetilde{B}_\mu^\lambda(\bR)$.  Because of this, we may think of $\widetilde{B}_\mu^\lambda(\bR)$ as a quotient of
$$\C[A_i^{(s)}: i\in I, 1\leq s\leq m_i] $$
Therefore a point in $\MaxSpec \widetilde{B}_\mu^\lambda(\bR)$ is equivalent to a tuple of multisets $\bS = (S_i)_{i\in I}$ of complex numbers, with $|S_i| = m_i$: such a tuple corresponds to the homomorphism $\Cartan \rightarrow \C$ given in series form by
$$ A_i(u) \mapsto \prod_{s\in S_i} (1-\tfrac{1}{2} s u^{-1}) $$

\begin{Lemma}
\label{Lemma: multiset conditions from series}
Let $q \in \B(\varpi_i,0)$, and write
$$ q = y_{i,0} z_{\mathbf{U}}^{-1} = \prod_{j,k} y_{j,k}^{b_{j,k}} $$
Consider $\bS = (S_i)_{i\in I}$ and the corresponding map $\Cartan \rightarrow \C$, as above.  Then the image of $\underline{H_q(u)}$ under this map is zero if and only if there is an inclusion of multisets of $\C$:
\begin{equation} \label{eq: multiset containment}
\bigcup_{\substack{j,k\\ b_{j,k} < 0}} (S_j - k)^{-b_{j,k}} \subset \Bigg( \bigcup_{j,k} (R_j - k)^{U_j(k-2)}\Bigg) \cup \Bigg( \bigcup_{\substack{j,k \\ b_{j,k} >0}} (S_j - k)^{b_{j,k}} \Bigg)
\end{equation}
Here, for a multiset $X$ and integer $n$, $X^n$ denotes the multiset union $\cup_{\ell =1 }^n X$. By convention, $X^0 = \emptyset$.
\end{Lemma}
\begin{proof}
By the definition of $H_q(u)$, its image under the map corresponding to $\bS$ is the rational function
$$ \Big(\prod_{j,k} \prod_{c\in R_j} (u-\tfrac{1}{2}c + \tfrac{1}{2} k)^{U_j(k-2)} \Big) \Big( \prod_{j,k} \prod_{s\in S_j} (u-\tfrac{1}{2} s+\tfrac{1}{2}k)^{b_{j,k}} \Big) $$
By Lemma \ref{Lemma: properties of principal part 2}, the principal part of this rational function is zero if and only if its denominator divides its numerator.  Since the multisets in (\ref{eq: multiset containment}) encode the roots of these polynomials, the principal part is zero if and only if (\ref{eq: multiset containment}) holds.
\end{proof}

Since $\MaxSpec \widetilde{B}_\mu^\lambda(\bR)$ is exactly the set of $\bS$ for which all $\underline{H_q(u)}$ map to zero, this Lemma is the key tool in the following result:

\begin{Theorem}
\label{thm: maxspec of conj B-algebra}
There is a bijection of sets
$$ \MaxSpec \widetilde{B}_\mu^\lambda(\bR) \longrightarrow \B(\lambda,\bR)_\mu$$
defined by $\bS \mapsto y_\bR z_\bS^{-1}$.
\end{Theorem}

\begin{proof}
Firstly, we will show that the image $p := y_\bR z_\bS^{-1}$ necessarily lands in $\B$, i.e. that all variables $y_{i,k}$ which appear have $k\in \Z$ and satisfy the parity condition $i \in I_{\overline{k}}$ as per Section \ref{subsection:Notation}.  Consider the element
$$\tilde{f}_i(y_{i,0}) = y_{i,0} z_{i,-2}^{-1} = y_{i,-2}^{-1} \prod_{j\con i} y_{j,-1} \in \B(\varpi_i,0)$$
The principal part of the corresponding series $H_{\tilde{f}_i(y_{i,0})}(u)$ must map to zero under $\bS$. By Lemma \ref{Lemma: multiset conditions from series}, this is equivalent to the inclusion
$$ S_i + 2 \subset R_i \cup \bigcup_{j\con i} (S_j + 1) $$
It is not hard to see that these containments for all $i\in I$, together with the integrality and parity conditions on $\bR$, imply the desired integrality and parity conditions on $\bS$.

Next, we will show that the image $p$ is a {\em regular} element of $\B$.  For each $n\in \Z$, there is an isomorphism of crystals $\B(\varpi_i, 0) \stackrel{\con}{\rightarrow} \B(\varpi_i, n)$ which acts by translation on the variables: $y_{j,k} \mapsto y_{j, k+n}$.   For $q \in \B(\varpi_i,0)$, denote its image by $q_n \in \B(\varpi_i, n)$.

For any fixed $q$, we claim that the corresponding inclusion of multisets (\ref{eq: multiset containment}) is equivalent to the inequalities $E_{q_n}(p) \geq 0$ for all $n$ of opposite parity to $i$.  Indeed, the integers $E_{q_n}(p)$ encode the difference in multiplicity of the number $n-1$ between the multisets appearing on the right-hand and left-hand sides of (\ref{eq: multiset containment}).  Because $\bR$ and $\bS$ satisfy the parity conditions, there is an inclusion of multisets (\ref{eq: multiset containment}) if and only if these multiplicities are non-negative.  By considering all $q \in \B(\varpi_i, 0)$, it follows that $p$ is regular.

Finally, since $p = y_\bR z_\bS^{-1}$ is regular, by Theorem \ref{Theorem: regular = product monomial crystal} we know that $p \in \B(\lambda,\bR)$.  Since $|S_i| = m_i$ for all $i$, it follows that $p \in \B(\lambda,\bR)_\mu$ as claimed.
\end{proof}

This completes the proof of Conjecture \ref{co:main} in type A.
\begin{Corollary}
\label{Cor:typeA}
For $\g$ of type A, the map $J\mapsto y(J)$ gives a bijection $H_\mu^\lambda(\bR) \cong \B(\lambda,\bR)_\mu$.
\end{Corollary}

For general $\g$, we expect that there exist series $T_{\gamma,\gamma}(u) \in Y[[u^{-1}]]$ such that the coefficients of $\underline{G_\gamma(u) T_{\gamma,\gamma}(u)}$ are in $L_\mu^\lambda$, such that in $\Cartan$ we have
$$\Pi(G_\gamma(u) T_{\gamma,\gamma}(u)) = H_\gamma(u), $$
and such that the coefficients of $\underline{H_\gamma(u)}$ generate $B(Y_\mu^\lambda(\bR))$. Exhibiting such elements would prove Conjecture \ref{co:main} in general. We can also hope to generalize Corollary \ref{cor: B-algebra in type A} outside of type A:

\begin{Conjecture}
\label{conj: B-algebra is correct in all types}
For any $\g$ of simply-laced type, $\widetilde{B}_\mu^\lambda(\bR)$ and $B\big(Y_\mu^\lambda(\bR)\big)$ are equal as quotients of $\Cartan$.
\end{Conjecture}

By a calculation similar those at the beginning of the proof of Proposition \ref{Prop: higher F relations} one can show that the series
$$ \Pi( \underline{ u^{\mu_i+ m_i} T_{f_i v_{\varpi_i}, f_i v_{\varpi_i}}(u) }) = \underline{H_{\tilde{f}_i(y_{i,0})}(u)}$$
are always in the ideal defining the B-algebra $B(Y_\mu^\lambda(\bR))$.  The next result follows from the argument at the beginning of the proof of the previous Theorem.

\begin{Proposition}
\label{prop:integrality and parity}
For general $\g$, consider a highest weight $J\in H_\mu^\lambda(\bR)$, and encode the action of $A_i(u)$ in a tuple of multisets $\bS = (S_i)_{i\in I}$.  Then these satisfy the inclusions of multisets
$$ S_i + 2 \subset R_i \cup \bigcup_{j\con i} (S_j + 1 ) $$
In particular, the map $J \mapsto y(J) = y_\bR z_\bS^{-1}$ lands in $\B$.
\end{Proposition}

\subsection{Connection to comultiplication for shifted Yangians}
\label{section: coproducts}

In this subsection we will outline an expected connection between B-algebras, the `product' structure on the product monomial crystal, and the comultiplication for shifted Yangians as defined in \cite{FKPRW}.

First consider the homomorphism $\Cartan \rightarrow \Cartan\otimes \Cartan$, defined in series form as
\begin{equation}
\label{eq: coproducts 1}
H_i(u) \mapsto H_i(u) \otimes H_i(u), \quad \text{for each } i \in I
\end{equation}

Suppose that we are given dominant coweights $\lambda = \lambda' + \lambda''$, coweights $\mu = \mu' + \mu''$ with $\lambda \geq \mu, \lambda' \geq \mu', \lambda'' \geq \mu''$, and a decomposition $\bR = \bR' \cup \bR''$ of tuples of multisets of sizes $\lambda, \lambda'$ and $\lambda''$ respectively.

\begin{Proposition}\label{prop: conj B-algebra coproduct}
\mbox{}

\begin{enumerate}
\item[(a)] The homomorphism $\Cartan \rightarrow \Cartan \otimes \Cartan$ defined by (\ref{eq: coproducts 1}) descends to the quotient algebras
\begin{equation}
\label{eq: coproducts 2}
\widetilde{B}_\mu^\lambda(\bR) \longrightarrow \widetilde{B}_{\mu'}^{\lambda'}(\bR') \otimes \widetilde{B}_{\mu''}^{\lambda''}(\bR'')
\end{equation}

\item[(b)] The corresponding map on maximal spectra (using Theorem \ref{thm: maxspec of conj B-algebra}),
$$ \B(\lambda', \bR')_{\mu'} \times \B(\lambda'', \bR'')_{\mu''} \rightarrow \B(\lambda, \bR)_\mu $$
is precisely the multiplication of monomials.
\end{enumerate}
\end{Proposition}

\begin{proof}
A monomial $p\in \B$ encodes a specialization $H_i(u) \mapsto \C[[u^{-1}]]$ for each $i\in I$, corresponding to the expansions of some rational functions determined by $p$ at $u = \infty$ (cf. Section \ref{subsec:highest}).  It is thus clear that the map (\ref{eq: coproducts 1}) corresponds to multiplication of monomials, on the level of maximal spectra.  So claim (b) follows from claim (a).

To prove claim (a), we observe that for each $q \in \B(\varpi_i, 0)$,
$$ H_q(u) \mapsto H_q'(u) \otimes H_q''(u) $$
where $H_q(u)$ (resp. $H'_q(u), H''_q(u)$) denotes the series (\ref{eq: def of H_q}) defined using data $\lambda, \mu, \bR$ (resp. $\lambda',\mu',\bR'$ and $\lambda'', \mu'', \bR''$). This is an application of the definition (\ref{eq: def of H_q}).  Hence the ideal defining $\widetilde{B}_\mu^\lambda(\bR)$ maps into the ideal defining $\widetilde{B}_{\mu'}^{\lambda'}(\bR')\otimes \widetilde{B}_{\mu''}^{\lambda''}(\bR'')$ as a quotient of $\Cartan\otimes \Cartan$, proving (a).
\end{proof}

In \cite[Section 4]{FKPRW}, a comultiplication homomorphism 
$$Y_\mu \rightarrow Y_{\mu'}\otimes Y_{\mu''}$$ 
is defined for any $\mu = \mu' + \mu''$ as above (including non-dominant cases).  This map preserves the principal gradations on both sides, and hence yields a map of their B-algebras 
$$B(Y_\mu) \rightarrow B(Y_{\mu'}) \otimes B( Y_{\mu''}).$$

\begin{Conjecture}
The comultiplication $Y_\mu \rightarrow Y_{\mu'}\otimes Y_{\mu''}$ descends to a homomorphism of quotients
$$ Y_\mu^\lambda(\bR) \rightarrow Y_{\mu'}^{\lambda'}(\bR') \otimes Y_{\mu''}^{\lambda''}(\bR'')$$
In particular, it yields a homomorphism of B-algebras
$$ B\big( Y_\mu^\lambda(\bR) \big) \rightarrow B\big( Y_{\mu'}^{\lambda'} (\bR') \big) \otimes B\big( Y_{\mu''}^{\lambda''}(\bR'') \big) $$
\end{Conjecture}

Let us assume for the moment that the above conjecture holds, and also that $\widetilde{B}_\mu^\lambda(\bR) = B\big( Y_\mu^\lambda(\bR) \big)$ as in Conjecture \ref{conj: B-algebra is correct in all types}.  Then it is easy to see that the coproduct defined by the above conjecture agrees with the coproduct defined explicitly in Proposition \ref{prop: conj B-algebra coproduct}.  Indeed, the comultiplication for $Y_\mu$ has the property that 
$$ H_i(u) \mapsto H_i(u) \otimes H_i(u) + Y_{\mu'}^< \otimes Y_{\mu''}^>,$$ 
This agrees with (\ref{eq: coproducts 1}) modulo terms which vanish  upon passing to B-algebras.  Since we have assumed that $\widetilde{B}_\mu^\lambda(\bR)$ and $B\big(Y_\mu^\lambda(\bR)\big)$ agree as quotients of $\Cartan$, this forces equality of both comultiplications.

\begin{Remark}
In the case of $\g = \mathfrak{sl}_n$ and dominant $\mu$, a proof of the previous conjecture was given in the fourth author's thesis \cite[Proposition 4.1.15]{AWthesis}.  The proof uses the connection between lifted minors and quantum determinants, as in Section \ref{section: type A, quantum determinants}.
\end{Remark}

\section{Link to quiver varieties}
\label{section: Link to quiver varieties}
In this section, we make the link with Nakajima's quiver varieties.  Throughout, fix a dominant coweight $ \lambda $ and an integral set of parameters $\bR$.

\subsection{Quiver varieties}

Consider the doubled quiver of the Dynkin diagram of $ \g $.   Recall that we have a fixed bipartition $ I = I_0 \cup I_1 $ of the vertex set $ I $.  Let $ \Omega $ denote those edges in the quiver which go from $ I_1 $ to $ I_0 $ and let $ \overline{\Omega} $ denote those edges which go from $ I_0 $ to $I _1 $.  For $ h \in \Omega \cup \overline{\Omega} $ we write $ \out(h), \inn(h) $ to denote the source and target of the edge $ h $.

Fix an $ I$-graded vector $V= \oplus_i V_i$ space with $ \dim V_i = m_i $ for all $ i \in I$.
Let
$$
\bM(V,W) = \bigoplus_{h \in \Omega \cup \overline{\Omega}} \Hom(V_{\out(h)}, V_{\inn(h)}) \oplus \bigoplus_{i \in I} \Hom(W_i, V_i) \oplus \bigoplus_{i \in I} \Hom(V_i, W_i).
$$
We will write an element of $ \bM(V,W) $ as
$$ (B, \eta, \varepsilon) =  ((B_h)_{h \in \Omega \cup \overline{\Omega}}, (\eta_i)_{i \in I}, (\varepsilon_i)_{i \in I}). $$
The group $ G_V = \prod_i GL(V_i) $ acts on $ \bM(V,W)$ by
$$
g \cdot ((B_h), (\eta_i), (\varepsilon_i)) = ((g_{\inn(h)} B_h g_{\out(h)}^{-1}), (g_i \eta_i), (\varepsilon_i g_i^{-1})).
$$
This action has moment map
$$
\mu(B, \eta, \varepsilon) = \sum_{h \in \Omega} B_h B_{\overline{h}} - \sum_{h \in \Omega} B_{\overline{h}} B_h + \sum_i \eta_i\varepsilon_i.
$$

Fix an $ I$-graded vector space $W$ with $ \dim W_i = \lambda_i $ for all $ i$.
Nakajima's quiver variety $ \fM(\bm,W) $ is the Hamiltonian reduction of $ \bM(V,W) $ by $ G_V $ at level $0$.  We work with the usual stability condition (see Definition 2.7 in \cite{Nak010}), and
$$
\fM(\bm, W) = \mu^{-1}(0)^s / G_V.
$$
Let $\displaystyle \fM(W) = \bigcup_{\bm} \fM(\bm,W) $ be the disjoint union of these quiver varieties.

\subsection{Graded quiver varieties} \label{se:gradedQV}

There is also an action of the group  $ G_W = \prod_i GL(W_i) $ on $ \bM(V,W) $, which descends to an action on $ \fM(\bm, W) $.

Fix homomorphisms $ \rho_i : \C^\times \rightarrow GL(W_i)$ so that
$ R_i $ is the set of weights for the action of $ \C^\times $ on $ W_i $.
This gives us an action of $ \C^\times $ on $ \bM(V,W) $ defined by
$$
t * ((B_h), (\eta_i), (\varepsilon_i)) = ((t B_h), ( \eta_i \rho_i(t)), (t^2 \rho_i(t)^{-1} \varepsilon_i))
$$
which descends to an action on $ \fM(\bm, W) $.
We let
$$
\fM(W, \bR) = \fM(W)^{\C^\times} \quad \text{ and } \quad  \fM(\bm,W, \bR) = \fM(\bm, W)^{\C^\times}
$$
be the fixed points for this action.  These varieties are called {\bf graded quiver varieties.}

By  \cite[\S 4.1]{Nak991}, each point $ x = [B, \eta, \varepsilon] \in \fM(\bm, W, \bR) $ determines maps  (unique up to conjugation)  $ \sigma_i: \C^\times \rightarrow GL(V_i) $,  such that $ \sigma_i(t) \cdot (B, \eta, \varepsilon) = t * (B, \eta, \varepsilon) $.

Thus a point $ x $ determines a collection $ \bS = ( S_i )_{i \in I} $
of multisets of integers with $ |S_i | = m_i$, which are the weights
of the above $\C^\times$ action on $V_i$ under $\sigma_i$.

Given such a collection $ \bS $, we let $ X(\bS) $ denote the
corresponding subset of $ \fM(\bm, W,\bR) $.  Put another way,
$X(\bS)$ is the collection of fixed points under $\C^\times$ where the
induced action on the fiber of the tautological bundle $V_i$ at that
point has spectrum $S_i$.

\begin{Remark}
In Nakajima's papers, $ \bS $ is denoted $ \rho$ and $ X(\bS) $ is denoted $ \mathfrak{Z}(\rho) $ or $ \mathfrak{M}(\rho)$.  We have changed certain signs in the $\C^\times$-actions to match the conventions from previous sections.
\end{Remark}

To make this more explicit, denote by $W_i(k)$ and $V_i(k)$ the $\C^\times$-weight spaces of weight $k\in \Z$ under the actions $\rho_i$, resp. $\sigma_i$.  Then the multisets $R_i$, $S_i$ encode dimensions:
$$ R_i(k) = \dim W_i(k), \ \ S_i(k) = \dim V_i(k).$$
The identity $\sigma_i(t) \cdot (B,\eta,\varepsilon) = t * (B,\eta, \varepsilon)$ is equivalent to
$$ B_h ( V_{\out(h)} (k) ) \subset V_{\inn(h)}(k+1), \ \ \eta_i(W_i(k)) \subset V_i(k), \ \ \varepsilon(V_i(k)) \subset W_i(k+2). $$

Consider the maps
$$
\begin{aligned}
& {\phi_i(k)}=  \varepsilon_i\oplus \bigoplus_{\out(h) = i} B_h: V_i(k-2) \rightarrow W_i(k) \oplus \bigoplus_{\out(h) = i} V_{\inn(h)}(k-1),  \text{ and} \\
& \tau_i(k) =  \eta_i - \sum_{\out(h) = i, h \in \Omega} B_{\overline{h}} + \sum_{\out(h) = i, h \in \overline{\Omega}} B_{\overline{h}} :  W_i(k) \oplus \bigoplus_{\out(h) = i} V_{\inn(h)}(k-1) \rightarrow V_i(k).
\end{aligned}
 $$
It follows from the stability condition that $\phi_i(k)$ is injective, c.f. \cite[Lemma 2.9.2]{Nak991}, and by the definition of the quiver variety that $\tau_i(k) \circ \phi_i(k)=0$.

Thus we get a complex $C_i(k)^p $, $ p = -1, 0, 1$ as follows
$$
V_i(k-2)   \xrightarrow{\phi_i(k)} W_i(k) \oplus \bigoplus_{\out(h) = i} V_{\inn(h)}(k-1)  \xrightarrow{\tau_i(k)} V_i(k).
$$

As in Section \ref{subsection: Collections of multisets and monomials}, given collections of multisets $ \bR, \bS$, we consider
$$
y_\bR z_\bS^{-1} =  \prod_{i \in I, c \in R_i} y_{i,c} \prod_{i \in I, k \in S_i} z_{i,k}^{-1} = \prod_{i,k} y_{i,k}^{a_{i,k}}.
$$
A simple computation shows the following result.
\begin{Lemma} \label{le:aikandN}
On the locus $ X(\bS) $, we have $ a_{i,k} = \sum_p (-1)^p \dim C_i(k)^p.$
\end{Lemma}

The following result follows from Nakajima \cite{Nak991}, Theorem 5.5.6.

\begin{Proposition} \label{prop:ccomp}
For each $ \bS $, if $ X(\bS) $ is non-empty, then it is a connected component of $ \fM(\bm, W, \bR) $.
\end{Proposition}
Thus we have an injective map
$$
\pi_0(\fM(W, \bR)) \rightarrow \B.
$$

\subsection{Comparison with weight 1 action} \label{se:WeightOne}
The following idea was suggested by Nakajima.

Define a new action of $ \C^\times $ on $ \bM(V,W) $ by
$$
t *' (B, \eta, \varepsilon) = ( (B_h)_{h \in \Omega}, (t^2 B_h)_{h \in \overline{\Omega}}, (\eta_i \rho_i(t))_{i \in I_0}, ( t\eta_i \rho_i(t))_{i \in I_1}, (t^2\rho_i(t)^{-1} \varepsilon_i)_{i \in I_0}, (t \rho_i(t)^{-1} \varepsilon_i)_{i \in I_1})
$$
Again the action descends to $ \fM(\bm,W) $, where it becomes the same action as before.

\begin{Lemma} \label{le:ActionsAgree}
For all $ [B, \eta, \varepsilon] \in \fM(\bm, W) $, we have $ t *' [B, \eta, \varepsilon] = t * [B, \eta, \varepsilon].$
\end{Lemma}

\begin{proof}
Define a map $ \C^\times \rightarrow G_V $ where $ t $ acts by $t $ on $V_i $ for $ i \in I_1 $ and by $ 1 $ for $ i \in I_0 $.  Using the action of $ G_V $ on $ \bM(V, W) $, this defines an action of $ \C^\times $ on $ \bM(V,W) $.

Then for any $(B, \eta, \varepsilon) \in \bM(V,W) $ we have
$$
t *' (B, \eta, \varepsilon) = t \cdot t * (B, \eta, \varepsilon)
$$
and so the result follows.
\end{proof}

Now, define a map $ \rho'_i : \C^\times \rightarrow GL(W_i) $ by using $ -R_i/2 $ as the set of weights if $ i \in I_0 $ and using $ -(R_i + 1)/2$ as the set of weights if $ i \in I_1 $.   Thus
$$
\rho'_i(t^2) = \begin{cases}
\rho_i(t)^{-1}, \text{ if } i \in I_0\\
t^{-1}\rho_i(t)^{-1}, \text{ if } i \in I_1
\end{cases}
$$

Then define an action of $ \C^\times $ on $ \bM(V,W) $  using $ \rho' $ by
$$
t \diam (B, \eta, \varepsilon) = ( (B_h)_{h \in \Omega}, (t B_h)_{h \in \overline{\Omega}}, (\eta_i \rho'_i(t)^{-1})_{i \in I}, (t \rho'_i(t) \varepsilon_i)_{i \in I})
$$
It is easy to see that we have $ t^2 \diam (B, \eta, \varepsilon) = t *' (B, \eta, \varepsilon)$.
In particular:

\begin{Proposition} \label{pr:TwoActions}
We have an equality
$$
\fM(W)^{\C^\times, *} = \fM(W)^{\C^\times, \diam}
$$
where the left hand side is the fixed point set for our original action and the right hand side is the fixed point set for this new action. \qed
\end{Proposition}

\subsection{Crystal structure}

In \cite[\S 8]{Nak010} Nakajima studies the $\diam$ action, and constructs a crystal structure on $ \pi_0(\fM(W)^{\C^\times, \diam})$, which by Proposition \ref{pr:TwoActions} is equal to $ \pi_0(\fM(W, \bR))$.

\begin{Theorem}
The map
\begin{align*}
\phi : \pi_0(\fM(W, \bR)) &\rightarrow \B \\
X(\bS) &\mapsto  y_\bR z_\bS^{-1}
\end{align*}
is an injective map of crystals.
\end{Theorem}

\begin{proof}
 Lemma 8.4 of \cite{Nak010} combined with Lemma \ref{le:aikandN} imply that $ \phi $ is a crystal map. $\phi$ is injective by Proposition \ref{prop:ccomp}.
\end{proof}

\begin{Proposition} \label{pr:pi0andB}
The image of the map $ \phi : \pi_0(\fM(W, \bR)) \rightarrow \B $ is $ \B(\lambda, \bR)$.  Thus we have an isomorphism of crystals $ \pi_0(\fM(W, \bR)) \cong \B(\lambda, \bR)$.
\end{Proposition}

\begin{proof}
If $ \lambda = \varpi_i $ is a fundamental weight (and so $ \bR = \bR^{(i,c)}$) the crystal $ \pi_0(\fM(W, \bR))$ is connected. Thus it suffices to check that $\phi $ takes the highest weight element to $ y_{i,c} $.  But this is obvious.  Now we consider general $\lambda$.

We first show the image of $\phi $ contains  $\B(\lambda, \bR)$. So, fix $p \in \B(\lambda, \bR)$.
By the definition of the product monomial crystal $ p = p_1 \cdots p_N $ for some choices of $ p_k \in \B(\varpi_{i_k}, c_k) $.  Since the result holds when $\lambda$ is a fundamental weight, for each $ k$ we can choose $ x_k \in X(\bS_k)$ such that $p_k =  y_{c_k} z_{\bS_k}^{-1}) $.
Then $ x = x_1 \oplus \cdots \oplus x_N $ is in $X(\bS) $ and so $ X(\bS) $ is non-empty.

It remains to show that the image is contained in $\B(\lambda, \bR)$. So, fix $\bS$ such that  $X(\bS) $ is non-empty.  The variety $X(\bS)$ is invariant under the maximal torus $T_W \subset G_W$.  By \cite[Proposition 4.1.2]{Nak991}, we have that $X(\bS)$ is homotopic to the intersection $X(\bS)\cap \mathcal L(W)$; in particular, the latter variety is non-empty.
Since $X(\bS)\cap \mathcal L(W)$ is projective and invariant under $T_W$, it contains at least 1 fixed point, and so
the fixed point locus $X(\bS)^{T_W}$ is non-empty.  The representation given by such a fixed point can be decomposed according to the weights of the associated $T_W$-action. Thus it is of the form
$x'= x'_1 \oplus \cdots \oplus x'_N$ with each $x'_p$ in some $m(c_p, \bS_p)$.  This is a simple calculation, and can be found in \cite[Lemma 3.2]{Nak010}. It is then immediate from the case when $\lambda$ is a fundamental weight that the image is in the product monomial crystal.
\end{proof}

The following is immediate from  Proposition \ref{pr:pi0andB} and is the first assertion of Theorem \ref{th:subcrystal}.

\begin{Corollary}
$\B(\lambda, \bR)$ is a subcrystal of the monomial crystal $ \B $. \qed
\end{Corollary}

Nakajima's crystal structure on fixed points has a natural geometric realization, which we will describe below. Consider the subvariety $\widetilde{\mathfrak Z}_\diam \subset \fM(m, W)$ consisting of points which have a limit in $\fM(W)^{\C^\times, \diam}$ under the ${\Bbb C}^\times$ action as $t\to \infty$.
There is a bijection between $\pi_0(\fM(W, \bR))$ and the set $\text{Irr} \ \widetilde{\mathfrak Z}_\diam$ of irreducible components of this subvariety sending a connected component to the closure of the points limiting to it. The latter set of components has a natural geometric crystal structure, defined as in \cite[Sec. 4]{Nak010}.  The crystal structure on the graded quiver variety can then be obtained from the geometric one using the bijection described above (much as in \cite{SamTin}).

This perspective makes the relationship to the tensor product crystal clearer.  The variety $\widetilde{\mathfrak Z}_\diam$ is in turn naturally a sub-variety of Nakajima's tensor product variety $\widetilde{\mathfrak Z}$ for $ \otimes_i \B(\varpi_i)^{\otimes \lambda_i}$, consisting of a subset of the irreducible components of that variety.  The crystal structure on $\text{Irr} \ \widetilde{\mathfrak Z}_\diam$ is by definition identical to the one on $\text{Irr} \ \widetilde{\mathfrak Z}$. 
This establishes that $ \B(\lambda, \bR) \subseteq \otimes_i \B(\varpi_i)^{\otimes \lambda_i} $, completing the proof of Theorem \ref{th:subcrystal}. 

Nakajima's variety $\widetilde{\mathfrak Z}$ is defined using a different torus action, and its irreducible components correspond to connected components of a different graded quiver variety. 
We note that, on the level of the graded quiver varieties, the embedding from Theorem \ref{th:subcrystal} can be subtle: the two torus actions are quite different, so a particular irreducible component of $\widetilde{\mathfrak Z}_\diam \subseteq \widetilde{\mathfrak Z}$ can correspond to quite different looking connected components of the two relevant graded quiver varieties. See e.g. \cite{SamTin} for some discussion of this phenomenon. 


\subsection{Some results used earlier} \label{se:earlier}

We can also now give the proof of Proposition \ref{pr:MaxSing}.
\begin{proof}[Proof of Proposition \ref{pr:MaxSing}]
Let $ \bR $ be as in the statement.  Then $ \bR $ determines $ \rho' $ as in Section \ref{se:WeightOne}.  Then $ \rho'_i(t) = I_{W_i} $ for all $ i $.  As observed by Nakajima, \cite[\S 8]{Nak010}, in this case the attracting set is precisely the core of the quiver variety and thus $ \pi_0(\fM(W, \bR))  \cong \B(\lambda) $ as desired.
\end{proof}

The following two results were used in the proofs of Lemmas \ref{lemma: regular kashiwara} and \ref{lemma:positive data}.

\begin{Lemma}
\label{Lemma: other multiplications}
Let $p = y_\bR z_\bS^{-1} \in \B(\lambda,\bR)$, and write $p = \prod_{i,k} y_{i,k}^{a_{i,k}}$.
\begin{enumerate}
\item If $a_{i,k} > 0$, then $z_{i,k-2}^{-1} p \in \B(\lambda,\bR)$.
\item If $a_{i,k} < 0$, then $z_{i,k} p \in \B(\lambda,\bR)$.
\end{enumerate}
\end{Lemma}
\begin{proof}
By Proposition \ref{pr:pi0andB}, we can find $X(\bS) \in  \pi_0(\fM(W, \bR))$ corresponding to $p$.
Fix $[B,\eta,\varepsilon] \in X(\bS)$.  In each case we will produce a point in the appropriate $X(\bS')$, proving that it is non-empty, and hence corresponds to an element of $ \B(\lambda,\bR)$.

{\bf Case  (1)}: Since $a_{i,k} >0$, Lemma \ref{le:aikandN} gives $ \dim \operatorname{Ker} \tau_i(k) / \operatorname{Im} \phi_i(k) > 0 $. Choose an embedding $\C \hookrightarrow \operatorname{Ker} \tau_i(k)$ whose image is not contained in $\operatorname{Im} \phi_i(k)$, and extend $B$ and $\varepsilon$ to $V_i(k-1) \oplus \C$  as the components of this embedding.  Reasoning as in Proposition 4.5 in \cite{Nak3}, we see that this extended datum lies in $\mu^{-1}(0)^s$. The dimension of $V_i(k-1)$ has increased by one, which corresponds to multiplying by  $z_{i,k-2}^{-1}$.

{\bf Case (2)}:  Since $a_{i,k}<0$, Lemma \ref{le:aikandN} implies $\tau_i(k)$ is not surjective.  Choose a codimension one subspace $\operatorname{Im} \tau_i(k) \subset V_i'(k+1) \subset V_i(k+1)$, and define $(B',\eta',\varepsilon')$ as the restrictions of $(B,\eta, \varepsilon)$. This decreases $V_i(k+1)$ by 1, so corresponds to multiplying by $z_{i,k}$.
\end{proof}

\begin{Corollary} \label{Cor: equiv rel on monomial crystal}
Multiplication by $z_{i,k}$ defines a bijection of sets
$$ \begin{pmatrix}  p\in \B(\lambda,\bR) \text{ with} \\ a_{i,k} <0 \text{ and } a_{i,k+2} \geq 0 \end{pmatrix} \stackrel{\con}{\longrightarrow} \begin{pmatrix} p'\in \B(\lambda,\bR) \text{ with}\\  a_{i,k}' \leq 0 \text{ and } a_{i,k+2}' >0 \end{pmatrix}$$

\end{Corollary}

\section{Conjectures of Hikita and Nakajima}

\subsection{Hikita's conjecture} \label{se:Hikita}
Consider the  action of $ \C^\times $ on $ \Grlmbar$ by using $ \rho^\vee : \C^\times \rightarrow T $.  (This is the classical limit of the action of $ \C^\times $ on $ Y^\lambda_\mu(\bR) $ which was used to defined the $B$-algebra.)  Under this $ \C^\times $ action, $ \Grlmbar $ has a unique fixed point, but an interesting $ \C^\times $ fixed subscheme, which we denote $ (\Grlmbar)^{\C^\times}$.  In this section, we will be interested in the coordinate ring $ \C[ (\Grlmbar)^{\C^\times}]$.  Note that this ring is the same thing as the $B$-algebra of the commutative ring $ \C[\Grlmbar] $.

On $\Grlmbar $, there is the commuting $\C^\times$ action by loop rotation and this descends to a $ \C^\times $ action on $ (\Grlmbar)^{\C^\times} $.  This makes $ \C[ (\Grlmbar)^{\C^\times}]$ into a graded algebra.

In our context, Hikita's conjecture is the following statement, which we will prove.
\begin{Theorem}
\label{theorem-Hikita}
There is an isomorphism of graded algebras
$$
\C[(\Grlmbar)^{\C^\times}] \cong H^*(\fM(\mathbf m, W))
$$
\end{Theorem}

\begin{Remark}
When $ \g = \mathfrak{sl}_n $, then $\Grlmbar $ is isomorphic to the intersection of a nilpotent orbit closure and a Slodowy slice and $ \fM(\mathbf m, W)$ is isomorphic to an $ S3 $ variety.  Formulated thus, Hikita proved this statement as Theorem A.1 of \cite{H}.
\end{Remark}

\subsection{Proof of Hikita's conjecture}
We begin by rewriting the $\C^\times $-fixed subscheme of $ \Grlmbar $ as a connected component of the $T$-fixed subscheme of $\overline{\Gr^\lambda} $.

The finite length scheme $\overline{\Gr^\lambda}^{T} $ is disconnected, with connected pieces $(\overline{\Gr^\lambda}^{T})_\nu $ indexed by the $ \C^\times $-fixed points $t^\nu $ in $ \overline{\Gr^\lambda} $, which are in turn labelled by those coweights $ \nu $ lying in the convex hull of the Weyl orbit through $\lambda $ and such that $ \lambda - \nu $ lies in the coroot lattice.  Let $ \overline{\Gr^\lambda}^T_\nu $ be the connected component of the $T$-fixed subscheme of $ \overline{\Gr^\lambda}$ which is supported at the point $ t^\nu $.

\begin{Proposition} \label{pr:CtimesT}
For dominant $ \mu$, there are natural isomorphisms
\begin{equation}
(\Grlmbar)^{\C^\times} \cong (\Grlmbar)^T \cong \overline{\Gr^\lambda}^T_\mu.\label{eq:C=T}
\end{equation}

\end{Proposition}

\begin{proof}
We first check that $ \Gr_\mu^{\C^\times}\cong \Gr_\mu^T $. This is clear because $ \C[\Gr_\mu] $ is a polynomial ring in infinitely many variables, whose weights are all roots of $ \g $.  The first isomorphism of \eqref{eq:C=T} then follows from imposing the ideal vanishing in $\Grlmbar$ on both sides.

Now, let $U= t^\mu G_1[t^{-1}] \cdot G[[t]]/G[[t]]$.  This is an open neighborhood of $t^\mu$ on which $t^\mu G_1[t^{-1}] t^{-\mu}$ acts freely.  Note that if  $H=G[[t]]\cap t^\mu G_1[t^{-1}] t^{-\mu}$, then we have $U\cap \Gr^{\mu}=H\cdot t^\mu$, with $H$ acting freely on this set.  The subgroup $H$ is normalized by $T$ and thus has a $T$-action by conjugation.  The action map $H\times \Grlmbar\to \overline{Gr^\la}\cap U$ is a $T$-equivariant isomorphism.  Thus   $\overline{\Gr^\lambda}^T_\mu\cong (\Grlmbar)^T \times H^T$.  However $H^T$ is just the identity with reduced scheme structure, since $H$ is unipotent, and thus $\C[H]$ is a polynomial ring whose generators all have non-zero weight.
\end{proof}

The key to our proof of Hikita's conjecture is the work of Zhu \cite{Z} on the $ \C^\times $--fixed subscheme of the affine Grassmannian.

We have the standard line bundle $ \O(1) $ on $ \Gr $ which we can restrict to $ \overline{\Gr^\lambda} $.  The following result is due to Zhu \cite{Z}.

\begin{Theorem}
The natural restriction map
$$
\Gamma( \overline{\Gr^\lambda}, \O(1)) \rightarrow \Gamma(\overline{\Gr^\lambda}^T, \O(1))$$
is an isomorphism of $ T[[t]] $-representations.
\end{Theorem}

Now we will relate $ \Gamma( \overline{\Gr^\lambda}, \O(1)) $ to the quiver variety.  First, there is an action of the algebra $ \g[[t]] $ on $ H^*(\fM(W)) $.  For this action, the weight decomposition is $ H^*(\fM(W)) \cong \oplus_\mu H^*(\fM(\bm, W)) $ (where as usual $ \mu $ and $ \bm $ are related by $ \lambda -\mu = \sum_i m_i \alpha_i $).  By a theorem of Kodera-Naoi \cite{KN} we have that $ H^*(\fM(W)) $ is isomorphic to a dual Weyl module.  By Fourier-Littelmann \cite{FL} the dual module is isomorphic to $\Gamma( \overline{\Gr^\lambda}, \O(1))$, so combining these results we have:

\begin{Theorem}
 There is an isomorphism of graded $\g[[t]]$-modules $ H^*(\fM(W)) \cong \Gamma( \overline{\Gr^\lambda}, \O(1))$.
\end{Theorem}

Combining together the previous two theorems, we deduce that there is an isomorphism of graded $ T[[t]] $-modules
\begin{equation} \label{eq:HandGamma}
 H^*(\fM(W)) \cong \Gamma(\overline{\Gr^\lambda}^T, \O(1)).
\end{equation}
The action of $ T $ on the fibre of $\O(1)$ over the point $ t^\mu $ is by the weight $ \mu $ (here we identify coweights with weights).  Thus taking the $ \mu $ weight spaces in (\ref{eq:HandGamma}) yields an isomorphism of $ T[[t]]$-modules
$$
H^*(\fM(\bm, W)) \cong \Gamma(\overline{\Gr^\lambda}^T_\mu, \O(1))
$$

Let $K^\la_\mu\subset U(\mathfrak t[[t]])$ be the annihilator of $H^*(\fM(\bm,W)) $ which is the same as the annihilator of $\Gamma(\overline{\Gr^\lambda}^T_\mu, \O(1))$ .

The following elementary Lemma will be crucial for us.

\begin{Lemma} \label{le:elementary}
 Let $ R, A $ be commutative rings and suppose we have a surjective map $ R \rightarrow A $.  Then $ A \cong R/K $, where $ K $ is the annihilator of $ A $ regarded as an $R$-module.
\end{Lemma}

\begin{Proposition} \label{pr:Kirwan}
We have an isomorphism of algebras $H^*(\fM(\bm, W)) \cong  U(\mathfrak t[[t]]) / K^\la_\mu$.
\end{Proposition}
\begin{proof}
  By definition, the action of $ U(\mathfrak t[[t]]) $ on $ H^*(\fM(\bm, W)) $ comes from the surjective Kirwan map $$  U(\mathfrak t[[t]]) \rightarrow H^*_{G_V}(pt) \rightarrow H^*(\fM(\bm, W)). $$

  Thus, by Lemma \ref{le:elementary}, the result follows.
\end{proof}

We will now need to relate $ \Gamma(\overline{\Gr^\lambda}^T_\mu, \O(1)) $ with $ \C[\overline{\Gr^\lambda}^T_\mu] $.  It will be easiest to do this by working in the affine Grassmannian of $ T $.

Recall that the affine Grassmannian $ \Gr_T $ is a non-reduced disconnected scheme with connected components $ \Gr_{T,\nu} $, where $ \nu $ ranges over the coweight lattice.  Each component $ \Gr_{T, \nu} $ has a simple transitive action of the group scheme $ T_1[t^{-1}]$.  Moreover, $T_1[t^{-1}] $ acts on the line bundle $ \O(1) $ and thus we can trivialize the line bundle to provide an isomorphism (of vector spaces) $ \C[T_1[t^{-1}]] \cong \Gamma(\Gr_{T,\nu})$.

We would like to understand the action of $ \mathfrak t[[t]] $ after this isomorphism.  To understand this, we recall that there is a $\C^\times $-extension of $ T((t)) $, denoted $ \widehat{T((t))} $, and a commutator map $$ \{, \} : T((t))  \times T((t))  \rightarrow \C^\times$$
which is defined using the Contou-Carr\`ere symbol (see for example \cite{Z}).  This extension gives rise to the line bundle $ \O(1) $ on $ \Gr_T $.

The adjoint action of $ \widehat{T((t))} $ on its Lie algebra is given by a map $ \widehat{T((t))} \rightarrow \mathfrak{t}((t))^* $ or equivalently by a map $  \mathfrak{t}((t)) \rightarrow \C[\widehat{T((t))}] $.  Thus we obtain a map $ \alpha : \mathfrak{t}[[t]] \rightarrow \C[T_1[t^{-1}]] $.

\begin{Lemma} \label{th:actbymult}
 The isomorphism $ \C[T_1[t^{-1}]] \cong \Gamma(\Gr_{T,\nu}, \O(1)) $ is $\mathfrak{t}[[t]]$-equivariant, where $ X \in \mathfrak{t}[[t]] $ acts on $ \C[T_1[t^{-1}]] $ by multiplication by $ \alpha(X) + \nu(X(0)) $.
\end{Lemma}

\begin{proof}
 Let $ g \in T[[t]] $.  Then the action of $ g $ on an element $ f \in \C[T_1[t^{-1}]] $ is given by $ g \cdot f = \{ g, - \} \{g, t^\nu \} f $ where $ \{g, - \} $ is regarded as element of $ \C[T_1[t^{-1}]] $.  Differentiating, we obtain the desired result.
\end{proof}

We are now is a position to complete the proof of Hikita's conjecture.
\begin{proof}[Proof of Theorem \ref{theorem-Hikita}]
From \cite{Z}, we know that there is a natural identification $ \Gr^T = \Gr_T $.  Thus we can embed $ \overline{\Gr^\lambda}^T_\mu $ as a subscheme of $ \Gr_{T,\mu} $.  Thus the trivialization of $ \O(1) $ on $ \Gr_{T, \mu} $ gives us a trivialization of $ \O(1) $ on $ \overline{\Gr^\lambda}^T_\mu $ and thus an isomorphism
\begin{equation}
\C[\overline{\Gr^\lambda}^T_\mu] \cong \Gamma(\overline{\Gr^\lambda}^T_\mu, \O(1)).\label{eq:fixed}
\end{equation}
This isomorphism is compatible with the isomorphism from Lemma \ref{th:actbymult} and thus we deduce that the map of \eqref{eq:fixed} is a $\mathfrak t[[t]]$-module isomorphism where $ \mathfrak t[[t]] $ acts on $\C[\overline{\Gr^\lambda}^T_\mu] $ by multiplication.   Thus we are in the situation of Lemma \ref{le:elementary} and thus we have an induced isomorphism $ \C[\overline{\Gr^\lambda}^T_\mu] \cong  U(\mathfrak t[[t]]) / K^\la_\mu $.

Applying Propositions \ref{pr:CtimesT} and \ref{pr:Kirwan} completes the proof.
\end{proof}

\subsection{Nakajima's conjecture}
\label{sec:conjecture-nakajima}
Nakajima has extended Hikita's conjecture to the non-commutative situation.  We can consider the $B$-algebra $B_{h}(\bR)$ of the
Rees algebra $Y_\mu^\lambda(\bR)_h \subset Y_\mu^\lambda(\bR)[h,h^{-1}]$; here we
consider the parameters $\bR$ as formal variables, and thus we obtain
a finite algebra over the polynomial ring $\C[e_k(R_i),h]$.  This
polynomial ring is isomorphic to the cohomology of the
classifying space of the group $G_{W}\times \C^\times$.
\begin{Conjecture}[Nakajima]
  There is an isomorphism of $\C[e_k(R_i),h]\cong
  H^*_{G_{W}\times \C^\times}(*)$-algebras \[B_{h}(\bR) \cong H^*_{G_{W}\times
    \C^\times}(\fM(\mathbf{m},W)).\]

In particular, specializing $ \bR$ to numerical values and setting $h=1/2$, we have $ B(Y^\lambda_\mu(\bR)) \cong H^*(\fM(\bm, W, \bR)) $.
\end{Conjecture}

At the moment, we do not know how to prove this conjecture.  However, let us explain how this conjecture is related to Conjecture \ref{co:main}.  First, we extend Nakajima's conjecture in the following way.  Recall that we have a surjection $ \Cartan \rightarrow B(Y^\lambda_\mu(\bR)) $ by Proposition \ref{pr:2Balgebras} .  Now, we define a map $ \Cartan \rightarrow H^*(\fM(\bm, W, \bR)) $ as the composition
$$
\Cartan \rightarrow H^*_{\C^\times}\big(\fM(\mathbf{m},W)\big) \rightarrow  H^*\big(\fM(\bm, W, \bR)\big)
$$
where the first map is given by sending $ A_i^{(s)} $ to the $s$th Chern class of the tautological vector bundle $V_i$, where the action of $ \C^\times $ on $ \fM(\mathbf{m}, W) $ is given by $ \bR $ (as in Section \ref{se:gradedQV}), and where the second arrow is given by specializing the equivariant parameter to $h = 1/2$.

We extend Nakajima's conjecture to the following statement.
\begin{Conjecture} \label{co:Nakajima2}
There is an isomorphism of $ \Cartan$-algebras \[ B\big(Y^\lambda_\mu(\bR)\big) \cong H^*\big(\fM(\bm, W, \bR)\big) .\]
\end{Conjecture}
In other words, not only are these two algebras isomorphic, they are the same quotient of $ \Cartan $.  This conjecture is equivalent to the equality of $ \Spec  B\big(Y^\lambda_\mu(\bR)\big) $ and $ \Spec H^*\big(\fM(\bm, W, \bR)\big) $ as subschemes of $ \Spec \Cartan $.  If we pass to the underlying subsets, we return to our main conjecture.

\begin{Proposition}
 Conjecture \ref{co:main} is equivalent to the equality of $ \MaxSpec B\big(Y^\lambda_\mu(\bR)\big)$ and $ \MaxSpec H^*\big(\fM(\bm, W, \bR)\big) $ as subsets of $ \Spec \Cartan $.
\end{Proposition}

\begin{proof}
 Recall that $ \MaxSpec B\big(Y^\lambda_\mu(\bR)\big)$ is the set of highest weights $H^\lambda_\mu(\bR) $ by Propositions \ref{Prop: B algebra and Vermas} and \ref{pr:2Balgebras}.  Moreover, given $ J \in H^\lambda_\mu(\bR) $, we can regard $ J $ as defining a $ \C$-algebra map $ \Cartan \rightarrow \C $ where $ A_i^{(s)} $ is sent to $ e_s(S_i/2) $ where $ y(J) = y_\bR z_\bS^{-1} $ (and $ e_s $ denotes the $s$th elementary symmetric function).

 Assuming Conjecture \ref{co:main}, this element $ y_\bR y_\bS^{-1} $ lies in the monomial crystal $ \B(\lambda, \bR) $ and thus gives us a connected component $X(\bS)$ of $ \fM(\bm, W, \bR) $ by Proposition \ref{pr:pi0andB}.  Note that for any algebraic variety $ Y $, the $\C$-algebra morphisms $ H^*(Y) \rightarrow \C $ are given by $ H^*(Y) \rightarrow H^*(Y_i) \rightarrow \C $, where $ Y_i $ ranges over the connected components of $ Y $.  Thus, Lemma \ref{le:composite} completes the proof.
\end{proof}

\begin{Lemma} \label{le:composite}
 For any $ \bS$ such that $  y_\bR y_\bS^{-1} \in \B(\lambda, \bR) $, the composite
\[ \Cartan \rightarrow H^*(\fM(\bm, W, \bR)) \rightarrow H^*(X(\bS)) \rightarrow \C, \]
takes $ A_i^{(s)} $ to the $s$th elementary symmetric function of the multiset $ S_i/2 $.

\end{Lemma}
\begin{proof}
For any $ \bS $, let us rewrite this composition as
\[ \Cartan \rightarrow H^*_{\C^\times}(\fM(\bm, W)) \rightarrow H^*_{\C^\times}(X(\bS)) \rightarrow \C[h,h^{-1}] \rightarrow \C, \]
where the third arrow comes from the inclusion of a point into the connected space $ X(\bS) $ and the fourth arrow is given by setting $h = 1/2 $.  Here we use the trivial action of $ \C^\times$ on $ X(\bS) $.

  On the locus $X(\bS)$, the restriction of the tautological bundle
  $V_i$ is a sum of bundles $V_i(k)$, with dimension $ \dim V_i(k) = S_i(k) $ and thus
  \[
  c(V_i) = \prod_k c(V_i(k))
  \]

  By definition, each $V_i(k)$ can be written as $ V_i(k) = U_i^{(k)} \otimes \C_k $, where $U_i^{(k)}$
  is a vector bundle with with trivial $\C^\times$ action and $ \C_k $ denotes the 1-dimensional vector space with action of $ \C^\times $ by weight $ k$.

 Thus, computing in the equivariant cohomology
  $H^*_{\C^{\times}}(X(\bS))$, we have $ c(V_i(k)) = (1 + kh)^{S_i(k)} + \dots $, where $ \dots $ represents an expression containing the Chern classes of the bundle $ U_i^{(k)} $.  Since the Chern classes of the bundle $ U_i^{(k)} $ are all sent to $ 0 $ in $\C[h,h^{-1}] $, we conclude that the image of $ c(V_i) $ in $ \C[h,h^{-1}] $ is $\prod_k (1 + kh)^{S_i(k)}$.  Setting $h = 1/2 $, we obtain the desired result.
\end{proof}

\begin{Remark}
Recall that we proved our main conjecture for $ \g = \mathfrak{sl}_n$.  In fact, in Proposition \ref{B algebra}, we have an explicit description of the $B$-algebra in this case.  Using this description together with Theorem \ref{theorem-Hikita}, a proof of Conjecture \ref{co:Nakajima2} for $\g = \mathfrak{sl}_n$ appears in \cite[Theorem 8.3.7]{AWthesis}.  It is also possible to prove this result by showing explicitly that the B-algebra agrees in this case with the (straightforward) equivariant generalization of the Brundan-Ostrik presentation \cite{BO} of the cohomology of S3 varieties.
\end{Remark}

\bibliography{./monbib}
\bibliographystyle{amsalpha}

\end{document}